\documentclass[12pt,american]{amsart}
\usepackage{amsmath,amsfonts,amssymb,amscd,amsthm,amsbsy,bbm,epsf,calc,graphicx,esint}
\usepackage{color}
\usepackage{datetime}
\usepackage{mathtools}

\usepackage[pdftex,bookmarks,colorlinks,breaklinks]{hyperref}

\numberwithin{equation}{section}

\newcounter{hours}\newcounter{minutes}

\theoremstyle{plain}
\newtheorem{thm}{Theorem}[section]

\newtheorem{lem}[thm]{Lemma}
\newtheorem{cor}[thm]{Corollary}
\newtheorem{prop}[thm]{Proposition}

\newtheorem{DEF}[thm]{Definition}

\theoremstyle{definition}                  
\newtheorem{rem}[thm]{Remark}

\newtheorem{Assumption}{Assumption}

\def\tr{\textnormal{tr}}
\def\dive{\textnormal{div}}

\def\Id{\mathbb{I}}
\def\astar{{a^*}}

\def\fin{f_{\textnormal{in}}}

\def\Ap{\mathcal{A}_p}
\def\A1{\mathcal{A}_1}
\def\Ainfty{\mathcal{A}_\infty}
\def\Lamp{\Lambda(\varepsilon)}
\def\Lampf{\Lambda_f(\varepsilon)}

\title{On $A_p$ weights and the Landau equation}

\begin{document}
	
\author{Maria Gualdani and Nestor Guillen}	
	
\begin{abstract}
  In this manuscript we investigate the regularization of solutions for the spatially homogeneous Landau equation. For moderately soft potentials, it is shown that weak solutions become smooth instantaneously and stay so over all times, and the estimates depend only on the initial mass, energy, and entropy. For very soft potentials we obtain a conditional regularity result, hinging on what may be described as a nonlinear Morrey space bound, assumed to hold uniformly over time. This bound always holds in the case of moderately soft potentials, and nearly holds for general potentials, including Coulomb. This latter phenomenon captures the intuition that for moderately soft potentials, the dissipative term in the equation is of the same order as the quadratic term driving the growth (and potentially, singularities). In particular, for the Coulomb case, the conditional regularity result shows a rate of regularization much stronger than what is usually expected for regular parabolic equations. The main feature of our proofs is the analysis of the linearized Landau operator around an arbitrary and possibly irregular distribution. This linear operator is shown to be a degenerate elliptic Schr\"odinger operator whose coefficients are controlled by $A_p$-weights. 
\end{abstract}

\maketitle

\baselineskip=14pt
\pagestyle{headings}		

\markboth{$A_p$ weights and the homogeneous Landau equation}{M. Gualdani, N. Guillen}

\section{Introduction}\label{section:introduction}

The Landau equation with Coulomb potential models the evolution of a particule distribution
\begin{align*}
 f(x,v,t): \Omega \times \mathbb{R}^d\times\mathbb{R}_+\to\mathbb{R},\;\; \Omega\subset\mathbb{R}^d,
\end{align*} 
and arises as a limit from the Boltzmann equation in the regime where only grazing collisions are predominant. Generally, the evolution of the particle density is described by the equation
\begin{align}\label{eqn:Landau}
  \partial_tf +v\cdot\nabla_x f =Q(f,f),
\end{align}
with initial and boundary conditions on $f$. The quadratic term $Q(f,f)$ describes how collisions affect the evolution of the particle distribution. The collision operator $Q(f,f)$ for Landau-Coulomb is given by
\begin{align}\label{eqn:Landau Collisional Term Classical Expression}
  Q(f,f) := \dive_v \left (\int_{\mathbb{R}^d} \Phi(v-w)\left (\Pi(v-w)(f(w)\nabla_vf(v)-f(v)\nabla_wf(w) \right )\;dw \right ),
  \end{align}
  with $\Pi(z)$ (for $z\neq 0$) the projection onto the orthogonal complement of $z$,
  \begin{align*}
  \Pi(z) := \Id- \frac{ z\otimes z}{|z|^{2}},
  \end{align*}
  and 
  \begin{align*}
    \Phi(z) = C_d|z|^{2-d}, \quad d\geq 3.
  \end{align*}
  A great deal of attention from the mathematical community has been devoted also to the the analysis of \eqref{eqn:Landau}-\eqref{eqn:Landau Collisional Term Classical Expression} with $\Phi$ of the more general form
  \begin{align}
    \Phi(z) = C_\gamma |z|^{2+\gamma},\;\;\;\textrm{for any $\gamma\in[-d,0]$}.	\label{gene_kernel}
  \end{align}	
  Borrowing the terminology from the Boltzmann equation, in the literature $ \Phi(z)$ is called a moderately soft potential when $\gamma \in [-2,0]$  and  a very soft potential if $\gamma < -2$.
     
  The resulting family of equations \eqref{eqn:Landau}-\eqref{gene_kernel} presents very interesting analytical questions, and their understanding has shone light on the physically relevant Landau-Coulomb equation ($\gamma=-d$).
  
  The following fact about the structure of $Q(f,f)$ has been key in a great deal of the literature on the Landau equation: for a smooth $f$, the interaction term $Q(f,f)$ can be expressed in terms of a second order elliptic operator with respect to $v$ and (this is most crucial) its coefficients are given by integral non-local operators of $f$. For $\Phi(z) $ as in \eqref{gene_kernel} these integral operators involve Riesz potentials of $f$. Indeed, for such $\Phi(z)$ the operator $Q(f,f)$ takes the form
  \begin{align*}
   Q(f,f)= \dive \left (A_{f,\gamma}\nabla f-f\nabla a_{f,\gamma} \right), \textnormal{ with } a_{f,\gamma} := \tr(A_{f,\gamma}),
  \end{align*}
  where
\begin{align}\label{Af_gamma}
  & A_{f,\gamma} := C(d,\gamma) \int_{\mathbb{R}^d} |v-w|^{2+\gamma}\Pi(v-w)f(w) \;dw.
\end{align}
The constant $C(d,\gamma)$ is such that
\begin{align*}
  a_{f,\gamma} := \tr(A_{f,\gamma}) = (-\Delta)^{-\frac{d+2+\gamma}{2}}f.
\end{align*}
In general, for any $\gamma \geq -d$, we will write
\begin{align}\label{hf_gamma}
    h_{f,\gamma}:=  -\Delta a_{f,\gamma} =  (-\Delta)^{-\frac{(d+\gamma)}{2}}(f) = \left\{  \begin{array}{cl} 
                        f, & \quad \textnormal{ if } \gamma=-d,\\ 
                         c(d,\gamma) f * |v|^{\gamma}, &  \quad \textnormal{ if } \gamma>-d.              \end{array}  \right.
\end{align}
The collision operator $Q(f,f) $ then has the ``non-divergence'' representation 
\begin{align*}
  Q(f,f) = \tr(A_{f,\gamma}D^2f)+fh_{f,\gamma}.
\end{align*}

Equation \eqref{eqn:Landau}-\eqref{gene_kernel}  therefore may be described as a nonlinear transport equation involving reaction, diffusion, and transport effects. If the initial distribution is spatially inhomogeneous, i.e. if $\fin(x,v) = \fin(v)$, one could consider a simpler equation in the velocity space only, where transport term drops. This yields the homogeneous Landau equation
\begin{align}\label{eqn:Landau homogeneous}
  \partial_tf = Q(f,f),
\end{align}
which is the equation studied hereafter.

{Let us illustrate the challenges in estimating solutions \eqref{eqn:Landau homogeneous} in the most physically important case, namely the Coulomb case $\gamma = -d$. The non-divergence expression for the equation is most helpful; for $\gamma=-d$ we have $h_{f,\gamma} = f$, and the equation becomes
\begin{align*}
  \partial_t f = \tr(A_{f,\gamma}D^2f)+f^2.
\end{align*}
The term $f^2$ could very well drive the equation to finite time blow up, as with the semilinear heat equation \cite{GigaKohn85}. However, there are not known examples of solutions blowing up in finite time, leaving open the possibility that global $L^\infty$ estimates may hold. Any attempt to obtain $L^\infty$ estimates for the equation must take into account the competition between the term $f^2$ (and $h_{f,\gamma} f$ in general) and the diffusive term $\tr(A_{f,\gamma}D^2f)$. This is discussed further near the end of Section \ref{section: contributions}. The other challenge is the behavior of $A_{f,\gamma}$ for large $v$; such term degenerates as $|v|\to +\infty$ and this weight needs to be taken into account in the coercivity estimates.

The current understanding of \eqref{eqn:Landau} is far less advanced than that of the equation \eqref{eqn:Landau homogeneous}, although a lot of recent activities on \eqref{eqn:Landau homogeneous} are changing this (this is briefly discussed in the literature overview below)}. As can be inferred, the analysis of \eqref{eqn:Landau homogeneous} can serve as a stepping stone towards understanding the full, spatially inhomogeneous system \eqref{eqn:Landau}. The question of regularity versus breakdown for \eqref{eqn:Landau homogeneous} (that is, the validity of $L^\infty$bounds discussed above) remains unresolved not only for the Coulomb case, but also for all very soft potentials $\gamma \in [-d,-2)$. 
Even more, some authors have proposed to consider $\gamma$ below $-d$. Observe that from this point on the operator $h_{f,\gamma}$ becomes a pseudo-differential operator as opposed to a convolution operator. 

In this manuscript we investigate $L^\infty$ estimates for the homogeneous Landau equation in various situations. First, for \eqref{eqn:Landau homogeneous} in the case $\gamma\in(-2,0)$ (Theorem \ref{thm:main soft potentials}) we show that weak solutions become instantaneously bounded locally in space, and remain so for all later times. The weak solutions considered are only assumed to have finite initial mass, second moment, and entropy.  For $\gamma \in [-d,-2]$ we obtain a conditional $L^\infty$ estimate for classical solutions (Theorems \ref{thm:very soft potentials estimate} and \ref{thm:Coulomb case good estimate}). The conditions imposed in this case are rather weak: in fact, as argued in Section \ref{sec:discussion of extra assumptions}, although these conditions may fail to hold in general, they appear to ``almost'' hold. Moreover, the conditional $L^\infty$ estimate we obtain for $\gamma=-d$ comes with a much faster rate of regularization than is expected for uniformly parabolic equations.

We review some of the recent relevant literature and provide some context for our results in Sections \ref{section: literature review} and \ref{section: contributions} below.

\subsection{Recent results in the regularity theory for kinetic equations}\label{section: literature review} The analysis of \eqref{eqn:Landau} and \eqref{eqn:Landau homogeneous} comprises a vast literature. Here we mention only a few of the results relevant to the question of regularity, with an emphasis on \eqref{eqn:Landau homogeneous}.

In general, the question of existence of weak solutions for \eqref{eqn:Landau homogeneous} (arbitrary $\gamma$) was addressed by Villani \cite{V98}, where the notion of $H$-solution --which exploits the $H$-theorem to make sense of the equation-- was introduced.  For the case of Maxwell potentials ($\gamma= 0$) and more general over-Maxwellian potentials, Desvillettes and Villani \cite{DesVil2000a,DesVil2000b} studied thoroughly the issues of existence, uniqueness, and convergence to equilibrium. First,  in \cite{DesVil2000a}[Theorem 5] they show that given any an initial datum $\fin$ which also lies in $L^2$, then there exists at least one weak solution to \eqref{eqn:Landau homogeneous} which becomes smooth for all positive times. The estimates depend only on the known conserved quantities. It is also shown that weak solutions are unique if $\fin$ lies in a weighted $L^2$ space  \cite{DesVil2000a}[Theorem 6]. In subsequent work \cite{DesVil2000b}, entropy dissipation estimates are used to prove  convergence of solutions to equilibrium, this being exponentially fast for over Maxwellian potentials, and with an algebraic rate for hard potentials. 

Analyzing the case of soft potentials ($\gamma \in (-2,0]$) has proved more difficult. Fournier and Guerin \cite{FournierGuerin2009} proved uniqueness of bounded weak solutions in the case of soft potentials. Later work of Fournier \cite{Fournier2010} extends this to the Coulomb potential.  Alexandre, Liao, and Lin  \cite{AleLinLia2013} obtained estimates on the growth of the $L^p$ norms of a solution, this being later extended to the case $\gamma=-2$ by Wu \cite{Wu13}. These last two results are based on the propagation of $L^p$ bounds, and thus they require initial data in $L^p$. The resulting upper bounds for the $L^p$ norm grow exponentially with time \cite{AleLinLia2013,Wu13}.

In \cite{Silvestre2015}, Silvestre obtains $L^\infty$ estimates for classical solutions to the homogeneous Landau equation for $\gamma> -2$, yielding by non divergence methods estimates for classical solutions that are similar to those in this paper (see discussion after Theorem \ref{thm:main soft potentials}). In the work of Golse, Imbert, Mouhot, and Vasseur \cite{GIMV16}  the inhomogeneous Landau equation is considered (and more generally, kinetic Fokker-Planck equations with rough coefficients) and the authors obtain H\"older regularity for bounded weak solutions for any $\gamma\ge-d$. We also point out earlier work on parabolic kinetic equations by Pascucci and Polidoro \cite{PascucciPolidoro2004}, which is a predecessor, in the Moser iteration flavor, of \cite{GIMV16}, which is based on De Giorgi's method. Subsequently, Cameron, Silvestre, and Snelson \cite{CamSilSne2017} applied the H\"older estimates from \cite{GIMV16} to obtain an priori estimates for the $L^\infty$ norm of bounded solutions in the case of moderately soft potentials, $\gamma>-2$ (see Section \ref{section: contributions} where we further discuss this result). For the Boltzmann equation, we mention recent regularity results by Silvestre \cite{Silvestre2014} by non-divergence methods (for the homogeneous case) and by Imbert and Silvestre \cite{Imb_Silv2017} using a combination of variational and non-divergence arguments (for the general case). For the (inhomogeneous) Landau near equilibrium, recently Kim, Guo, and Hwang \cite{KimGuoHwang2016} used De Giorgi style iterations following \cite{GIMV16} and showed the existence of global smooth solutions for initial data close enough to equiibrium, the closeness measured just in the $L^\infty$ norm.

For very soft potentials much less is known: this includes (i) local existence, and (ii) global existence of a classical solution for initial data close to equilibrium by Guo \cite{Guo02} (covering the spatially inhomogeneous case). Chen, Desvillettes and He \cite{CDH09},  Carrapatoso, Tristani and Wu \cite{CTW15} for $\gamma \ge -2$ showed existence of the so called H-solutions by Villani in \cite{V98}, and of weak solutions by Desvillettes \cite{Desvillettes14} for $\gamma=-d$. In a related work, Carrapatoso, Desvillettes and He \cite{CDH15} showed that any weak solution to the Landau equation with $\gamma =-3$ converges to the equilibrium function for any initial data with finite mass, energy and higher $L^1$-moments.

The question of $L^\infty$ estimates for very soft potentials $\gamma<-2$ (and notably, the Coulomb potential $\gamma=-d$) remains a difficult one. There has been partial but encouraging progress for an equation similar to \eqref{eqn:Landau homogeneous}, introduced by Krieger and Strain in \cite{KriStr2012}. In successive works, Krieger and Strain \cite{KriStr2012} and later Gressman, Krieger and Strain \cite{GreKriStr2012} consider for $d=3$ the equation 
\begin{align}\label{eqn:Krieger-Strain}
  \partial_t f = a_f \Delta f + \alpha f^2,\;\;a_f = (-\Delta)^{-1}f,
\end{align}	
and show that for  spherically symmetric and radially decreasing initial data there is a global in time solution (enjoying the same symmetries for positive time) which remains bounded, provided that $\alpha \in (0,\tfrac{74}{75})$. The restriction on the parameter $\alpha$ corresponds to the range of $\alpha$ for which solutions to \eqref{eqn:Krieger-Strain} have finite $L^{3/2}$ norm. Observe that for $\alpha=1$ the above equation may be rewritten as
\begin{align*}
  \partial_t f = \dive\left ( a_f \nabla f -f\nabla a_f\right ),
\end{align*}	
which resembles \eqref{eqn:Landau homogeneous} in the Coulomb case $\gamma=-d$. In \cite{GuGu15}, the authors of this current manuscript proved that for all $\alpha \in [0,1]$ equation \eqref{eqn:Krieger-Strain} has spherically symmetric, radially decreasing solutions that remain bounded for all positive times, assuming the initial data have high enough integrability. The proof relies on a barrier argument applied to a mass function associated to the solution. Unfortunately, this barrier argument falls short of yielding a result for the Landau-Coulomb equation. Nevertheless, the approach in \cite{GuGu15} does show that a (spherically symmetric, radially decreasing) solution to the Landau equation with $\gamma=-d$ does not blow up as long as it remains bounded in some $L^p$ for any $p>\frac{d}{2}$. In light of what is known about blow up for semilinear parabolic equations, this last result is to be expected. We note moreover, that the aforementioned work of Silvestre \cite{Silvestre2015} obtains regularization under the assumption that similar $L^p$ norms remain bounded in time.  We remark that the tools in \cite{GuGu15} and \cite{Silvestre2015} are somewhat different: while \cite{Silvestre2015} relies on methods from fully nonlinear equations, \cite{GuGu15} uses a mass function barrier argument to propagate an initial $L^p$ bound with $p>d/2$ over time, ultimately leading to an {\em non conditional } $L^\infty$ bound for spherical functions that so far have not been obtained with the methods from \cite{Silvestre2015}.  

Regarding higher regularity of the solutions, in the homogeneous case, once the $L^\infty$ norm is controlled, one can obtain higher regularity of the solution (see for instance \cite[Section 4]{GuGu15} for radially symmetric case). For the inhomogeneous case traditional parabolic tools are insufficient. In \cite{HendersonSnelson2017} Henderson and Snelson provide $L^\infty$ estimates for all the derivatives of bounded weak solutions to the inhomogeneous Landau equation, and subsequently in \cite{HendersonSnelsonTarfulea2017} Henderson, Snelson, and Tarfulea improve on the previous result by showing first that solutions spread mass instantaneously to all parts of the domain, which leads to improved estimates on the derivatives.

\subsection{Contributions of the present work}\label{section: contributions} There are two main ideas we would like to highlight in this manuscript: (i) how the theory of $A_p$ weights and weighted Sobolev and Poincare's estimates plays a role in the analysis of the Landau equation; and (ii) the observation that there are inequalities that guarantee $L^\infty$ estimates for solutions and these inequalities \textbf{hold for all solutions} in the case of moderately soft potentials and \textbf{nearly hold} in the case of very soft potentials (see discussion in Section \ref{sec:discussion of extra assumptions}). 

This last claim refers to the fact that the ``reaction'' term in \eqref{eqn:Landau homogeneous} although possibly very singular for $\gamma<-2$ is of the same order as the diffusion term in the equation (see Proposition \ref{prop:assumption 1 almost holds}). This means that \eqref{eqn:Landau homogeneous} ought to be thought of as a critical equation in some sense. This suggestion is further reinforced by the recent results in \cite{GuGu15} regarding estimates for solutions to \eqref{eqn:Krieger-Strain}. 

The main {\em non-conditional} result in this work states that weak solutions (Definition \ref{def:solution}) to the Landau equation with $\gamma\in(-2,0]$ become instantaneously bounded.
\begin{thm}\label{thm:main soft potentials}
  Let $f: \mathbb{R}^d\times \mathbb{R}_+\to\mathbb{R}$ be a weak solution of the Landau equation \eqref{eqn:Landau homogeneous} where $\gamma \in(-2,0]$ and the initial data has finite mass, energy and entropy. Then, there is a constant $C_0$ determined by $d$, $\gamma$, the mass, energy, and entropy of $\fin$, such that
  \begin{align*}
    f(v,t) \leq C_0\left ( 1 + \frac{1}{t} \right )^{\frac{d}{2}} \left (1+|v| \right)^{-\gamma\tfrac{d}{2}}.
  \end{align*}  
\end{thm}
The result of Theorem \ref{thm:main soft potentials} is certainly expected, save for the appearance of the factor  $(1+|v|)^{-\gamma\tfrac{d}{2}}$, which only reflects a shortcoming of our arguments. Indeed, the a priori estimate obtained in \cite{Silvestre2015} and \cite{CamSilSne2017} does not contain such a factor. 
However, the barrier arguments in \cite{Silvestre2015,CamSilSne2017} although covering weak solutions to the PDE, do require these weak solutions to be bounded. 

Unlike in \cite{Silvestre2015,CamSilSne2017}, the result in Theorem \ref{thm:main soft potentials} makes no boundedness assumption and is therefore new. This subtle but non-trivial distinction would dissapear as soon as uniqueness results for unbounded weak solutions are discovered: indeed, the estimate for bounded weak solutions in \cite{Silvestre2015,CamSilSne2017} would be carried to arbitrary (that is unbounded) weak solutions by a density argument as soon as such weak solutions are known to be unique. 

The other two main results in this work are conditional estimates, Theorem \ref{thm:very soft potentials estimate} and Theorem \ref{thm:Coulomb case good estimate}, dealing with very soft potentials and the Coulomb potential respectively. These are discussed at length in the next section. However, we want to highlight the second of these two theorems here (Theorem \ref{thm:Coulomb case good estimate}), and state its assumptions in informal terms before going forward (the assumptions are Assumptions \ref{Assumption:Epsilon Poincare}-\ref{Assumption:Local Doubling} in the following section). We explain what we mean by ``$\varepsilon$-Poincar\'e inequality'' below, while the notion of ``doubling'' is reviewed in Section \ref{section: Main result}.

\begin{thm}\label{thm:main Coulomb conditional}
  Let $f: \mathbb{R}^d\times [0,T] \to\mathbb{R}$ be a weak solution of the Landau equation \eqref{eqn:Landau homogeneous} with Coulomb potential ($\gamma=-d$) and whose initial $\fin$ data has finite mass, energy and entropy.  Assume $f$ is doubling and the $\varepsilon$-Poincar\'e inequality holds uniformly over the time interval $[0,T]$. Then, for every $s\in (0,1)$ there is a constant $C_0$ determined by $s$, the initial and the $\varepsilon$-Poincar\'e and doubling constants of $f$ such that
  \begin{align*}
    f(v,t) \leq C_0(1+|v|)^s\left ( 1 + \frac{1}{t} \right )^{1+s},\;\;\forall\;t\in [0,T].
  \end{align*}  
\end{thm}

The theorem is still a conditional one, but we want to remark that within those conditions it produces a very fast regularization rate (as close to $t^{-1}$ as desired). We believe this reflects how the diffusive power of the equation is particularly strong when $f$ is singular, and that such a rate should be expected if the quadratic term does not overcome the diffusion term in general.

Of course, it could turn out that the assumptions of the theorem only hold when the resulting regularization rate becomes trivial; for instance, if $f(t)$ is assumed to be uniformly in $L^p$ when $t\in [0,T]$ for some sufficiently large $p$. However this is unclear at the moment. In any case, we note that a function $f$ can be doubling without enjoying high integrability, while the $\varepsilon$-Poincar\'e inequality ``almost'' holds for arbitrary $f \in L^1$, see Proposition \ref{prop:assumption 1 almost holds} and the discussion that follows it.

The approach taken towards all of our results starts from the tautological observation that a solution to \eqref{eqn:Landau homogeneous} solves a linear equation
\begin{align}\label{eqn:linear equation}
  \partial_t f = Q(g,f),
\end{align}
where $g = f$. Then, one aims to prove as much as possible about linear operators of the form $Q(g,\cdot)$ by only using properties of $g$ that are preserved by the equation (i.e. its mass, second moment, and entropy bounds). This is of course not a novel approach at all; the novelty here lies in the observation that the coefficients arising in the differential operator $Q(g,\cdot)$ are bounded by Muckenhoupt weights and that the strength of the term driving growth is of the same order as the term driving the dissipation. Let us elaborate on these two points. 

{For the latter point, suppose we were interested in controlling the $L^2$ norm of $f$ solving \eqref{eqn:linear equation} for some fixed $g$. Roughly speaking, this entails determining the positivity of the differential operator given by $L(\cdot) = -Q(g,\cdot)$, so
\begin{align*}
  Lf = -\dive(A_{g,\gamma} \nabla f - f \nabla a_{g,\gamma}) = -\textnormal{tr}(A_{g,\gamma}D^2f) - fh_{g,\gamma}.
\end{align*}
Integrating several times, we have that (for sufficiently smooth $f$)
\begin{align*}
  \int_{\mathbb{R}^d} fLf\;dv = \int_{\mathbb{R}^d} (A_{g,\gamma}\nabla f,\nabla f)\;dv-\tfrac{1}{2}\int_{\mathbb{R}^d}h_{g,\gamma} f^2\;dv.
\end{align*}
The ``potential'' $h_{g,\gamma}$ corresponds to the ``reactive'' part of the equation, in particular for the actual Landau equation \eqref{eqn:Landau homogeneous} with $\gamma=-d$ we would  have $g=f$ and $h_{f,\gamma} = f$. Then for general $g$, determining the positivity of $L$ would amount to the validity of the weighted Poincar\'e inequality
\begin{align*}
  \tfrac{1}{2}\int_{\mathbb{R}^d}h_{g,\gamma} f^2\;dv\leq \int_{\mathbb{R}^d} (A_{g,\gamma}\nabla f,\nabla f)\;dv.
\end{align*}
We are in fact content with bounding the growth of the $L^2$ norm instead of showing monotonicity and a slightly weaker inequality would suffice (for general $g$, monotonicity is not a reasonable expectation). Therefore, for the purposes of bounding the growth of $f$, it would suffice to show there is some $C = C(g)$ such that for all $f$,
\begin{align}\label{eqn:pre epsilon Poincare}
  \tfrac{1}{2}\int_{\mathbb{R}^d}h_{g,\gamma} f^2\;dv\leq \int_{\mathbb{R}^d} (A_{g,\gamma}\nabla f,\nabla f)\;dv + C\int_{\mathbb{R}^d}f^2\;dv.
\end{align}
Then, the validity of such an inequality depends on certain properties about the function $g$; moreover if (\ref{eqn:pre epsilon Poincare}) holds, it would mean that the diffusive term in $Q(g,\cdot)$ overcomes the growth term, represented by the potential $h_{g,\gamma}$. We will show that this inequality (actually a more general version of it) always holds when $\gamma \in (-2,0]$. When $\gamma \in [-d,-2]$ it is not clear at the moment  if this inequality holds, but we show that for such range of $\gamma$ it is always very close to holding (see Proposition \ref{prop:assumption 1 almost holds} together with Theorem \ref{thm:local weighted Sobolev inequalities}). }

This is where the theory of Muckenhoupt weights comes in. The fact that, for any positive $g \in L^1$ with finite mass, momentum and entropy, the coefficients $A_{g,\gamma}$ and $h_{g,\gamma}$ are controlled by Muckenhoupt weights facilitates the proof of certain weighted Poincar\'e and Sobolev inequalities, under conditions that vary depending on the range of $\gamma$. Concretely, we prove a refined version of \eqref{eqn:pre epsilon Poincare}, which we call the ``$\varepsilon$-Poincar\'e inequality'' and for $\varepsilon \in (0,1)$ reads as
\begin{align*}
  \int_{\mathbb{R}^d} h_{f,\gamma} \phi^2\;dv \leq \varepsilon \int_{\mathbb{R}^d} (A_{f,\gamma} \nabla \phi,\nabla \phi)\;dv + \Lambda(\varepsilon) \int_{\mathbb{R}^d} \phi^2\;dv,
\end{align*}
for all functions $\phi$ that makes the right hand side finite.
Again, we will show that such inequality always holds for any $f$ solution to the Landau equation with $\gamma \in (-2,0]$ and {\em almost holds} when $\gamma \in [-d,-2]$. 
This inequality is the decisive point in the proof. To see how, note that thanks to the above inequality, one can control the growth of the $L^p$ norms of $f$. For instance, for $p=2$, applying the inequality above with $\phi=f$, we are led to
\begin{align*}
\frac{1}{2}\frac{d}{dt} \int_{\mathbb{R}^d} f^2\;dv & =  - \int_{\mathbb{R}^d} (A_{f,\gamma} \nabla f,\nabla f)\;dv +\int_{\mathbb{R}^d} h_{f,\gamma} f^2\;dv \\
& \leq  -(1-\varepsilon) \int_{\mathbb{R}^d} (A_{f,\gamma} \nabla f,\nabla f)\;dv  +  \Lambda(\varepsilon) \int_{\mathbb{R}^d} f^2\;dv.
\end{align*}
The derivation of $L^\infty$ estimates (using the aforementioned weighted inequalities), follows to a great extent the De Giorgi-Nash-Moser theory for divergence for parabolic equations. As the respective ``parabolic'' equations have degenerate coefficients, the ideas developed in the 80's in order to deal with singular weights \cite{FKS82, GuWhe91}  has a significant role in our arguments (see, in particular, Lemma \ref{lem:weighted Moser estimate}). 

The validity of weighted inequalities like the one above is directly related to the theory of $A_p$-weights and to the positivity of Schr\"odinger operators with non-constant coefficients. In particular, for a generic distribution $f$, one can think of the linearized Landau operator $L(\cdot) = Q(g,\cdot)$ as a Schr\"odinger operator with very rough coefficients. The observation that even for a very singular $f$ the resulting coefficients of $L$ satisfy pointwise bounds with respect to certain $A_p$-weights, permits us to invoke the extensive literature on weighted integral inequalities (including, but not limited to, \cite{Fefferman1983}). We remark that for the present work it is necessary to obtain weighted inequalities corresponding to mass distributions $f$ which may be far from equilibrium, so we must assume as little about $f$ as possible.

Moreover, these inequalities are intrinsically related to the uncertainty principle and its generalization, as noted by Fefferman in \cite{Fefferman1983}. The relevance of uncertainty principles to kinetic equations has been noted before, and we refer the reader to work of Alexandre, Morimoto, Ukai, Xu and Yang \cite{AMUXY08} for further discussion. 


As already mentioned, an $\varepsilon$-Poincar\'e appears to hold for any $\gamma\geq -d$ for some universal $\varepsilon <<1$, independently of even the mass, second moment, and entropy of $f$. For moderately soft potentials, such an inequality also holds for all small $\varepsilon$. It is helpful to make an analogy with the situation seen for the linear heat equation with a Hardy potential, that is
\begin{align}\label{eqn:Heat equation Hardy potential}
  \partial_tf = \Delta f + \lambda |v|^{-2}f.
\end{align}
It was shown by Baras and Goldstein \cite{BarasGoldstein1984} that there is a critical value of $\lambda$ ( $=(d-2)/2)^2$) across which \eqref{eqn:Heat equation Hardy potential} admits or not solutions. Going back to \eqref{eqn:Landau homogeneous}, the question is whether the linear operator $Q(g,\cdot)$ lies ``on the side'' of regularization for arbitrary $g$. However, a simple calculation shows that for $g$ of the form $|v|^{-m}\chi_{B_1}$ with $m$ close to $d$, the $\varepsilon$-Poincar\'e inequality cannot hold with an arbitrary small $\varepsilon$, suggesting that an entirely different approach is needed for $\gamma<-2$.
 
\subsection{Organization of the paper} In Section \ref{section:preliminaries and main result} we discuss some preliminary concepts and state the main results. In Section \ref{section:A_p class and weighted inequalities} we review relevant facts from the theory of $A_p$ weights and in Section \ref{section:epsilon Poincare} we apply them to obtain weighted normed inequalities of relevance to the Landau equation. In Section \ref{section: Regularization in L infinity} we derive an energy inequality adapted to the Landau equation and prove a De Giorgi-Nash-Moser type estimate. Lastly, in Section \ref{section:Coulomb regularization} we focus on the case of the Coulomb potential and show a conditional but anomalous rate of rate of regularization in time.

\subsection{Notation} 

A constant will be said to be universal if it is determined by $d$, $\gamma$ and the mass, energy, and entropy of $\fin$. Universal constants will be denoted by $c,c_0,c_1,C_0,C_1,C$.

Vectors in $\mathbb{R}^d$ will be denoted by $v,w,x,y$ and so on, the inner product between $v$ and $w$ will be written $(v,w)$. $B_R(v_0)$ denotes the closed ball of radius $R$ centered at $v_0$, if $v_0=0$ we simply write $B_R$. When talking about a family of balls or a generic ball we will omit the radius and center altogether and write $B,B',B''$.
    The identity matrix will be noted by $\mathbb{I}$, the trace of a matrix $X$ will be denoted $\tr(X)$.  The initial condition for the Cauchy problem will always be denoted by $\fin$ and any constant of the form $C(\fin)$ is a constant that only depends on the mass, second moment and entropy of $\fin$.
	
 If $K$ is a convex set and $\lambda>0$, we will use $\lambda K$ to denote the set of points of the form $\lambda (v-v_c)+v_c$, where $v\in K$ and $v_c$ is the center of mass of $K$. In particular, if $K$ is a Euclidean ball or a cube of radius $r$, then $\lambda K$ is respectively a Euclidean ball or a cube of radius $\lambda r$.

\subsection{Acknowledgements} MPG is supported by NSF DMS-1412748 and DMS-1514761. NG is partially supported by NSF-DMS 1700307. NG would like to thank the Fields Institute for Research in Mathematical Sciences, where part of the work in this manuscript was carried out in the Fall of 2014. MPG would like to thank NCTS Mathematics Division Taipei for their kind hospitality. The authors thank Luis Silvestre for many fruitful communications, as well as the anonymous referee for many useful remarks that helped improve this paper.

\section{Preliminaries and main result}\label{section:preliminaries and main result}

\subsection{The matrix $A_{f,\gamma}$ and Riesz potentials}\label{section:the Matrix A and Riesz potentials}

We will be interested in how the matrix $A_{f,\gamma}$ defined in \eqref{Af_gamma} acts along a given unit direction. For that we introduce a few related functions.

\begin{DEF}\label{def: the matrix A and its functions} Let $e\in\mathbb{S}^{d-1}$, and define:
  \begin{align*}
	a^*_{f,\gamma,e}(v) & := (A_{f,\gamma}(v)e,e),\\
	a^*_{f,\gamma}(v) & := \inf \limits_{e\in\mathbb{S}^{d-1}}a^*_{f,\gamma,e}.
  \end{align*}

\end{DEF}

\begin{rem} Throughout the paper, we will often omit the subindex $f$, $\gamma$ or both and understand $A, a, h,a^*_e,a^*$ to be always determined by $f$ and $\gamma$ as in the above definitions.

\end{rem}

The following lemma is a standard lower bound for $A_{f,\gamma}$ and $a_{f,\gamma}$ in terms of the mass, energy and entropy of $f$ that will be used later in the manuscript, see Desvilletes and Villani \cite{DesVil2000a}.
\begin{lem}\label{lem: A lower bound in terms of conserved quantities}
  There is a constant $c$ determined by $\gamma$, the dimension $d$, as well as the mass, energy, and entropy of $f$, such that for all $v\in \mathbb{R}^d$ 
  \begin{align*}
  a_{f,\gamma}(v) &\geq c \langle v\rangle^{\gamma+2}, \\
    A_{f,\gamma}(v) &\geq c \langle v\rangle^\gamma\mathbb{I}.	  
  \end{align*}	  

\end{lem}

\subsection{Basic assumptions: weak solutions}
In all what follows we fix an initial condition $\fin$ which is a distribution on $\mathbb{R}^d$. For convenience, we will assume $\fin$ is normalized thus
\begin{align*}
  \int_{\mathbb{R}^d} \fin(v) \;dv = 1,\;\;\int_{\mathbb{R}^d} \fin(v) v\;dv =0,\;\;\int_{\mathbb{R}^d} \fin(v) |v|^2\;dv =d,	
\end{align*}	
and we shall assume the initial entropy of $\fin$ is finite
\begin{align*}
   H(\fin) =\int_{\mathbb{R}^d}\fin(v)\log(\fin(v))\;dv < \infty.
\end{align*}	
\begin{DEF} \label{def:solution} A weak solution of \eqref{eqn:Landau homogeneous} in $(0,T)$ with initial data $\fin$ is a function 
\begin{align*}
  f:\mathbb{R}^d\times [0,T)\mapsto \mathbb{R}
\end{align*}  
satisfying the following conditions:
  \begin{enumerate}  
  \item $f\geq 0$ a.e. in $\mathbb{R}^d\times (0,T)$ and	  
  \begin{align*}  
    f\in C([0,T);\mathcal{D}'),\;\; f(t) \in L^1_2\cap L\log(L)\;\;\forall\;t\in(0,T),\;f(0) = \fin.
  \end{align*}		
  \item For every $t\in (0,T)$
  \begin{align*}
     \int_{\mathbb{R}^d}f(v,t)\;dv = 1,\;\int_{\mathbb{R}^d}f(v,t)v_i\;dv = 0,\;\int_{\mathbb{R}^d}f(v,t)\;dv = d.  
  \end{align*}
  \item The equation is understood in the following weak sense: given $\eta \in C^2_c(\mathbb{R}^d)$ and $\phi\in C^2(\mathbb{R})$ such that $\phi''(s)\equiv 0$ for all large $s$ and $\lim\limits_{s\to 0}s^{-1}\phi(s) = 0$, and times $t_1<t_2$, we have
    \begin{align}
      & \int_{\mathbb{R}^d} \eta^2 \phi(f(t_2))\;dv-\int_{\mathbb{R}^d} \eta^2 \phi(f(t_1))\;dv \label{eqn:weak formulation}\\ 
      & = -\int_{t_1}^{t_2}\int_{\mathbb{R}^d} (A_{f,\gamma}\nabla f-f\nabla a_{f,\gamma}, \nabla (\eta^2 \phi'(f))) \;dvdt. \notag
    \end{align}

   \item The entropy functional $H(f(t)):=\int_{\mathbb{R}^d} f \log (f)dv$ is non increasing in $t$ and for any pair of times $0\leq t_1<t_2 <T$ we have
    \begin{align*}
      H(f(t_1))-H(f(t_2)) = \int_{t_1}^{t_2} D(f(t))\;dt \ge 0,
    \end{align*}	
    where $D(f)$ denotes the entropy production, that is
    \begin{align*}
      D(f(t)) & = \int_{\mathbb{R}^d}4(A_{f,\gamma}\nabla \sqrt{f}, \nabla \sqrt{f})-f h_{f,\gamma}\;dv = -\int_{\mathbb{R}^d}Q(f,f)\log(f)\;dv.
    \end{align*}
\end{enumerate}
  
\end{DEF}
	
\subsection{Further assumptions for very soft potentials: Local Doubling and an $\varepsilon$-Poincar\'e inequality}

For $\gamma\leq -2$ all our results are conditional, that is, the regularity estimates are obtained under two further assumptions on the solution $f$, which may not hold in general. These two assumptions are described below.

Given $\varepsilon \in (0,1)$ we say that ``$f$ satisfies the $\varepsilon$-Poincar\'e inequality'' if there is some positive constant $\Lambda>0$ such that for any Lipschitz, compactly supported $\phi$ we have
\begin{align*}
  \int_{\mathbb{R}^d} \phi^2 h_{f,\gamma}\;dv \leq \varepsilon \int_{\mathbb{R}^d} (A_{f,\gamma} \nabla\phi,\nabla\phi)\;dv + \Lambda\int_{\mathbb{R}^d} \phi^2\;dv.
\end{align*}
This can be thought of as an ``uncertainty estimate'' for the Landau equation (see the discussion regarding the uncertainty principle at the end of Section \ref{section: contributions}). We will also say that ``$f$ satisfies the strong Poincar\'e inequality'' if there is a decreasing, non-negative function $\Lambda:(0,1)\mapsto \mathbb{R}$ such that for any smooth $\phi$ and $\varepsilon \in (0,1)$ we have
\begin{align*}
  \int_{\mathbb{R}^d} \phi^2 h_{f,\gamma}\;dv \leq \varepsilon \int_{\mathbb{R}^d} (A_{f,\gamma}\nabla\phi,\nabla\phi)\;dv + \Lamp \int_{\mathbb{R}^d}\phi^2\;dv.
\end{align*}
\begin{rem}\label{rem:compactness of supports and inequalities}
  Although the inequality is stated for compactly supported $\phi$, it is clear (by density) that once it holds for compactly supported $\phi$ it also holds for any $\phi$ with
  \begin{align*}
    \int_{\mathbb{R}^d} (A_{f,\gamma} \nabla\phi,\nabla\phi)\;dv<\infty \textnormal{ and } \int_{\mathbb{R}^d} \phi^2\;dv<\infty.  
  \end{align*}	  
  Accordingly, throughout the manuscript we may apply the $\varepsilon$-Poincar\'e inequality (and other weighted Sobolev inequalities) to $\phi$ which may not be compactly supported but satisfy these two integrability conditions. 
\end{rem}

The two conditions above might be thought of as bounds on the spectral functional,
\begin{align}\label{eqn:Lambda f epsillon definition}
  \Lampf := \sup \left \{ \int_{\mathbb{R}^d}\phi^2 h_{f,\gamma} - \varepsilon( A_{f,\gamma}\nabla \phi,\nabla \phi)\;dv \;:\; \|\phi\|_{L^2(\mathbb{R}^d)} = 1,\;\; t\in (0,T) \right \}.
\end{align}
This brings us to the first of the extra Assumptions needed for very soft potentials.
\begin{Assumption}\label{Assumption:Epsilon Poincare}
  There are constants $C_p,\kappa>0$ such that for $\Lampf$ defined in \eqref{eqn:Lambda f epsillon definition},
  \begin{align}\label{eqn:Spectral Assumption on f}
    \Lampf \leq C_P\varepsilon^{-\kappa}\;\; \forall\;\varepsilon\in (0,1).
  \end{align}
\end{Assumption}

\begin{rem}\label{rem: h in Ld/2 implies epsilon Poincare} 
 If $\gamma \in (-2,0]$, any $f$ satisfies the $\varepsilon$-Poincar\'e inequality for all $\varepsilon \in (0,1)$ with universal constants (Theorem \ref{thm:sufficient conditions for the Poincare inequality}). This can be intuited from the fact that for $\gamma \in (-2,0]$ the function $ h_{f,\gamma} $ defined in \eqref{hf_gamma} belongs to $ L^{p}_{\textnormal{loc}}$ for some $p>d/2$, with a norm controlled by a universal constant. Then, at least for functions with compact support, the standard Sobolev inequality and the integrability of $h_{f,\gamma}$ can be used to prove a $\varepsilon$-Poincar\'e inequality.

 \end{rem}

The other important assumption that will be required for $\gamma \leq -2$ is the following ``local doubling'' property.
\begin{Assumption}\label{Assumption:Local Doubling}
  There is a constant $C_D>1$ such that for all times $t\in (0,T)$, we have
  \begin{align}\label{eqn:doubling local}
    \int_{B_{2r}(v_0)}f(v,t)\;dv \leq C_D\int_{B_r(v_0)}f(v,t)\;dv,\;\;\forall\;v_0\in\mathbb{R}^d,\;r\in(0,1).
  \end{align}

\end{Assumption}
  The reason we say this is ``local'' is that it only involves balls with radius less than $1$. A locally integrable function which is globally doubling (meaning the above holds for balls with any radius $r>0$) cannot belong to $L^1(\mathbb{R}^d)$. On the other hand, plenty of functions in $L^1(\mathbb{R}^d)$ can be locally doubling. For a discussion on how reasonable this assumption is, see Section \ref{sec:discussion of extra assumptions}.

\subsection{Main results}\label{section: Main result}

  The main results of this manuscript are summarized in the next theorems. First, we present sufficient conditions to guarantee that $f$ satisfies a $\varepsilon$-Poincar\'e inequality.
\begin{thm}\label{thm:sufficient conditions for the Poincare inequality}

  Let $\gamma\in (-2,0]$. Then, given any $f\geq 0$ with mass, second moment, and entropy as $\fin$, the following holds for any Lipschitz $\phi$ with compact support
  \begin{align}\label{eqn:epsilon_Poincare_inequality gamma>-2}
    \int_{\mathbb{R}^d} \phi^2 h_{f,\gamma} \;dv \leq \varepsilon \int_{\mathbb{R}^d} (A_{f,\gamma}\nabla\phi,\nabla\phi)\;dv + C(\fin,d,\gamma)\varepsilon^{\frac{\gamma}{2+\gamma}}\int_{\mathbb{R}^d}\phi^2\langle v\rangle ^\gamma\;dv.
  \end{align} 
  Let $\gamma \in [-d,-2]$. Assume $f\geq 0$ is doubling in the sense of \eqref{eqn:doubling local}, and assume there exist $s>1$ and a modulus of continuity $\eta(\cdot)$ such that for any $Q$ with side length $r\in(0,1)$
 \begin{align*}
    |Q|^{\frac{1}{d}} \left( \fint_{Q}{h_{f,\gamma}}^s\;dv\right)^{\frac{1}{2s}}\left(\fint_{Q}(a_{f,\gamma}^*)^{-s}\;dv\right)^{\frac{1}{2s}}  \leq \eta(r).
  \end{align*}	
  Then, the following Poincar\'e inequality holds for any $\varepsilon \in (0,1)$
  \begin{align}\label{eqn:epsilon_Poincare_inequality strongerII_bounded}
    \int_{\mathbb{R}^d} \phi^2h_{f,\gamma} \;dv \leq \varepsilon \int_{\mathbb{R}^d} (A_{f,\gamma}\nabla\phi,\nabla\phi)\;dv + C(\fin,d,\gamma)\tilde \eta(\varepsilon)\int_{\mathbb{R}^d}	\phi^2\;dv,
  \end{align} 
  where $\tilde \eta:(0,1)\mapsto\mathbb{R}$ is a decreasing function with $\eta(0+)=\infty$ determined by $\eta$.
\end{thm}
A very important consequence of Theorem \ref{thm:sufficient conditions for the Poincare inequality} are the following $L^\infty$-estimates for solutions to \eqref{eqn:Landau homogeneous}. 

\begin{thm}\label{thm:main_1} 

Let $\gamma \in (-2,0]$ and $f:\mathbb{R}^d\times [0,\infty)\to\mathbb{R}$ be a weak solution to \eqref{eqn:Landau homogeneous}. There is a constant $C=C(\fin,d,\gamma)$ such that for any $R>1$,
  \begin{align*}
    \|f(t)\|_{L^\infty(B_R(0))}\leq CR^{-\gamma\tfrac{d}{2}}\left (1+\frac{1}{t}\right )^{d/2},\;\;\forall\;t\in(0,+\infty).	  
  \end{align*}	   	
  
  \end{thm}
  
\begin{thm}\label{thm:very soft potentials estimate}  
  Let $-d< \gamma \leq -2$. Let $f:\mathbb{R}^d\times [0,T)\to\mathbb{R}$ be a classical solution to \eqref{eqn:Landau homogeneous} for which Assumptions \ref{Assumption:Epsilon Poincare} and \ref{Assumption:Local Doubling} hold. Then, there is a constant $C=C(\fin,d,\gamma,C_P,\kappa_P,C_D)$ such that for any $R>1$,
  \begin{align*}
    \|f(t)\|_{L^\infty(B_R(0))} \leq CR^{-\gamma\tfrac{d}{2}}\left (1+\frac{1}{t}\right )^{\frac{d}{2}},\;\;\forall\;t\in(0,T).
  \end{align*}	  
  
\end{thm}

  \begin{thm}\label{thm:Coulomb case good estimate} 
  
  Let $\gamma =-d$ and $f:\mathbb{R}^d\times [0,T)\to\mathbb{R}$ be a classical solution to \eqref{eqn:Landau homogeneous} for which Assumptions \ref{Assumption:Epsilon Poincare} and \ref{Assumption:Local Doubling} hold. Given any $s \in (0,1)$ there is a constant $C=C(s, \fin,d,\gamma,C_P,\kappa_P,C_D)$ such that for any $R>1$,
  \begin{align*}
    \|f(t)\|_{L^\infty(B_R(0))} \leq C R^s\left (1+\frac{1}{t}\right )^{1+s},\;\;\forall\;t\in(0,T).
  \end{align*}	  
\end{thm}

\begin{rem}\label{rem:Coulomb case good estimate discussion} The assumptions in Theorem \ref{thm:Coulomb case good estimate}, particularly \eqref{eqn:epsilon_Poincare_inequality strongerII_bounded}, are not satisfactory from the point of view of a general theory for weak solutions. As such, Theorem \ref{thm:Coulomb case good estimate} should be seen as a conditional result which provides relatively simple conditions that guarantee strong regularization. Most interesting, however, is the rate of decay in the $L^\infty$ estimate. For any solution to an uniformly parabolic equation one expects an estimate of the form 
  \begin{align*}	
    \|f(t)\|_{L^\infty(\mathbb{R}^d)} \lesssim t^{-d/2},	
  \end{align*}
  at least for small times $t$ when the initial data belongs to $L^1(\mathbb{R}^d)$. Instead, for the case of Coulomb potential we have a much faster (even if conditional) rate, a rate which can be made to be  arbitrary close to $t^{-1}$. 
  
  This faster rate of regularization can be seen as a reflection of the fact that the ``diffusivity'' in \eqref{eqn:Landau homogeneous} (for a lack of a better term) has a strength of order comparable to $(-\Delta)^{-1}f$, therefore, near points where $f$ fails to be in $L^{\frac{d}{2}}$, the diffusivity becomes stronger, and therefore tends to drive the solution $f$ towards smoothness more dramatically. 
  
  In our proof, the specific way in which the stronger regularization arises in Theorem \ref{thm:Coulomb case good estimate} is through the fact that $h_{f,\gamma}=f$ when $\gamma=-d$ gives an extra power of $f$ in the $\varepsilon$-Poincar\'e inequality, or alternatively, via a nonlinear integral inequality found first by Gressman, Krieger, and Strain \cite{GreKriStr2012}, \eqref{eqn:GressmannKriegerStrain} in Section \ref{section:Coulomb regularization}. This inequality captures in an analytical way the fact that the diffusivity in the equation (which tends to drive the solution towards smoothness) is at least of the same order as the quadratic term (which tends to drive the solution towards blow-up).

\end{rem}
  
A side-product of our study are weighted Sobolev inequalities with weights involving $\astar_{f,\gamma}$ and $h_{f,\gamma}$, this under (relatively) mild assumptions on $f$ when $\gamma <-2$. These inequalities may be of independent interest. In this manuscript they are used in Section \ref{section: Regularization in L infinity} in deriving De Giorgi-Nash-Moser type estimates with the proper weights.

\begin{thm}\label{lem:Inequalities Sobolev weight aII}
  Let $\gamma \in [-2,0]$, there is a constant $C=C(d,\gamma,\fin)$ such that given a function $f\geq 0$ with mass, energy, and entropy equal $\fin$ then for all Lipschitz, compactly supported functions $\phi:\mathbb{R}^d\to\mathbb{R}$,
  \begin{align*}
    \left ( \int_{\mathbb{R}^d}  \phi^{\frac{2d}{d-2}}(\astar)_{f,\gamma}^{\frac{d-2}{d}} \;dv\right )^{\frac{d-2}{d}}\leq C\int_{\mathbb{R}^d} \astar_{f,\gamma} |\nabla \phi|^2\;dv+C\int_{\mathbb{R}^d}|\phi|^2\;dv.		
  \end{align*}			
  Furthermore, if we have $f:\mathbb{R}^d\times I \mapsto \mathbb{R}$ such that $f(\cdot,t)$ satisfies the previous assumptions for all $t\in I$, then for any $\phi:\mathbb{R}^d\times  I\to\mathbb{R}$ which is spatially Lipschitz and spatially compactly supported 
  \begin{align*}
    \left ( \int_{I}\int_{\mathbb{R}^d} \phi^{2(1+\frac{2}{d})} \astar_{f,\gamma} \;dvdt \right )^{\frac{1}{(1+\frac{2}{d})}} & \leq C\left ( \int_{I} \int_{\mathbb{R}^d} \astar_{f,\gamma} |\nabla \phi|^2 \;dvdt +  \sup \limits_{I} \int_{\mathbb{R}^d} \phi^{2}\;dv \right ).
  \end{align*}
  Such inequalities extend to $\phi$'s that might not be Lipschitz in the $v$ variable, as long as they can be approximated in the respective norm by Lipschitz, compactly supported $\phi$'s.
\end{thm}

\begin{thm}\label{thm:Inequalities Sobolev weight gamma below -2}
	   
   Let  $\gamma \in (-d,-2)$. Let $f$ and $\phi$ be as in the first part of the previous theorem, and assume further $f$ is locally doubling \eqref{eqn:doubling local}, then there is a constant $C = C(d,\gamma,\fin,C_D)$ (with $C_D$ as in \eqref{eqn:doubling local}) such that
   \begin{align*}
      \left ( \int_{\mathbb{R}^d}  \phi^{\frac{2d}{d-2}}(\astar_{f,\gamma})^{\frac{d-2}{d}} \;dv\right )^{\frac{d-2}{d}}\leq C\int_{\mathbb{R}^d} \astar_{f,\gamma} |\nabla \phi|^2+\phi^2\;dv.		
    \end{align*}			
    Moreover, if $f$ and $\phi$ are as in the second previous theorem, and additionally $f$ satisfies \eqref{eqn:doubling local} uniformly in time, then there is some $C=C(d,\gamma,\fin,C_D)$ such that
      \begin{align*}
     \left ( \int_{I}\int_{\mathbb{R}^d} \phi^{2(1+\frac{2}{d})} \astar_{f,\gamma} \;dvdt \right )^{\frac{1}{(1+\frac{2}{d})}} & \leq C\left ( \int_{I} \int_{\mathbb{R}^d} \astar_{f,\gamma} |\nabla \phi|^2 \;dvdt +  \sup \limits_{I} \int_{\mathbb{R}^d} \phi^{2}\;dv \right ).
    \end{align*} 
    Finally, for $\gamma=-d$ if $f,\phi$ are as in the first part of the previous theorem, $f$ satisfies \eqref{eqn:doubling local}, then for any $m\in (0,\tfrac{d}{d-2})$ there is some constant $C=C(d,m,\gamma,\fin,C_D)$ such that	 
   \begin{align*}
      \left ( \int_{\mathbb{R}^d}  \phi^{2m}(\astar_{f,\gamma})^{m} \;dv\right )^{\frac{1}{m}}\leq C\int_{\mathbb{R}^d} \astar_{f,\gamma} |\nabla \phi|^2+\phi^2\;dv.		
    \end{align*}			
    Lats but not least, if $f,\phi:\mathbb{R}^d\times I \mapsto \mathbb{R}$ are as in the second part of the previous theorem, $f$ satisfies \eqref{eqn:doubling local} uniformly in time, then for any $q\in (1,2(1+\frac{2}{d}))$ there is a constant $C = C(d,q,\gamma,\fin,C_D)$ such that
    \begin{align*}
     \left ( \int_{I}\int_{\mathbb{R}^d} \phi^{q} \astar_{f,\gamma} \;dvdt \right )^{\frac{2}{q}} & \leq C\left ( \int_{I} \int_{\mathbb{R}^d} \astar_{f,\gamma} |\nabla \phi|^2 \;dvdt +  \sup \limits_{I} \int_{\mathbb{R}^d} \phi^{2}\;dv \right ).
    \end{align*} 
\end{thm}

\subsection{Regarding the extra assumptions for the case $\gamma\leq -2$}\label{sec:discussion of extra assumptions}

Let us make a few remarks about Assumption \ref{Assumption:Epsilon Poincare} and Assumption \ref{Assumption:Local Doubling}, and make the case that although they are rather unsatisfactory, the fact that they ``almost hold'' is encouraging as it suggests that \eqref{eqn:Landau homogeneous} may enjoy regularity estimates in the regime of very soft potentials. Let us explain and elaborate on this point.

In light of Theorem \ref{thm:sufficient conditions for the Poincare inequality}, \ref{thm:very soft potentials estimate} and \ref{thm:Coulomb case good estimate}, regularity for solutions in the case of very soft potentials is guaranteed for any time interval where one can show (i) Assumption \ref{Assumption:Local Doubling}, and (ii) there are some $s>1$ and $\delta>0$ such that
\begin{align}\label{eqn:Sufficient Condition for epsilon Poincare}
  |Q|^{\frac{1}{d}}\left ( \fint_{Q}h_{f,\gamma}^s\;dv\right )^{\frac{1}{2s}}\left (\fint_{Q} (\astar_{f,\gamma})^{-s} dv\right )^{\frac{1}{2s}}  \leq  \varepsilon,
\end{align}
for some small enough $\varepsilon$ and all cubes of side length $\leq \delta$. This condition is always true for moderately soft potentials $\gamma \in (-2,0]$,  (see Proposition \ref{new_pro_global_space}).  It is therefore encouraging that a quantity with close resemblance to \eqref{eqn:Sufficient Condition for epsilon Poincare} is bounded by a constant depending only on $d$ and $\gamma$, even for $\gamma \in [-d,-2]$, as shown in the next proposition.

\begin{prop}\label{prop:assumption 1 almost holds}
   For each $\gamma \in [-d,-2]$ there is a  $C=C(d,\gamma)$ such that for any non-negative $f\in L^1(\mathbb{R}^d)$ and any cube $Q \subset \mathbb{R}^d$,
  \begin{align}\label{eqn:assumption 1 almost holds ratio}
    |Q|^{\frac{1}{d}}\left ( \fint_{Q}h_{f,\gamma}\;dv \right )^{\frac{1}{2}}\left ( \fint_{Q} (a_{f,\gamma})^{-1} dv\right )^{\frac{1}{2}}   \leq C(d,\gamma). 
  \end{align}
\end{prop}  

\begin{proof}
  Let $Q$ be a cube and $v_0$ it's center. If $\gamma=-d$, then
  \begin{align*}
    \fint_{Q}a_{f,\gamma}\;dv \approx \int_{\mathbb{R}^d}f(w)\max\{|Q|^{\frac{1}{d}},|v_0-w|\}^{2-d}\;dw \gtrsim |Q|^{\frac{2-d}{d}}\int_{Q}f(w)\;dw.
  \end{align*}
  Therefore (the implied constants depending only on $d$ and $\gamma$)
  \begin{align*}
    |Q|^{\frac{2}{d}}\fint_{Q}f\;dv & \lesssim \fint_{Q}a_{f,\gamma}\;dv.
  \end{align*}	
  For the case $\gamma \in(-d,-2]$ the computations are somewhat similar, first note that
  \begin{align*}
    |Q|^{\frac{2}{d}}\fint_{Q}h\;dv & \approx |Q|^{\frac{2}{d}}\int_{\mathbb{R}^d} f(w) \max \{|Q|^{\frac{1}{d}},|v_0-w|\}^\gamma\;dw\\
    & \lesssim |Q|^{\frac{2+\gamma}{d}}\int_{Q}f(w)\;dw+\int_{\mathbb{R}^d\setminus Q} f(w)|v_0-w|^{2+\gamma}\;dw.
  \end{align*}
  Next, for $\gamma \in (-d,-2]$ we have
  \begin{align*}
    \fint_{Q} a\;dv \gtrsim |Q|^{\frac{2+\gamma}{d}}\int_{Q}f(w)\;dw+\int_{\mathbb{R}^d\setminus Q} f(w)|v_0-w|^{2+\gamma}\;dw.	
  \end{align*}	
  It follows that
  \begin{align*}
    |Q|^{\frac{2}{d}} \fint_{Q}h_{f,\gamma}\;dv  \lesssim \fint_{Q}a_{f,\gamma}\;dv.	
  \end{align*}
  We conclude that both for $\gamma=-d$ and $\gamma\in (-d,-2]$ we have,
  \begin{align*}
    |Q|^{\frac{1}{d}}\left ( \fint_{Q}h_{f,\gamma}\;dv \right )^{\frac{1}{2}}\left ( \fint_{Q} a_{f,\gamma}^{-1} dv\right )^{\frac{1}{2}}  \lesssim \left ( \fint_{Q}a_{f,\gamma}\;dv\right )^{\frac{1}{2}}\left (  \fint_{Q}(a_{f,\gamma})^{-1}\;dv \right )^{\frac{1}{2}}. 
  \end{align*}
  Finally, as we shall see later (Proposition \ref{prop: Riesz potentials are A1 weights}), $a_{f,\gamma}$ is a $\mathcal{A}_1$ with a constant determined by $d$ and $\gamma$, which means that
  \begin{align*}
    \left ( \fint_{Q}a_{f,\gamma}\;dv\right )^{\frac{1}{2}}\left (  \fint_{Q}(a_{f,\gamma})^{-1}\;dv \right )^{\frac{1}{2}} \lesssim 1,
  \end{align*}
  with the implied constants being determined by $d$ and $\gamma$.	
\end{proof}
   
  The estimate in the last proposition yields a bound similar to \eqref{eqn:Sufficient Condition for epsilon Poincare}, except $\astar$ is replaced with $a$ and there is no $s>1$. Since $a_{f,\gamma}$ is of class $\mathcal{A}_p$  for any $p\geq 1$ (see Proposition \ref{prop: Riesz potentials are A1 weights}), it is possible to show that for some $s=s(d,\gamma)>1$ we have (again with implied constants determined by $d$ and $\gamma$)
  \begin{align*}
    \left(\fint_{Q} a_{f,\gamma}^{-s} \;dv \right)^{1/s} \lesssim C \fint_{Q} a_{f,\gamma} \;dv,\;\;    \left ( \fint_Q h^s\;dv \right )^{\frac{1}{s}}\lesssim \fint_Q h\;dv,
  \end{align*}
  the second inequality holding only in the case $\gamma\neq -d$. The same observation extends to $h_{f,\gamma}$ in the case where $\gamma \neq -d$. So the lack of a power $s$ in Proposition \ref{prop:assumption 1 almost holds} is not the main obstacle to overcoming the need for Assumption \ref{Assumption:Epsilon Poincare}. On the other hand, if Assumption \ref{Assumption:Local Doubling} holds, we shall see later that for $\gamma\leq -2$, $\astar$ and $a$ are pointwise comparable in compact sets (Lemma \ref{prop:a vs a star gamma larger than -2}).  So, in Proposition \ref{prop:assumption 1 almost holds} the fact that we have $a$ instead of $\astar$ is also not a serious shortcoming. 
  
  This brings us to the crux of the problem for Assumption \ref{Assumption:Epsilon Poincare} and the limits of Proposition \ref{prop:assumption 1 almost holds}: the estimate in \eqref{eqn:assumption 1 almost holds ratio} is in some sense sharp for generic $f$, since there are $f$ for which this quantity remains of order $1$ even as the size of the cube $Q$ goes to zero. This is what underlines the idea mentioned in the Introduction that the strength of the nonlinearity for the Landau equation is of the same order as the diffusion, and therefore obtaining unconditional estimates for the Landau equation (if they at all exist) will require looking in a more precise way at the relationship between these two terms. 
  
  Lastly, let us make two more points regarding Assumption \ref{Assumption:Local Doubling}. The first one is that it is worthwhile to investigate whether solutions to \eqref{eqn:Landau homogeneous} gain a local doubling property similar to that in Assumption \ref{Assumption:Local Doubling}, since this property is not dissimilar to the weak Harnack inequality for supersolutions to parabolic equations. The second point regarding the local doubling assumption is related to the isotropic equation \eqref{eqn:Krieger-Strain}, and we record it as a Remark.
\begin{rem}
  Exchanging $\astar_{f,\gamma}$ with $a_{f,\gamma}$, the estimates in Section \ref{section:A_p class and weighted inequalities} can be used to show inequalities akin to those theorems \ref{thm:sufficient conditions for the Poincare inequality} and \ref{thm:Inequalities Sobolev weight gamma below -2} \emph{without assuming $f$ is locally doubling}. In particular, one always has the inequality for every $\gamma \in (-d,-2]$,
  \begin{align*}
    \int_{\mathbb{R}^d}|\phi|^2h_{f,\gamma}\;dv \leq C(d,\gamma)\int_{\mathbb{R}^d}a_{f,\gamma}|\nabla \phi|^2\;dv.
  \end{align*}
 The above inequality is still far from an $\varepsilon$-Poincar\'e inequality, and it's not clear that for a generic $f$ one can do better than some positive constant. In any case, we still have the corresponding weighted Sobolev inequality, that is, for $\gamma\in(-d,-2]$ there is a constant $C(d,\gamma)$ such that
  \begin{align*}
    \left ( \int_{\mathbb{R}^d} |\phi|^{\frac{2d}{d-2}} a_{f,\gamma}^{\frac{d}{d-2}}\;dv\right )^{\frac{d-2}{d}} \leq C(d,\gamma) \int_{\mathbb{R}^d}a_{f,\gamma}|\nabla \phi|^2\;dv.
  \end{align*}
\end{rem}

As we are done with the preliminaries and the statement of the main results, let us say a word about the location of the proofs throughout the rest of the paper: Theorem \ref{thm:sufficient conditions for the Poincare inequality} is proved in Section \ref{section:A_p class and weighted inequalities},  Theorem \ref{lem:Inequalities Sobolev weight aII} in Section \ref{Space time inequalities}, Theorem \ref{thm:Coulomb case good estimate} and  Theorem \ref{thm:very soft potentials estimate}  are proved at the end of  Section \ref{section: Regularization in L infinity}, lastly, Theorem \ref{thm:main_1} in Section \ref{section:Coulomb regularization}.

\section{The $\Ap$ class and weighted inequalities for the Landau equation}\label{section:A_p class and weighted inequalities}

  We recall necessary facts from the theory of Muckenhoupt weights. The interested reader should consult the book of Garcia-Cuerva and Rubio De Francia \cite[Chapter IV]{GarciaCuervaRubioDeFrancia1985} and \cite{Turesson} for a thorough discussion.  
  
  \begin{DEF}
Given $p\in(1,\infty)$ and a ball $B_0$, a nonnegative locally integrable function $\omega:\mathbb{R}^d\to\mathbb{R}$ is said to be a {\bf{$\Ap(B_0)$-weight}} if 
    \begin{align}\label{def_Ap}
      \sup \limits_{B \subset 8B_0}\left( \fint_{B} \omega(v) \;dv \right) \left(\fint_{B} \omega(v)^{-\frac{1}{p-1}} \;dv \right)^{p-1}<\infty.
    \end{align}
   The above supremum will be called the $\Ap$ constant of $\omega$.
    
     The class $\Ainfty(B_0)$ will refer to the union of all the classes $\mathcal{A}_p(B_0)$, $p>1$. Finally, the class $\A1$ refers to weights such that for some $C>0$ one has for any $B\subset 8B_0$ and almost every $v_0 \in B$,
    \begin{align*}
      \fint_{B}\omega(v) \;dv \leq C \omega(v_0).
    \end{align*}
	The infimum of the $C$'s for which the above holds will be called the $\mathcal{A}_1$ constant of $\omega$.
  \end{DEF}
  The class $\Ap$ appears in many important questions in analysis, here we are mostly interested in their properties relevant to weighted Sobolev inequalities. As we will see below, natural weights associated to the Landau equation will be in this class. 
  
  \begin{DEF}A non-negative locally integrable function $w$ is said to satisfy the {\em reverse H\"older inequality} with exponent $m\ge 1 $ and constant $C$ if for any ball $B$ we have
   \begin{align*}
     \left (\fint_{B}w^m\;dv \right )^{\frac{1}{m}} \leq C\fint_{B}w\;dv.
   \end{align*}  
  \end{DEF}
  As it turns out, $\mathcal{A}_p$ weights satisfy a reverse H\"older inequality.
  
  \begin{lem}\label{lem:A_p implies reverse Holder}
    \emph{($\mathcal{A}_p$ weights satisfy a reverse H\"older inequality.)} Let $\omega$ be a $\mathcal{A}_p$ weight, $1\leq p< +\infty$. There are $\varepsilon >0$ and $C>1$ --determined by $p$ and the $\mathcal{A}_p$ constant for $\omega$-- such that $\omega$ satisfies the reverse H\"older inequality with exponent $1+\varepsilon$ and constant $C$.
  \end{lem}	  
  \begin{proof}
    See \cite[Chapter 4, Lemma 2.5]{GarciaCuervaRubioDeFrancia1985} for a proof.  
  \end{proof}  

  The next proposition shows that Riesz potentials are indeed $\mathcal{A}_1$ weights.
  
  \begin{prop}\label{prop: Riesz potentials are A1 weights}
  For $m \in [0,d)$ and $f\in L^1(\mathbb{R}^d)$, $f\geq 0$, the function $F:=f* |v|^{-m}$ is in the class $\mathcal{A}_1$ with an $\mathcal{A}_1$ constant that only depends on $m$ and $\|f\|_{L^1(\mathbb{R}^d)}$. Consequently the function 
  \begin{align*}
    h_{f,\gamma}(v)= \int f(w)|v-w|^\gamma \;dw,
  \end{align*}
  belongs to the  class $\mathcal{A}_1$ for each $-d<\gamma \le 0$, $d\ge3$ and 
  \begin{align*}
  a_{f,\gamma}(v)= \int f(w)|v-w|^{2+\gamma}\; dw,
  \end{align*}
  belongs to $\mathcal{A}_1$ for each $-d\le \gamma \le -2$, $d\ge3$.
  
\end{prop}	  

\begin{proof}
 Fix $m \in [0,d)$ and $w \in \mathbb{R}^d$, we analyze the function
  \begin{align*}	
    v \to |v-w|^{-m},\;\;v\in\mathbb{R}^d.	
  \end{align*}
   Lemma \ref{lem:averages for powers of v} and Remark \ref{rem:averages for powers of v} imply that  
  \begin{align*}
    \fint_{B_r(v_0)}|v-w|^{-m}\;dv \leq C_m|v_1-w|^{-m},\;\;\forall\; v_1 \in B_r(v_0),\;w\in\mathbb{R}^d.
  \end{align*}
  Thus, $|v-w|^{-m}$ belongs to the class $\mathcal{A}_1$. Since the class $\mathcal{A}_1$ is convex, it follows that the convolution of $|v|^{-m}$ with any nonnegative $L^1$ function is also in the class $\mathcal{A}_1$. Indeed, let $f\in L^1(\mathbb{R}^d)$, $f\geq 0$ and consider the function
  \begin{align*}
    F(v) = \int_{\mathbb{R}^d}f(w)|v-w|^{-m}\;dw.	  
  \end{align*}	  
  It is evident that $F\in L^1_{\textnormal{loc}}(\mathbb{R}^d)$ by Fubini's theorem. Moreover, if $v_0\in\mathbb{R}^d$ and $r>0$,
  \begin{align*}
    \fint_{B_r(v_0)} F(v)\;dv & = \int_{\mathbb{R}^d}f(w)\fint_{B_r(v_0)}|v-w|^{-m}\;dvdw\\
	  & \leq C_m\int_{\mathbb{R}^d}f(w)|v_0-w|^{-m}\;dw = C_m F(v_1),\;\;\forall\;v_1\in B_r(v_0).
  \end{align*}	  
  \end{proof}
  Proposition \ref{prop: Riesz potentials are A1 weights} says that the function $a_{f,\gamma}$ (respectively $h_{f,\gamma}$) for $\gamma \in [-d,-2]$ (resp. for $\gamma \in (-d,0]$) is an $\mathcal{A}_1$-weight and satisfies a reverse H\"older inequality for some exponent $q=1+\varepsilon$ (Lemma \ref{lem:A_p implies reverse Holder}).  However, we can obtain a better, explicit exponent for which these functions satisfy a reverse H\"older inequality. The following result illustrates what can be expected for the special weights we are dealing with. 
  \begin{lem}\label{lem:a is reverse Holder with a nice constant}
    Let {{$\gamma \in [-d,-2]$}} and $m \in (0,\frac{d}{|2+\gamma|})$. There is a universal constant $C_{m,\gamma}$ such that $a_{f,\gamma}$ satisfies the reverse H\"older inequality with exponent $m$ and constant $C_{m,\gamma}$. In other words, for any ball $B$ (or cube $Q$) we have		
    \begin{align*}
      \left ( \fint_{B} a_{f,\gamma}^m \;dv \right )^{\frac{1}{m}} \leq C_{m,\gamma} \fint_{B}a_{f,\gamma}\;dv.	
    \end{align*}	 
  \end{lem}

  \begin{proof}[Proof of Lemma \ref{lem:a is reverse Holder with a nice constant}] 
    The Lemma follows immediately applying Remark \ref{rem:reverse Holder for inverse powers of v} and Proposition  \ref{prop:reverse Holder is a convex condition} (proved below) taking $\Phi=|v|^{2+\gamma}$,  and $g = f$.
  \end{proof}

  \begin{prop}\label{prop:reverse Holder is a convex condition}
   Let $g\in L^1(\mathbb{R}^d)$ be non-negative.  Let $\Phi(x)$ be a non-negative, locally integrable function that satisfies the reverse H\"older inequality with exponent $m>1$ and constant $C>0$. Then, the convolution $g* \Phi$ also satisfies the reverse H\"older inequality with same exponent $m$ and same constant $C$. That is, for any ball $B$ we have
    \begin{align*}
      \left (\fint_{B} (g * \Phi (v))^m \;dv \right )^{\frac{1}{m}} \leq C \fint_{B} (g *\Phi)(v)\;dv.
    \end{align*}	
  \end{prop}

  \begin{proof}
    It is clear that	  
    \begin{align*}
     \left (\fint_{B} (g * \Phi (v))^m \;dv \right )^{\frac{1}{m}} =\left ( \fint_B \left | \int_{\mathbb{R}^d} \Phi(v-w) g(w)\;dw \right |^mdv \right )^{\frac{1}{m}}.
    \end{align*}			
    Minkowski's integral inequality says that
    \begin{align*}
     \left ( \fint_B \left | \int_{\mathbb{R}^d} \Phi(v-w) g(w)\;dw \right |^mdv \right )^{\frac{1}{m}} \leq \int_{\mathbb{R}^d} \left (\fint_B  |\Phi(v-w)|^m\;dv \right )^{\frac{1}{m}}g(w)\;dw. 	 		
    \end{align*}			
    By the reverse H\"older assumption for $\Phi$, we have for every $w\in \mathbb{R}^d$,
    \begin{align*}
     \left (\fint_B  |\Phi(v-w)|^m\;dv \right )^{\frac{1}{m}} \leq C\fint_{B}|\Phi(v-w)|\;dv = C\fint_{B}\Phi(v-w)\;dv.
    \end{align*}			
    Lastly, we use the positivity of $g$ and apply Fubini's theorem, which yields,
    \begin{align*}
     \left (\fint_{B} (g * \Phi (v))^m \;dv \right )^{\frac{1}{m}} & \leq C\int_{\mathbb{R}^d} g(w)\fint_{B}\Phi(v-w)\;dvdw = C\fint_B (g* \Phi)(v)\;dv.	 	 	 		
    \end{align*}
  \end{proof}

  \begin{rem}\label{rem:reverse Holder for inverse powers of v}
    As shown in the Appendix, $|v|^{-q}$ is in the class $\mathcal{A}_1$ for all $0<q<d$. In particular, if $m\ge 0$ and $p \ge 0$ are such that $0\le mp<d$, we have
    \begin{align*}
      \fint_{B}|w|^{-mp}\;dw \leq C_{mp}|v|^{-mp},\;\;\forall\;v\in B.		
    \end{align*}			
    This may be rewritten as
    \begin{align*}
      \left ( \fint_{B}|w|^{-mp}\;dw \right )^{\frac{1}{m}}\leq (C_{mp})^{\frac{1}{m}}|v|^{-p},\;\;\forall\;v\in B.		
    \end{align*}			
    Taking the average for $v\in B$, it follows that	
    \begin{align*}
      \left (\fint_{B}|w|^{-mp}\;dw \right )^{\frac{1}{m}}\leq (C_{mp})^{\frac{1}{m}}\fint_{B}|w|^{-p}\;dw.		
    \end{align*}			
  \end{rem}

\subsection{Averages of $\astar_{f,\gamma}$ and $a_{f,\gamma}$}
  Recall the definition of $\astar_{f,\gamma}$ and $a_{f,\gamma}$, 
  \begin{align*}
    \astar_{f,\gamma} &= \inf \limits_{e\in\mathbb{S}^{d-1}} C_\gamma \int_{\mathbb{R}^d} |v-w|^{2+\gamma}\langle \Pi(v-w)e,e\rangle f(w)\;dw,\\
    a_{f,\gamma} &= c(d,\gamma)\int_{\mathbb{R}^d}f(w)|v-w|^{-2-\gamma}\;dw.
  \end{align*}
  Since $|\langle \Pi(v-w)e,e\rangle|\leq 1$ it is clear that $\astar_{f,\gamma} (v)\le a_{f,\gamma}(v)$ for each $v\in \mathbb{R}^d$. 
  
  We briefly recall bounds on the functions $a_{f,\gamma}$ and $A_{f,\gamma}$ that will be used later. These estimates are now standard (see Desvilletes and Villani \cite{DesVil2000a} or Silvestre \cite[Lemma 3.2]{Silvestre2015}) but we include a proof for the sake of completeness. 
  
\begin{prop}\label{prop:a vs a star gamma larger than -2}
  If $\gamma \in [-2,0]$ there exist constants that only depend on $d$, $\fin$ and $\gamma$ such that for any $f\geq 0$ and $v\in\mathbb{R}^d$ we have
  \begin{align*}
    c_{\fin,d,\gamma}\langle v\rangle^{2+\gamma} & \leq a_{f,\gamma}(v) \;\leq C_{d,\gamma}\langle v\rangle^{2+\gamma},\\ 
    c_{\fin,d,\gamma}\langle v\rangle^\gamma & \leq \astar_{f,\gamma}(v) \leq C_{\fin,d,\gamma}\langle v\rangle^\gamma.
  \end{align*}
  Moreover, for any $v\neq 0$, we have $(A_{f,\gamma}(v)e,e) \leq C\astar_{f,\gamma}$ for $e=v/|v|$.
\end{prop}

\begin{proof}
  We drop the subindices $f,\gamma$. The lower bounds were already stated in Lemma \ref{lem: A lower bound in terms of conserved quantities}, and therefore here we only need to prove the upper bounds for $a$ and $\astar$. 
  
  Since $2+\gamma\geq 0$, for any $v,w\in\mathbb{R}^d$
  \begin{align*}
    |v-w|^{2+\gamma} \leq \left ( |v|+|w| \right )^{2+\gamma} \leq 4 \left ( |v|^{2+\gamma}+|w|^{2+\gamma}\right ).
  \end{align*}
  Using conservation of mass and energy, and that $2+\gamma\leq 2$, we have 
  \begin{align*}
    a_{f,\gamma}(v) = C(d,\gamma)\int_{\mathbb{R}^d} f(w)|v-w|^{2+\gamma}\;dw & \leq 4C(d,\gamma)\int_{\mathbb{R}^d}f(w)\left( |v|^{2+\gamma}+|w|^{2+\gamma} \right )\;dw\\
	  & = 4C(d,\gamma)|v|^{2+\gamma}+4C(d,\gamma)\int_{\mathbb{R}^d}f(w)|w|^{2+\gamma}\;dw\\
	  & \le C(\fin,d,\gamma) \langle v\rangle^{2+\gamma}.
  \end{align*}
  In conclusion, for some constant $C(d,\gamma)$
  \begin{align*}
    a_{f,\gamma}(v) \leq  C(d,\gamma)\langle v\rangle^{2+\gamma}.
  \end{align*}
  It remains to prove the upper bound for $\astar$. Let $v\in\mathbb{R}^d$, $|v|\geq 1$ and $e = {v}/|v|$. We estimate the ``tail'' of the integral defining $a^*_{f,\gamma,e}$, noting that
  \begin{align*}
    & \int_{B_{\langle v\rangle}(v)^c}f(w)|v-w|^{2+\gamma}(\Pi(v-w)e,e)\;dw\\
    & = \int_{B_{\langle v\rangle}(v)^c}f(w)|v-w|^{\gamma}(|v-w|^2-(v-w,e)^2 )\;dw.  		  
  \end{align*}	  
  Since $e$ and $v$ are parallel, it follows that the difference of squares in the last integral is just the square of the distance from $w$ to the line spanned by $e$, therefore
  \begin{align*}
      |v-w|^2-(v-w,e)^2 \leq |w|^2,\;\;\forall\;w\in\mathbb{R}^d.
  \end{align*}
  With this inequality and using that $\gamma\leq 0$, we conclude that
  \begin{align*}
    \int_{B_{\langle v\rangle/2}(v)^c}f(w)|v-w|^{2+\gamma}(\Pi(v-w)e,e)\;dw \leq 2^{-\gamma}\langle v\rangle^\gamma \int_{B_{\langle v\rangle/2}(v)^c}f(w)|w|^2\;dw.  		  
  \end{align*}	  
  On the other hand, since $2+\gamma\geq 0$, we have
  \begin{align*}
    \int_{B_{\langle v\rangle/2}(v)}f(w)|v-w|^{2+\gamma}(\Pi(v-w)e,e)\;dw & \leq \langle v\rangle^{2+\gamma}\int_{B_{\langle v\rangle/2}(v)}f(w)\;dw\\
	  & \leq C(\gamma) \langle v\rangle^\gamma\int_{B_{\langle v\rangle/2}(v)}f(w)\langle w\rangle^2\;dw.
  \end{align*}	  
  This shows that
  \begin{align*}
    (A_{f,\gamma}(v)e,e)\leq C(d,\gamma)\langle v\rangle^\gamma,\;\;e=\frac{v}{|v|},  
  \end{align*}
  from where it follows that $\astar_{f,\gamma}(v)\leq C\langle v\rangle^\gamma$ and $(A_{f,\gamma}(v)e,e)\leq C\astar_{f,\gamma}(v)$ for all $v\neq 0$.
\end{proof}

\begin{prop}\label{prop:a vs a star gamma below -2}
  Let $\gamma \in [-d,-2)$, then, under Assumption \ref{Assumption:Local Doubling}, there is a $C$ such that 
  \begin{align*}  
    a_{f,\gamma}(v)\leq C\langle v\rangle^{\max\{-\gamma-2,2\}}\astar_{f,\gamma}(v),\;\; C=C(\fin,d,\gamma,C_D).
  \end{align*}
  Moreover, if one  assumes $f$ has a uniformly bounded moment of order $-\gamma-2$, we have
  \begin{align*}  
    a_{f,\gamma}(v)\leq C\langle v\rangle^{2}\astar_{f,\gamma}(v),
  \end{align*}
  with the $C$ this time depending also on the bound on the $-\gamma-2$ moment of $f$.
\end{prop}

\begin{proof}
  Let $r>0$, then we may write
  \begin{align*}
    a_{f,\gamma}(v) = C(d,\gamma)\int_{B_r(v)}f(w)|v-w|^{2+\gamma}\;dw+C(d,\gamma)\int_{\mathbb{R}^d\setminus B_r(v)}f(w)|v-w|^{2+\gamma}\;dv. 
  \end{align*}
  Choosing $r = \frac{1}{2}\langle v\rangle$ and using that $2+\gamma\leq 0$ it follows that
  \begin{align*}
    a_{f,\gamma}(v) \leq C\int_{B_{\langle v\rangle/2}(v)}f(w)|v-w|^{2+\gamma}\;dw+C\langle v\rangle^{2+\gamma}.
  \end{align*}
  Again, since $2+\gamma\leq 0$, we have $|v-w|^{2+\gamma}\leq 1$ in $\mathbb{R}^d\setminus B_1(v)$, therefore
  \begin{align*}
    \int_{B_{\langle v\rangle/2}(v) \setminus B_1(v)}f(w)|v-w|^{2+\gamma}\;dw \leq \int_{B_{\langle v\rangle/2}(v)}f(w)\;dw . 
  \end{align*}
  Now, in $B_{\langle v\rangle/2}(v)$ we have 
  \begin{align*}
    \int_{B_{\langle v\rangle/2}(v) \setminus B_1(v)}f(w)|v-w|^{2+\gamma}\;dw & \leq \langle v\rangle^{-2}\int_{\mathbb{R}^d}f(w)|w|^{2}\;dv \leq C\langle v\rangle^{-2-\gamma}\astar. 
  \end{align*}
  By the local doubling property of $f$, for any $e\in\mathbb{S}^{d-1}$ we have 
  \begin{align*}
    \int_{B_1(v)}f(w)|v-w|^{2+\gamma}\;dw & \leq C\int_{B_1(v)}f(w)|v-w|^{2+\gamma}(\Pi(v-w)e,e)\;dw \leq a_{f,\gamma,e}(v).	
  \end{align*}	  
  Taking the infimum over $e$, and using the estimates above, 
  \begin{align*}
    a_{f,\gamma}(v) \leq C\langle v\rangle^2 \astar_{f,\gamma}+C\langle v\rangle^{-2-\gamma}\astar_{f,\gamma}(v) \leq C\langle v\rangle^{\max\{2,-\gamma-2\} }\astar_{f,\gamma}(v).	  
  \end{align*}	  
  Lastly, if $f$ has a finite $-\gamma-2$ moment, then above one has the bound
  \begin{align*}
    \int_{B_{\langle v\rangle/2}(v) \setminus B_1(v)}f(w)|v-w|^{2+\gamma}\;dw & \leq \langle v\rangle^{2-\gamma}\int_{\mathbb{R}^d}f(w)|w|^{-\gamma-2}\;dw \\
	  & \leq C\langle v\rangle^2 \astar_{f,\gamma}(v),
  \end{align*}
  from where the respective estimate follows.
\end{proof}

\begin{rem}
  Observe that if $d=3$, then $\max\{-\gamma-2,2\} =2$ for all $\gamma \geq -4$, so the above Proposition yields the pointwise  relation $a_{f,\gamma}\leq C\langle v\rangle^2 \astar$ for the case of Coulomb potential when $d=3$ under Assumption \ref{Assumption:Local Doubling}.
\end{rem}

We now show a result that will be used later to prove Theorem \ref{thm:sufficient conditions for the Poincare inequality}. 
\begin{prop}\label{new_pro_global_space}
  If $\gamma \in [-2,0]$ there is $s=s(d,\gamma) >1$ such that for any cube $Q \subset \mathbb{R}^d$ with sides of length $\leq 1$, we have
  \begin{align*}
    |Q_r|^{\frac{1}{d}} \left( \fint_{Q_r}{h_{f,\gamma}}^{s}\;dv\right)^{\frac{1}{2s}}\left(\fint_{Q_r}(a_{f,\gamma}^*)^{-s}\;dv\right)^{\frac{1}{2s}} \le C(\fin,d,\gamma) |Q|^{\frac{2+\gamma}{d}}.
  \end{align*}	
\end{prop}

\begin{proof}
  This estimate will follow by a kind of interpolation, estimating an integral respectively outside and inside a ball, using in each case respectively either that $\gamma> -2$ or the mass and second moment of $f$. 
  
   Fix a cube $Q$ with center $v_0 \in \mathbb{R}^d$ and sides of length $r$ with $2r\leq 1 $. By Proposition \ref{prop:a vs a star gamma larger than -2} for any $s>1$ we have,
    \begin{align*}
      \fint_{Q}{(a_{f,\gamma}^*)^{-s}}\;dv \leq C(\fin,d,\gamma,s)\langle v_0\rangle^{-s\gamma}.
    \end{align*}
  Above we have used that the diameter of $Q$ was at most $1$ along with the lower estimate for $a^*_{f,\gamma}$ from Lemma \ref{lem: A lower bound in terms of conserved quantities}. Now, let us bound the integral of $h_{f,\gamma}$. To do this, we make use of the identity 
    \begin{align*}
    \fint_{Q}h_{f,\gamma}\;dv  = C(d,\gamma)\int_{\mathbb{R}^d}\fint_{Q}f(w)|v-w|^\gamma\;dvdw,  
  \end{align*}
  and of the estimate
  \begin{align*}
     \fint_{Q} |v-w|^\gamma\;dv \leq C(d,\gamma) \left ( \max\{r,|v_0-w|\} \right )^\gamma,
  \end{align*}	 
 (see Lemma \ref{lem:averages for powers of v}). We also use the fact that $h_{f,\gamma}$ is an $\mathcal{A}_1$ weight and hence satisfies a reverse H\"older inequality 
 \begin{align*}
   \left(  \fint_{Q}h_{f,\gamma}^s\;dv\right)^{1/s} \le C  \fint_{Q}h_{f,\gamma}\;dv,
 \end{align*}
 for some $s=s(d,\gamma)>1$.

 Next, we split $\mathbb{R}^d$ as the union of the ball  $B_{\langle v_0\rangle/2}(v_0)$ and its complement. Using conservation of mass, the second moment of $f$, and that $-\gamma\leq 2$, it follows that
   \begin{align*}
  \int_{B_{|v_0|/2}(v_0)} f(w) \int_Q |v-w|^{\gamma}\;dv\;dw  & \leq |Q|^{1+\frac{\gamma}{d}}\langle v_0 \rangle^{\gamma} \int_{\mathbb{R}^d}f(w)\langle w\rangle^{-\gamma}\;dw\\
	 & \leq C(d,\gamma,\fin)|Q|^{1+\frac{\gamma}{d}} \langle v_0\rangle^{\gamma}.
  \end{align*}
  For the second integral we distinguish two cases depending on whether $|v_0|$ is greater or less than one. 
  If $|v_0| \ge 1$ we have 
  \begin{align*}
  \int_{B_{|v_0|/2}^c(v_0)} f(w) \int_Q |v-w|^{\gamma}\;dv\;dw &\le C(d,\gamma) |Q|   \int_{B_{|v_0|/2}^c(v_0)}f(w)|v_0-w|^{\gamma}\;dw \\
  & \leq C(d,\gamma)|Q| | v_0|^{\gamma} \int_{B_{|v_0|/2}(v_0)^c}f(w)\;dw\\
	  & \leq  C(\fin,d,\gamma)\langle v_0\rangle^{\gamma}|Q|.
  \end{align*}
  On the other hand, if $|v_0| < 1$ then $1\leq \langle v_0\rangle \le 2$ and 
   \begin{align*}
  \int_{B_{|v_0|/2}^c(v_0)} f(w) \int_Q |v-w|^{\gamma}\;dv\;dw  & \leq  r^\gamma |Q| \int_{B_{|v_0|/2}} f(w)\;dw \leq   C(\fin,d,\gamma) |Q|^{1+\frac{\gamma}{d}} \langle v_0\rangle^{\gamma}.
  \end{align*}
  Combining these inequalities, we see that 
  \begin{align}\label{hwithr=1}
   \fint_{Q}h\;dv  \le C(\fin,d,\gamma)\langle v_0\rangle^{\gamma} |Q|^{\frac{\gamma}{d}} 
  \end{align}
  and  we conclude that
   \begin{align*}
    |Q|^{\frac{2}{d}} \left( \fint_{Q}{h_{f,\gamma}}^s\;dv\right)^{\frac{1}{s}}\left(\fint_{Q}(a_{f,\gamma}^*)^{-s}\;dv\right)^{\frac{1}{s}} \leq C(\fin,d,\gamma)   |Q|^{\frac{2+\gamma}{d}}. 
    \end{align*}
\end{proof}

The following Lemma, dealing with the cases where $\gamma \in [-d,-2]$, is one of the places where Assumption \ref{Assumption:Local Doubling} is key. It guarantees that $\astar$ is (locally) a $\mathcal{A}_1$ weight with a constant depending on the constant $C_D$ from Assumption \ref{Assumption:Local Doubling}.

\begin{lem}\label{lem:a star is A1 if f is doubling}
 Let $\gamma \in [-d,-2)$ and assume $f$ satisfies Assumption \ref{Assumption:Local Doubling}, then, for any $v_0 \in \mathbb{R}^d$ and $r\in(0,1/2)$ we have
  \begin{align*}
    \fint_{B_r(v_0)}a^*_{f,\gamma}(v) \;dv\leq Ca^*_{f,\gamma}(v_1), \;\;\forall\;v_1 \in B_r(v_0),
  \end{align*}	  
  with a constant $C$ determined by $\gamma$, $d$, and the constant $C_D$ in Assumption \ref{Assumption:Local Doubling}.
\end{lem}

\begin{proof}
  Fix $e\in\mathbb{S}^{d-1}$, $v_0 \in \mathbb{R}^d$ and $r\in (0,1)$. 
  We are first going to prove that 
  \begin{align}\label{eqn:weight a_e A1c center}
    \fint_{B_r(v_0)}a_e^*(v)\;dv\leq Ca_e^*(v_0).	  
  \end{align}	
  We define  the set 
  \begin{align*}
    S(v_0,r,e) := \{ w \in\mathbb{R}^d \mid \textnormal{dist}(v_0-w,[e]) \geq 2r\}.
  \end{align*}
  We make use of the inequality	
  \begin{align*}
    \fint_{B_r(v_0)}a^*_e(v)\;dv  & \leq Ca^*_e(v_0)+r^{2+\gamma}\int_{B_{2r}(v_0)}f(w)\;dw\\
	  & \;\;\;\;+r^2\int_{S(v_0,r,e)^c \cap B_{2r}(v_0)^c}f(w)|v_0-w|^\gamma\;dw,
  \end{align*}
  proved in Proposition \ref{prop:weight a_e is almost A1}
  
  Next, let $K_e(v_0) = \{ w\mid (\Pi(v_0-w)e,e)\geq \tfrac{1}{2} \}$, applying (repeatedly) Assumption \ref{Assumption:Local Doubling} to a cube contained in $B_{2r}(v_0) \cap K_e(v_0)$ it can be seen that
  \begin{align*}
    \int_{B_{2r}(v_0)}f(w)\;dw \leq C\int_{B_{2r}(v_0) \cap K_e(v_0) }f(w)\;dw,\;\; C = C(C_D).
  \end{align*}	  
  Then, since $2+\gamma\leq 0$, it follows that
  \begin{align*}
    r^{2+\gamma}\int_{B_{2r}(v_0)}f(w)\;dw \leq C\int_{B_{2r}(v_0)}f(w)|v_0-w|^{2+\gamma}(\Pi(v_0-w)e,e)\;dw, 	  
  \end{align*}	   
  and thus
  \begin{align*}
    r^{2+\gamma}\int_{B_{2r}(v_0)}f(w)\;dw \leq Ca_e^*(v_0),\;C=C(C_D).
  \end{align*}	   
  An analogous argument using the local doubling property also yields  
  \begin{align*}
    r^2\int_{S(v_0,r,e)^c \cap B_{2r}(v_0)^c}f(w)|v_0-w|^\gamma\;dw \leq Ca^*_e(v_0),\;\;C=C(C_D).
  \end{align*}	  
  These last two estimates prove \eqref{eqn:weight a_e A1c center}. 
  
  Taking the minimum over $e\in\mathbb{S}^{d-1}$ on both sides of \eqref{eqn:weight a_e A1c center} it is immediate that
  \begin{align*}
    \fint_{B_r(v_0)}\astar(v_0)\;dv\leq C\astar(v_0).	  
  \end{align*}
  To finish the proof, assume now that $r\in (0,1/2)$ and note that if $v_1 \in B_{r}(v_0)$ then $B_{r}(v_0) \subset B_{2r}(v_1)$, thus, we can apply the above estimate in the ball $B_{2r}(v_1)$, so that
  \begin{align*}
    \fint_{B_r(v_0)}\astar(v)\;dv	\leq 2^d \fint_{B_{2r}(v_1)}\astar(v)\;dv \leq C2^d\astar(v_1).
  \end{align*}	  
\end{proof}

The following lemma uses the estimates on $a^*_{f,\gamma}$ and the kernel averages in Proposition \ref{prop:kernel averages} to show that $a^*_{f,\gamma}$ satisfies a certain condition that will be used later in Proposition \ref{prop:LpLp via interpolation with weight}  for a weighted Sobolev inequality.
\begin{lem} \label{this_is_new_cond_Sob}
  Let $\gamma \in [-d,-2)$, $s>1$, and $m \in (0, \tfrac{d}{|2+\gamma|})$.  If Assumption \ref{Assumption:Local Doubling} holds, then for any $t\in (0,T)$, $v_0 \in \mathbb{R}^d$, and $r\in(0,1/2)$, we have 
  \begin{align*}
    \left ( \fint_{Q}{(a^*_{f,\gamma})}^m\;dv\right )^{\frac{1}{m}} \left ( \fint_{Q}(a^*_{f,\gamma})^{-s}\;dv \right )^{\frac{1}{s}} \leq C.
  \end{align*}
  Here $Q=Q_{r}(v_0)$, $C=C(d,\gamma,m,\fin,C_D)$,  where $C_D$ is the constant from Assumption \ref{Assumption:Local Doubling}.
\end{lem}

\begin{proof}
  By Minkowski's inequality, if $K_e(v) := |v|^{2+\gamma} \left( \Pi(v)e,e\right)$ ($e\in\mathbb{S}^{d-1}$ fixed), we have
  \begin{align*}
    \left ( \fint_{Q}(a^*_{e})^m\;dv\right )^{\frac{1}{m}} \leq \int_{\mathbb{R}^d}f(w)\left ( \fint_{Q}K_e(v-w)^m\;dv\right )^{\frac{1}{m}}dw.
  \end{align*}
  From this inequality and Proposition \ref{prop:kernel averages}, we obtain for any $v_1 \in Q$,
  \begin{align*}
   \left ( \fint_{Q}(a^*_e)^m\;dv \right )^{\frac{1}{m}}& \leq C\int_{S(v_1,r,e)^c}f(w)|v_1-w|^{2+\gamma}(\Pi(v_1-w)e,e)\;dw\\
   & \;\;\;\;+Cr^2\int_{S(v_1,r,e)}f(w)\max\{2r,|v_1-w|\}^{\gamma }\;dw.
  \end{align*}
  We observe that the first term is controlled by $a^*_e(v_1)$ (cf. Proposition \ref{prop:weight a_e is almost A1}), therefore
  \begin{align*}
   \left ( \fint_{Q}(a^*_e)^m\;dv\right )^{\frac{1}{m}} & \leq Ca^*_e(v_1)+r^2\int_{S(v_1,r,e)}f(w)\max\{2r,|v_1-w|\}^\gamma\;dw,\;\;\forall\;v_1\in Q.
  \end{align*}
  Next, arguing using Assumption \ref{Assumption:Local Doubling} just as in Lemma \ref{lem:a star is A1 if f is doubling}, and taking the minimum with respect to $e$, we conclude that
  \begin{align*}
   \left ( \fint_{Q}(\astar)^m\;dv\right )^{\frac{1}{m}} & \leq C\astar(v_1),\;\;\forall\;v_1\in Q,
  \end{align*}
  where $C=C(d,\gamma,m,\fin,C_D)$. This proves the Lemma, since for any $s>1$ we have
  \begin{align*}
    \left ( \fint_{Q}(\astar)^m\;dv\right )^{\frac{1}{m}} \left ( \fint_{Q}(\astar)^{-s}\;dv \right )^{\frac{1}{s}} \leq C\left ( \inf \limits_{Q} \astar \right )\left ( \fint_{Q}(\astar)^{-s}\;dv \right )^{\frac{1}{s}}\leq C.	
  \end{align*}
\end{proof}

\subsection{Global weighted inequalities} In this section we make us of known results on Sobolev-type inequalities with weights, namely those of the form
  \begin{align}\label{eqn:weighted Sobolev}
    \left( \int_{Q}|\phi|^qw_1\;dv\right)^{1/q}\leq C\left( \int_{Q} |\nabla \phi|^p\;w_2dv\right)^{1/p},
  \end{align} 
  where $q\leq p$, $C$ is independent of $\phi$, and $\phi$ has either compact support in $Q$ or its average over $Q$ vanishes. The literature on inequalities of this form is much too vast to discuss properly here, so we shall limit ourselves to highlighting the work of Chanillo and Wheeden \cite{ChanilloWheeden1985,ChanilloWheeden1985II} and work of Sawyer and Wheeden \cite{SawWhe1992}, since they are the ones that we invoke directly here.
  
  These results say --roughly speaking-- that \eqref{eqn:weighted Sobolev} is guaranteed to hold if certain averages involving $w_1$ and $w_2$ are bounded over all cubes contained in a cube slightly larger than $Q$. For $p=2$ and $2\leq q<\infty$, $s>1$, a cube $Q$, weights $w_1$ and $w_2$, we define
\begin{align*}
  \sigma_{q,2,s}(Q,w_1,w_2) := |Q|^{\frac{1}{d}-\frac{1}{2}+\frac{1}{q}}\left (\fint_{Q}w_1^s\;dv \right )^{\frac{1}{q s}}\left (\fint_{Q}w_2^{-s}\;dv \right )^{\frac{1}{2 s}}.
\end{align*}
If $p=q=2$ and $w_1$ is in the class $\mathcal{A}_p$ for \emph{some} $p$ and $w_2$ satisfies a reverse H\"older inequality\footnote{so, $w_2$ also being a $\mathcal{A}_p$ weight would suffice, per Lemma \ref{lem:A_p implies reverse Holder}} then an alternative to $\sigma$ is (see \cite[Theorem 1.2]{ChanilloWheeden1985}),
  \begin{align*}
    \sigma_{2,2}(Q,w_1,w_2):=  C|Q|^{\frac{2}{d}} \frac{\fint_{Q}w_1\;dv}{\fint_{Q}w_2\;dv}. 
  \end{align*}
  The theorem we will be using is stated below. For its proof the reader is referred to \cite[Theorem 1]{SawWhe1992}, \cite[Theorem 1.5]{ChanilloWheeden1985II}, and \cite[Theorem 1.2]{ChanilloWheeden1985}.
\begin{thm}\label{thm:local weighted Sobolev inequalities} 
  Consider two weights $w_1$, $w_2$, a cube $Q$, $2\leq q<\infty$ and $s>1$. Then, for any Lipschitz function $\phi$ which has compact support in the interior of a cube $Q$ or is such that $\int_Q \phi\;dv = 0$, we have
  \begin{align}\label{eqn:Local Weighted Sobolev Inequality General Weights}
    \left( \int_{Q}|\phi|^qw_1\;dv\right)^{1/q}\leq   \mathcal{C}_{q,2,s}(Q,w_1,w_2)\left( \int_{Q}|\nabla \phi|^2\; w_2dv\right)^{1/2},
  \end{align} 
  where for some constant $C(d,s,q)$,
  \begin{align*}
    \mathcal{C}_{q,2,s}(Q,w_1,w_2) := C(d,s,q)\sup \limits_{Q' \subset 8Q} \sigma_{q,2,s}(Q',w_1,w_2).
  \end{align*}
  
\end{thm}

In the next lemma we make use of Theorem \ref{thm:local weighted Sobolev inequalities} to obtain an inequality over $\mathbb{R}^d$. 
\begin{lem}\label{lem:Global weighted Sobolev from local ones}
  For $\ell\in(0,1)$, $q\geq 2$ and some $s>1$, define
  \begin{align*}
    \mu(\ell) := \sup \{   \mathcal{C}_{q,2,s}(Q,w_1,w_2) \;:\;	Q \textnormal{ cube of side length } \leq \ell \} .	  \end{align*}	  
  Then, there is a $C=C(q)$ such that any compactly supported Lipschitz $\phi$ satisfies 
  \begin{align*}
    \left ( \int_{\mathbb{R}^d}|\phi|^q w_1\;dv\right )^{\frac{1}{q}} \leq C(q)\mu(8\ell)\left ( \int_{\mathbb{R}^d}  |\nabla \phi|^2\; w_2 dv \right )^{\frac{1}{2}}+C(q)\left (\int_{\mathbb{R}^d}E_\ell(w_1)|\phi|^2\;dv \right )^{\frac{1}{2}}.
  \end{align*}
  Here, $E_{\ell}(w_1)$ denotes the function
  \begin{align*}
    E_\ell(w_1) = \sum \limits_{Q\in\mathcal{Q}_\ell} \frac{1}{|Q|}\left ( \int_{Q}w_1\;dv \right )^{\frac{2}{q}}\chi_{Q},  
  \end{align*}
  and $\mathcal{Q}_{\ell}$ denotes the class of cubes having the form $[0,\ell]^d+ k\ell$ with $k\in \mathbb{Z}^d$.
\end{lem}

\begin{proof}
  The idea is to decompose the integral over $\mathbb{R}^d$ as integrals over the family of cubes $\mathcal{Q}_{\ell}$, which are non overlapping and cover all of $\mathbb{R}^d$. Then, we apply the weighted Sobolev inequality from Theorem \ref{thm:local weighted Sobolev inequalities} in each cube --making sure to control the extra term involving the average of the function over each cube. 
  
  Denote by $(\phi)_{Q}$ the average of $\phi$ over $Q$. Young's inequality yields
  \begin{align*}	
    \int_{\mathbb{R}^d}|\phi|^qw_1\;dv & = \sum \limits_{Q\in\mathcal{Q}_\ell} \int_{Q}|\phi|^qw_1\;dv  = \sum \limits_{n} \int_{Q}|\phi-(\phi)_{Q}+(\phi)_{Q} |^qw_1\;dv \\
	  & \leq C(q)\sum \limits_{Q\in\mathcal{Q}_{\ell}} \int_{Q}|\phi-(\phi)_{Q}|^qw_1\;dv+C(q)\sum \limits_{Q\in\mathcal{Q}_\ell} |(\phi)_{Q}|^q \int_{Q}w_1\;dv. 	
  \end{align*}
  On the other hand, for each $Q\in\mathcal{Q}_\ell$ we apply Theorem \ref{thm:local weighted Sobolev inequalities}, and get 
  \begin{align*}
    \int_{Q}|\phi-(\phi)_{Q}|^qw_1\;dv & \leq \mathcal{C}_{q,2,s}(Q,w_1,w_2)^q\left ( \int_{Q} |\nabla \phi|^2\;w_2 dv \right )^{\frac{q}{2}}\\
	& \leq \mu(8\ell)^q \left ( \int_{Q} |\nabla \phi|^2\;w_2 dv \right )^{\frac{q}{2}}.
  \end{align*}	
  Since $q\geq 2$,  
  \begin{align*}
    \sum \limits_{Q\in\mathcal{Q}_\ell}\left ( \int_{Q} |\nabla \phi|^2 w_2\;dv \right )^{\frac{q}{2}} & \leq C \left (\sum\limits_{Q\in\mathcal{Q}_\ell} \int_{Q}|\nabla \phi|^2\;w_2 dv \right )^{\frac{q}{2}} = C \left (\int_{\mathbb{R}^d}|\nabla \phi|^2\;w_2 dv \right )^{\frac{q}{2}}.	  
  \end{align*}	
  It follows that
  \begin{align*}
    \int_{\mathbb{R}^d}|\phi|^qw_1\;dv\leq C(q)\mu(8\ell)^q\left (\int_{\mathbb{R}^d}|\nabla \phi|^2\;w_2 dv \right )^{\frac{q}{2}}+C\sum \limits_{Q\in\mathcal{Q}_\ell} |(\phi)_{Q}|^q \int_{Q}w_1\;dv .	  	  
  \end{align*}	  
  To deal with the second term, we apply Jensen's inequality
  \begin{align*}  
|(\phi)_{Q}|^q   \int_{Q}w_1\;dv  \leq \left( \frac{1}{|Q|}\int_{Q}|\phi|^2\;dv \right )^{\frac{q}{2}}\int_{Q}w_1\;dv ,\;\;\forall\;Q\in\mathcal{Q}_\ell.
  \end{align*}
  Adding all these inequalities, it follows that 
  \begin{align*}  
    \sum \limits_{Q\in\mathcal{Q}_\ell} \int_{Q}w_1\;dv |(\phi)_{Q}|^q \leq \left ( \int_{\mathbb{R}^d}E_{\ell}(w_1)|\phi|^2\;dv \right )^{\frac{q}{2}}.
  \end{align*}
  In conclusion,
  \begin{align*}	
    \int_{\mathbb{R}^d}|\phi|^qw_1\;dv \leq C \mu(8\ell)^q \left (\int_{\mathbb{R}^d}|\nabla \phi|^2\;w_2 dv \right )^{\frac{q}{2}}+C\left ( \int_{\mathbb{R}^d}E_{\ell}(w_1)|\phi|^2\;dv \right )^{\frac{q}{2}},
  \end{align*}  	
  raising both sides to the power $1/q$, the Lemma follows.
\end{proof}

\section{The $\varepsilon$-Poincar\'e inequality}\label{section:epsilon Poincare}  We are now ready to prove Theorem  \ref{thm:sufficient conditions for the Poincare inequality}. For simplicity, throughout this section we drop the subindices in $h_{f,\gamma}$ and $\astar_{f,\gamma}$.

\begin{proof}[Proof of Theorem  \ref{thm:sufficient conditions for the Poincare inequality}] We consider the cases $\gamma \in (-2,0]$ and $\gamma \in [-d,-2]$ separately.
	
\emph{Case $\gamma \in (-2,0]$}:  Proposition \ref{new_pro_global_space} says that for any $Q$ with sides less or equal $1$ we have
  \begin{align*}
  \sigma_{2,2}(Q,h,a^*) \leq C(d,\gamma,\fin)|Q|^{\frac{2+\gamma}{d}},
  \end{align*}
  and therefore  $\mu(8\ell) \leq  C(d,\gamma,\fin)\ell^{2+\gamma}$ for $\ell \in (0,1/8)$.
  
  Then, applying Lemma \ref{lem:Global weighted Sobolev from local ones} with $w_1 = h$ and $w_2 = \astar$,
  \begin{align*}
    \int_{\mathbb{R}^d}|\phi|^2 h\;dv \leq C(d,\gamma,\fin)\ell^{2+\gamma}\int_{\mathbb{R}^d} \astar|\nabla \phi|^2\;dv+C(2)\int_{\mathbb{R}^d}E_\ell(h)|\phi|^2\;dv. 
  \end{align*}
  We may use the estimate \eqref{hwithr=1} to bound $E_{\ell}(h)$, indeed, for any $Q\in \mathcal{Q}_\ell$ we have
  \begin{align*}
    \fint_{Q} h\;dv &\leq C(d,\gamma,\fin)\ell^\gamma \langle v_Q\rangle^\gamma ,
  \end{align*}
  where $v_Q$ denotes the center of $Q$. Since each $Q$ has sides $\leq 1$, it follows that $\langle v\rangle \approx \langle v_Q\rangle$ for all $v\in Q$, therefore
  \begin{align*}
    E_{\ell}(h) \leq C\ell^\gamma\langle v\rangle^\gamma	  .
  \end{align*}	  
  Taking $\ell = (\min\{C(d,\gamma,\fin)^{-1},2^{-|2+\gamma|}\}\varepsilon )^{\frac{1}{2+\gamma}}$ for $\varepsilon \in(0,1)$, it follows that for any $\phi$ with compact support we have
  \begin{align*}	
    \int_{\mathbb{R}^d} \phi^2 h\;dv \leq \varepsilon \int_{\mathbb{R}^d}a^*|\nabla \phi|^2\;dv + C\varepsilon^{\frac{\gamma}{2+\gamma}}\int_{\mathbb{R}^d} \phi^2\langle v\rangle^\gamma\;dv,\;\;C=C(d,\gamma,\fin).
  \end{align*}	
  \noindent \emph{Case $\gamma \in [-d,-2]:$} As before we apply Lemma \ref{lem:Global weighted Sobolev from local ones}. This time, however, for $\varepsilon \in (0,1)$ we let $\ell = \ell_\mu(\varepsilon)$ be given as the largest number in $(0,1/10]$ such that $C(2)\mu(8 \ell) \leq \varepsilon$. Observe that 
  \begin{align*}
    \int_{Q}h\;dv & \leq C(d,\gamma)\int_{\mathbb{R}^d}  f(w) \int_{Q}|v-w|^{\gamma} \; dvdw \leq  C(d,\gamma)|Q |^{\frac{d+\gamma}{d}}.
  \end{align*} 
  The above inequality allows us to bound $E_{\ell_\mu(\varepsilon)}(h)$ in this case, since
  \begin{align*}
    \fint_{Q} h\;dv\leq C(d,\gamma,\fin)\ell_\mu(\varepsilon)^\gamma,\;\;\forall\;Q\in\mathcal{Q}_{\ell_\mu(\varepsilon)}.
  \end{align*}  
  Then, back to Lemma \ref{lem:Global weighted Sobolev from local ones}, we conclude that for any compactly supported, Lipschitz $\phi$,
  \begin{align*}	
    \int_{\mathbb{R}^d} \phi^2 h\;dv \leq  \varepsilon\int_{\mathbb{R}^d}a|\nabla \phi|^2\;dv + C(d,\gamma,\fin)\ell_\mu(\varepsilon)^\gamma\int_{\mathbb{R}^d} \phi^2\;dv.	
  \end{align*}	
      
\end{proof}

  \subsection{Space-time inequalities}\label{Space time inequalities}

  \begin{lem}\label{lem:Sobolev weight astar astar to the m}
    Let $\gamma \in [-d,-2)$. If Assumption \ref{Assumption:Local Doubling} holds, then there is a $C=C(d,\gamma,\fin,C_D)$  such that for any Lipschitz $\phi$ with compact support we have
    \begin{align*}
      \left ( \int_{\mathbb{R}^d}  \phi^{2m}(\astar_{f,\gamma})^m \;dv\right )^{\frac{1}{m}}\leq C\int_{\mathbb{R}^d} a_{f,\gamma}^*|\nabla \phi|^2\;dv+C\int_{\mathbb{R}^d}\phi^2\;dv,		
    \end{align*}	        
    Here, $m = \frac{d}{d-2}$ if $\gamma>-d$, otherwise, $m$ can be chosen to be any $m<\frac{d}{d-2}$ when $\gamma=-d$.
  \end{lem}
  
  \begin{proof}
  Let us take first $\gamma \in (-d,-2)$. Then, let $m =d/(d-2)$. Such $m$ solves
  \begin{align}\label{eqn:exponent relation for d over d-2}
    \frac{1}{d}-\frac{1}{2}+\frac{1}{2m} = 0.	  
  \end{align}	
  With this $m$, and with $\gamma>-d$ there is always some $s=s(d,\gamma)>1$ such that $m s |\gamma+2|<d$. Then Lemma \ref{this_is_new_cond_Sob} with $q=ms$, $w_1 = {a^*}^{ms}$, and $w_2 = (a^*)^{-s}$ yields 
 \begin{align*}
  \left( \fint_Q (a_*)^{m s}\;dv \right)^{\frac{1}{ms}}\left( \fint_Q (a_*)^{-s}\;dv\right)^{\frac{1}{s}} \leq C(d,\gamma,\fin,C_D)
 \end{align*}
  for all $Q$ with side length smaller than $1/2$. Taking the square root of both sides, and using \eqref{eqn:exponent relation for d over d-2}, we conclude that for some $C=C(d,\gamma,\fin,C_D)$
  \begin{align*}
   |Q|^{\frac{1}{d}-\frac{1}{2}+\frac{1}{2m}}\left( \fint_Q (a_*)^{ms}\;dv \right)^{\frac{1}{2ms}}\left( \fint_Q (a_*)^{-s}\;dv\right)^{\frac{1}{2s}} \leq C.
  \end{align*}
  This again being for all cubes $Q$ with side length smaller than $1/2$. Then, we apply the covering argument of Lemma \ref{lem:Global weighted Sobolev from local ones} with $\ell= 1/16$, and obtain
  \begin{align*}
    \left ( \int_{\mathbb{R}^d}|\phi|^{2m}(\astar)^m\;dv\right )^{\frac{1}{m}} \leq C\int_{\mathbb{R}^d}\astar |\nabla \phi|^2\;dv+C\int_{\mathbb{R}^d}E_{\frac{1}{16}((\astar)^m)}|\phi|^2\;dv.
  \end{align*}
  It remains to bound $E_{\frac{1}{16}((\astar)^m)}$, for which it suffices to bound $\frac{1}{|Q|}\left (\int_{Q}(\astar)^m\;dv \right )^{\frac{1}{m}}$ for any cube $Q$ with side length $\leq 1/2$. On the other hand, for any cube $Q$ with side length less than $1$ we have
  \begin{align*}
    \frac{1}{|Q|}\left (\int_{Q}(\astar)^m\;dv \right )^{\frac{1}{m}} & = \frac{1}{|Q|^{1-\frac{1}{m}}}\left (\fint_{Q}(\astar)^m\;dv \right )^{\frac{1}{m}} \leq C\frac{1}{|Q|^{1-\frac{1}{m}}}\fint_{Q}\astar\;dv.
  \end{align*}
  In particular, if $Q$ has side length exactly equal to $1/16$, then
  \begin{align*}
    \frac{1}{|Q|}\left (\int_{Q}(\astar)^m\;dv \right )^{\frac{1}{m}} \leq C(d,\gamma,\fin,C_D).
  \end{align*}
  This implies that $\|E_\ell((\astar)^m)\|_{L^\infty(\mathbb{R}^d)}\leq C$, which finishes the proof for $\gamma \in (-d,-2)$. Now, for the case $\gamma=-d$, we let $m$ be any number in $(1,\frac{d}{d-2})$. In this case, we have
  \begin{align*}
    \frac{1}{d}-\frac{1}{2}+\frac{1}{2m} \geq 0.	  
  \end{align*}	
  and, we also have that there is some $s=s(d,m)>1$ such that $sm<\frac{d}{d-2}$. Then, in light of the above inequality involving $m$, we have for all cubes of side length $\leq 1$,
  \begin{align*}
   |Q|^{\frac{1}{d}-\frac{1}{2}+\frac{1}{2m}}\left( \fint_Q (a_*)^{ms}\;dv \right)^{\frac{1}{2ms}}\left( \fint_Q (a_*)^{-s}\;dv\right)^{\frac{1}{2s}} \leq C.
  \end{align*}
  From here on, we apply Lemma \ref{lem:Global weighted Sobolev from local ones} with $\ell=1/16$ and bound $\| E_{1/16}(\astar)^m \|_{L^\infty(\mathbb{R}^d)}$ just as before, finishing the proof.
  
  \end{proof}

  \begin{rem}\label{rem:Sobolev weight astar to the m case gamma bigger -2}
  For $\gamma\in [-2,0]$, without requiring any doubling assumption on $f$, we always have the inequality
  \begin{align*}
    \left ( \int_{\mathbb{R}^d}  \phi^{2m}(\astar_{f,\gamma})^m \;dv\right )^{\frac{1}{m}}\leq C\int_{\mathbb{R}^d} a_{f,\gamma}^*|\nabla \phi|^2\;dv+C\int_{\mathbb{R}^d}|\phi|^2\;dv,\;\; m = \frac{d}{d-2}.
  \end{align*}
  In fact, recall that $m$ as above satisfies \eqref{eqn:exponent relation for d over d-2}. Then using Lemma \ref{lem:averages for powers of bracket v} one can see there are $C=C(d,\gamma)>1$ and $s=s(d,\gamma)>1$  such that for any cube $Q$ with side length $\leq 1$, 
  \begin{align*}
    \left ( \fint_{Q}\langle v \rangle^{\gamma m s} \;dv \right )^{\frac{1}{2m s}}\left (\fint_{Q}\langle v\rangle^{-\gamma s} \;dv \right )^{\frac{1}{2 s}} \leq C.
  \end{align*}
  Then, Lemma \ref{lem:Global weighted Sobolev from local ones} with $q=2m$, $w_1 = \langle v\rangle^{m\gamma}$, $w_2 = \langle v\rangle^\gamma$, and $\ell = 1/2$ yields the basic weighted inequality,
  \begin{align*}
    \left ( \int_{\mathbb{R}^d}|\phi|^{\frac{2d}{d-2}}\langle v\rangle^{\gamma m}\;dv\right )^{\frac{d-2}{d}} \leq C_{d,\gamma}\int_{\mathbb{R}^d}\langle v\rangle^\gamma|\nabla \phi|^2\;dv+C_{d,\gamma}\int_{\mathbb{R}^d}|\phi|^2\;dv.	  
  \end{align*}
  Finally, by Proposition \ref{prop:a vs a star gamma larger than -2}, for $\gamma\in [-2,0]$ we have
  \begin{align*}
    \astar \approx \langle v\rangle^\gamma,
  \end{align*}
  substituting this pointwise bound in the weighted inequality above, the estimate follows.
\end{rem}

 With the above tools in hand, we prove now Theorems \ref{lem:Inequalities Sobolev weight aII} and \ref{thm:Inequalities Sobolev weight gamma below -2}. We combine both proofs since they are basically identical. 

  \begin{proof}[Proof of Theorem \ref{lem:Inequalities Sobolev weight aII} and Theorem \ref{thm:Inequalities Sobolev weight gamma below -2} ] 
 In the following $m = \frac{d}{d-2}$ if $\gamma>-d$, otherwise, $m$ can be chosen to be any $m<\frac{d}{d-2}$ when $\gamma=-d$. The first part on both theorems has been shown respectively in Remark \ref{rem:Sobolev weight astar to the m case gamma bigger -2} and Lemma \ref{lem:Sobolev weight astar astar to the m}. For the second part  let $\phi: \mathbb{R}^d \times I \mapsto \mathbb{R}$ and fix $t\in I$. Let $q = 4-\frac{2}{m}$, and note that $2<q<2m$. Interpolation yields 
  \begin{align*}
    \int_{\mathbb{R}^d}\phi^q \astar \;dv & =  \int_{\mathbb{R}^d} \phi^{q\theta + q(1-\theta)} \astar \;dv \\
      & \leq \left ( \int_{\mathbb{R}^d} \phi^{2m} (\astar)^{\frac{2m}{(1-\theta)q}} \;dv \right )^{\frac{(1-\theta)q}{2m}} \left (\int_{\mathbb{R}^d} \phi^{2}\;dv \right )^{\frac{q \theta}{2}},
  \end{align*}
  where $\frac{1}{q} = \frac{\theta}{2}+\frac{1-\theta}{2m}$. An elementary calculation shows that
  \begin{align*}
    q\theta = 2\left (1-\frac{1}{m} \right ),\;\;(1-\theta)q= 2. 
  \end{align*}	  
  Therefore, the last inequality turns out to be
  \begin{align*}
    \int_{\mathbb{R}^d} \phi^q \astar \;dv \leq \left (  \int_{\mathbb{R}^d} \phi^{2m} (\astar)^m\;dv\right )^{\frac{1}{m}} \left (\int_{\mathbb{R}^d} \phi^2\;dv \right )^{1-\frac{1}{m}}.
  \end{align*}
 This inequality holds for each $t\in I$, integrating in time we arrive at
  \begin{align*}
    \int_I\int_{\mathbb{R}^d}\phi^q \astar \;dvdt & \leq \sup \limits_{I} \left(\int_{\mathbb{R}^d} \phi^{2}\;dv \right )^{1-\frac{1}{m}} \int_{I}\left (  \int_{\mathbb{R}^d} \phi^{2m} (\astar)^{m} \;dv \right )^{\frac{1}{m}}\;dt.
   \end{align*}
   Then, by Lemma \ref{lem:Sobolev weight astar astar to the m} or Remark \ref{rem:Sobolev weight astar to the m case gamma bigger -2},
   \begin{align*}
    \int_{I}\int_{\mathbb{R}^d}\phi^q \astar \;dvdt & \leq C \left(\sup \limits_{I} \int_{\mathbb{R}^d} \phi^{2}\;dv \right )^{1-\frac{1}{m}}  \int_{I}\int_{\mathbb{R}^d}\astar |\nabla \phi|^2 +\phi^2\;dvdt.
   \end{align*}	
   From the above, it is elementary to see that
   \begin{align*}	
    \int_{I}\int_{\mathbb{R}^d}\phi^q \astar \;dvdt  & \leq C'\left (  \int_{I}\int_{\mathbb{R}^d}\astar |\nabla \phi|^2\;dvdt +\sup \limits_{I}\int_{\mathbb{R}^d}\phi^2\;dv\right )^{2-\frac{1}{m}},
    \end{align*}
    then, recalling that $q = 4-\frac{2}{m}$, this inequality becomes
    \begin{align*}
    \int_{I}  \int\phi^q \astar \;dvdt & \leq C'\left ( \int_{I} \int \astar |\nabla \phi|^2 \;dvdt + \sup \limits_{I} \int \phi^{2}\;dv \right )^{\frac{q}{2}},
    \end{align*}  	
    Finally, we note that if $m=\frac{d}{d-2}$ then $q=2\left (1+\frac{2}{d}\right )$.	
  \end{proof}

\section{$L^\infty$-Estimates}\label{section: Regularization in L infinity}

  \subsection{From entropy production to $L^pL^p$ estimates}
  
  For the purposes of $L^\infty$ estimates, it is important to have a weighted $L^pL^p$ bound for $f$ for some $p>1$ and a proper weight. With this aim, we recall the standard and well known fact that through the entropy production one can obtain a space-time integral bound for a derivative of $f$ (see, for instance \cite{Desvillettes14}). Also in what follows, recall that the notion of weak solution used here is described in Definition \ref{def:solution}.

  Let $f$ be a solution to \eqref{eqn:Landau homogeneous}, this being understood as a weak solution if $\gamma \in (-2,0]$ or as a classical solution in the case $\gamma \in [-d,-2]$. Moreover, let $\rho_{in}$ denote the Maxwellian with same mass, center of mass, and energy as $\fin$. We have
  \begin{align}\label{eqn:entropy production bound}
    \int_{t_1}^{t_2}\int_{\mathbb{R}^d}4(A_{f,\gamma}\nabla f^{1/2},\nabla f^{1/2})-f^2\;dvdt \leq H(\fin) - H(\rho_{\fin}) = C(\fin),
  \end{align}	
  and $C(\fin)$ only depends on $\fin$ through it's mass, second moment, and entropy. Moreover the estimate \eqref{eqn:entropy production bound} implies 
  \begin{align}\label{eqn:entropy production bound_II}
  \| f \langle v \rangle^{\gamma}\|_{L^1(0,T,L^{d/d-2}(\mathbb{R}^d))} \le C(T,\fin),
  \end{align}
(see Theorem 3 and Lemma 3 in \cite{Desvillettes14}).

   The inequality \eqref{eqn:entropy production bound_II}, combined with the $\varepsilon$-Poincar\'e inequality and the conserved quantities yields a weighted $L^pL^p$ bound. This is the content of the following proposition.

\begin{prop}\label{prop:LpLp via interpolation with weight}  
  If $\gamma \in (-2,0]$, and $f$ is a weak solution, then for any pair of times $T_1,T_2 \in (0,T)$ with $0\leq T_1<T_2\leq T_1+1$ we have
  \begin{align*}
    \int_{T_1}^{T_2}\int_{\mathbb{R}^d} f^{1+\frac{2}{d}}a^*_{f,\gamma}\;dvdt \leq C(d,\gamma,\fin).
  \end{align*}	  
  Meanwhile, if $\gamma\in (-d,-2]$, $f$ is a classical solution, and Assumptions \ref{Assumption:Epsilon Poincare} and \ref{Assumption:Local Doubling} hold, then for $T_1$ and $T_2$ as before
  \begin{align*}
    \int_{T_1}^{T_2}\int_{\mathbb{R}^d} f^{1+\frac{2}{d}}a^*_{f,\gamma}\;dvdt \leq C(d,\gamma,\fin,C_D,C_P,\kappa_P).
  \end{align*}	  
  The constants $C_D,C_P$, and $\kappa_P$  being as in Assumptions \ref{Assumption:Epsilon Poincare} and \ref{Assumption:Local Doubling}.
\end{prop}
  
\begin{proof}  
  Let $p$ denote $1+\frac{2}{d}$. Since  $p \in (1,\frac{d}{d-2})$ there is a $\theta \in (0,1)$ such that
  \begin{align*} 
    1 = \theta p + \frac{(1-\theta)p}{m},\;\; m = \frac{d}{d-2}.    
  \end{align*}	
  Note that $\theta$ entirely determines $p$, and vice versa. H\"older's inequality then yields,
  \begin{align*}
    \int_{\mathbb{R}^d}  f^p a^*\;dv & \leq \left ( \int_{\mathbb{R}^d} f \;dv \right )^{\theta p}	\left (\int_{\mathbb{R}^d} f^{m} (a^*)^{\frac{m}{(1-\theta)p}}\;dv \right )^{\frac{(1-\theta)p}{m}} = \left (\int_{\mathbb{R}^d} f^{m} (a^*)^{\frac{m}{(1-\theta)p}}\;dv \right )^{\frac{(1-\theta)p}{m}}.	  
  \end{align*}
  The fact that $p=1+\frac{2}{d}$ guarantees that $(1-\theta) p =1$, and therefore (since $\|f\|_{L^1(\mathbb{R}^d)}\equiv1$)
  \begin{align*}
    \int_{\mathbb{R}^d}  f^p a^*\;dv \leq \left ( \int_{\mathbb{R}^d} f^{m} (a^*)^{m}\;dv \right )^{\frac{1}{m}}.	  
  \end{align*}	  
  Integrating in $(T_1,T_2)$, it follows that
  \begin{align*}
    \int_{T_1}^{T_2}\int_{\mathbb{R}^d} f^pa^*\;dvdt \leq C \int_{T_1}^{T_2}\left ( \int_{\mathbb{R}^d} f^{m}(a^*)^{m} \;dv \right )^{\frac{1}{m}}\;dt.	  
  \end{align*}	  
  In this case, the weighted Sobolev inequality in Theorem \ref{lem:Inequalities Sobolev weight aII} (if $\gamma\in [-2,0]$) or in Theorem \ref{thm:Inequalities Sobolev weight gamma below -2} (if $\gamma\in [-d,-2)$) guarantees that
\begin{align*}
  & \left( \int_{\mathbb{R}^d} f^{m} (a^*)^{m} \;dv \right )^{\frac{1}{m}} \leq C\int_{\mathbb{R}^d} a^*|\nabla f^{\frac{1}{2}}|^2\;dv+C \int_{\mathbb{R}^d}f\;dv.
\end{align*}
  Then, 
 \begin{align*}
    \int_{T_1}^{T_2}\int_{\mathbb{R}^d}  f^p a^*\;dvdt \leq C \int_{T_1}^{T_2}\int_{\mathbb{R}^d} a^*|\nabla f^{\frac{1}{2}}|^2\;dvdt+C(T_2-T_1).
  \end{align*}
  It remains to bound the right hand side. For this we make use of the bound, 
  \begin{align*}
    4\int_{T_1}^{T_2}\int_{\mathbb{R}^d}(A_{f,\gamma}\nabla f^{1/2},\nabla f^{1/2})\;dvdt-\int_{T_1}^{T_2}\int_{\mathbb{R}^d}hf\;dvdt \leq H(\fin)-H(\rho_{\fin}).
  \end{align*}
  Now, the $\varepsilon$-Poincar\'e inequality applied with $\phi = \sqrt{f}$ and $\varepsilon = 1/8$ says that
  \begin{align*}
    \int_{T_1}^{T_2}\int_{\mathbb{R}^d}f h_{f,\gamma}\;dvdt \leq 2\int_{T_1}^{T_2}\int_{\mathbb{R}^d}(A_f\nabla f^{1/2},\nabla f^{1/2})\;dvdt+\Lambda_f(1/8)\int_{T_1}^{T_2}\int_{\mathbb{R}^d}f\;dvdt.
  \end{align*}
  In which case, we have
  \begin{align*}
    2\int_{T_1}^{T_2}\int_{\mathbb{R}^d} (A_{f,\gamma}\nabla f^{\frac{1}{2}},\nabla f^{\frac{1}{2}}) \;dvdt 
	  & \leq H(\fin)-H(\rho_{\fin})+\Lambda_f(1/8).
  \end{align*}
\end{proof}  

We shall make use of one further $L^pL^p$ bound, but this time there is no weight.

 \begin{prop}\label{prop: LpLp bound via interpolation1}
  Let $f$ be a solution of \eqref{eqn:Landau homogeneous}. Then, 
  \begin{align*}
    \|f\|_{L^{p_{\gamma}}(0,T,L^{p_\gamma}(\mathbb{R}^d))} & \leq C(\fin,d,\gamma,T).
  \end{align*}	
  where 
  \begin{align*}
    p_\gamma := \min \left \{\frac{d(2-\gamma)}{2(d-2)-d\gamma}\;, 1+\frac{2}{d}  \right \},
  \end{align*}
  and $C$ does not depend on $T$ if $\gamma \in [-\tfrac{4}{d},0]$ (observe this corresponds to $p_\gamma = 1+\frac{2}{d}$).
  \end{prop}
  
  \begin{proof}
  The proof is analogous to the previous one. As before, with $p=p_\gamma$ we have
  \begin{align*}
      \int_{\mathbb{R}^d} f^p\;dv  & \le \left(\int_{\mathbb{R}^d} \left(f^{p(1-\theta)} \langle v \rangle^{-m}\right)^q dv\right)^{1/q} \left(\int_{\mathbb{R}^d} \left(f^{p\theta} \langle v \rangle^{m}\right)^{\frac{1}{p\theta}} dv\right)^{p\theta},
\end{align*}
  with $ q = \frac{1}{(1-p)\theta}$. The aim is to estimate $ \int f^p\;dv $ with the quantities $ \int f\langle v \rangle^{2}\;dv$ and $ \int f^{\frac{d}{d-2}}\langle v \rangle^{\frac{d}{d-2}\gamma}\;dv$. First, suppose that $p_\gamma = d(2-\gamma)/ (2(d-2)-d\gamma)$, then we choose
  \begin{align*}
      \frac{m}{p\theta} = 2, \quad {p(1-\theta)q}{} = \frac{d}{d-2}, \quad  mq = -\gamma\frac{d}{d-2}.
  \end{align*}
  and it follows that for some $C=C(d,\gamma)$ we have
  \begin{align*}
      \int_{\mathbb{R}^d} f^{p_\gamma}\;dv  & \leq C\left(\int_{\mathbb{R}^d} f^{\frac{d}{d-2}}\langle v \rangle^{{\frac{d}{d-2}}\gamma}\;dv\right)^{\frac{2(d-2)}{2(d-2)-d\gamma}}\left(\int_{\mathbb{R}^d} f\langle v \rangle^{2}\;dv\right)^{\frac{d\gamma}{d\gamma-2(d-2)}}.
  \end{align*}
  Now we use the above inequality for $f$ and integrate for $t\in (0,T)$: 
  \begin{align*}
    \int_0^T\left(\int_{\mathbb{R}^d} f^p \;dv\right)\;dt & \leq \int_0^T \|f \langle v\rangle^\gamma\|_{L^{\frac{d}{d-2}}}^{\frac{2d}{2(d-2)-d\gamma}} \|f \langle v\rangle^2\|_{L^1}^{\frac{d\gamma}{d\gamma-2(d-2)}}	\;dt \\
    & \leq   \sup_{(0,T)} \|f \langle v\rangle^2\|_{L^1}^{\frac{d\gamma}{d\gamma-2(d-2)}}\int_0^T \|f \langle v\rangle^\gamma\|_{L^{\frac{d}{d-2}}}^{\frac{2d}{2(d-2)-d\gamma}}\;dt \\
    & \leq C(\fin) T^{1- {\frac{2d}{2(d-2)-d\gamma}}}\left(\int_0^T \|f \langle v\rangle^\gamma\|_{L^{\frac{d}{d-2}}}\;dt\right)^{\frac{2d}{2(d-2)-d\gamma}},
  \end{align*}	
  thanks to \eqref{eqn:entropy production bound_II} and the conservation of the second momentum of $f$. In the last inequality we applied Jensen's inequality since $\frac{2d}{2(d-2)-d\gamma} <1$ if $\gamma <-\frac{4}{d}$ (which is the case when $p_\gamma>1+\frac{2}{d}$).
    
  Second, suppose that  $p_\gamma = 1+\frac{2}{d}$, then, we take
  \begin{align*}
    q = \frac{d}{d-2}, \; p = 1+\frac{2}{d},\; \theta = \frac{2}{2+d},
  \end{align*}
  and $m$ such that 
  \begin{align*}
    mq \ge -\frac{\gamma d}{d-2},\quad \frac{m}{p\theta} \leq 2
  \end{align*}	
  and obtain 
  \begin{align*}
    \int_0^T  \int_{\mathbb{R}^d} f^p\;dvdt  & \le  \int_0^T\left(\int_{\mathbb{R}^3} f^{\frac{d}{d-2}}\langle v \rangle^{{\frac{\gamma d}{d-2}}}\;dv\right)^{\frac{d-2}{d}}\left(\int_{\mathbb{R}^d} f\langle v \rangle^{2}\;dv\right)^{\frac{2}{d}}\;dt  \leq  C(\fin,d,\gamma),
  \end{align*}
 using again \eqref{eqn:entropy production bound_II} and conservation of second moment.    
    
\end{proof}

\subsection{Energy inequality}\label{sec: energy ineq}
  At the center of the iteration argument are integrals of functions of the distribution $f$. That is, integrals of the form
  \begin{align}\label{eqn:energy_inequality basic integrals}
    \int_{\mathbb{R}^d} \eta^2 \phi(f) \;dv,
  \end{align}
  with $\eta$ a generic $C^2_c(\mathbb{R}^d)$-function and $\phi:\mathbb{R}\mapsto\mathbb{R}$ a $C^\infty$ function such that
  \begin{align}\label{eqn:properties of phi}
    \phi''(s) \geq 0 \;\;\forall\;s \textnormal{ and } \phi''(s) = 0 \textnormal{ for all large } s, \textnormal{ and } \lim\limits_{s\to 0} s^{-1}\phi(s) = 0.
  \end{align}
  The integrals \eqref{eqn:energy_inequality basic integrals} will be understood using an energy inequality derived from the equation, understood in the sense of Definition \ref{def:solution}. We need some preliminaries on the type of functions $\phi$ that will be used, and certain functions to them that will appear in the computations below.
  \begin{rem}\label{rem:definition of overline phi and underline phi}
    Let $\phi:\mathbb{R}\mapsto\mathbb{R}$ be as in \eqref{eqn:properties of phi}, for $u\geq 0$ define functions $\overline{\phi}$ and $\underline{\phi}$ by
    \begin{align}\label{eqn:overline phi and underline phi}
      \overline{\phi}(u) := \int_0^u (\phi''(s))^{\frac{1}{2}}\;ds,\;\;\underline{\phi}(u) = \int_0^u s\phi''(s)\;ds.
    \end{align}
    Observe that for every $s$, we have
    \begin{align*}
      (s\phi'(s)-\underline{\phi}(s)-\phi(s))' = \phi'(s)+s\phi''(s)-s\phi''(s)-\phi'(s) = 0,\;\;\forall\;s,
    \end{align*}	  
    thus $s\phi'(s)-\underline{\phi}(s)-\phi(s) = -\underline{\phi}(0)-\phi(0) = 0$ for all $s$, which means that
    \begin{align*}
      s\phi'(s)-\underline{\phi}(s) = \phi(s).
    \end{align*}	   
  \end{rem}

  In all that follows, we drop the subindices $f,\gamma$ from $A_{f,\gamma}$, $a_{f,\gamma}$, $h_{f,\gamma}$ and $a_{f,\gamma}^*$. The first Proposition aims to provide a more workable expression for the right hand side of the integral formulation for the equation (Definition \ref{def:solution}).

  \begin{prop}\label{prop:basic energy inequality}
    Let $f$ be a weak solution (if $\gamma\in [-2,0]$) or a classical solution (if $\gamma \in [-d,-2)$), $\eta \in C^2_c(\mathbb{R}^d)$ and $\phi$ as \eqref{eqn:properties of phi}. Then, we have the identity
    \begin{align*}
      -\int_{t_1}^{t_2}\int_{\mathbb{R}^d}(A\nabla f-f\nabla a,\nabla (\eta^2\phi'(f)))\;dvdt = (\textnormal{I})+(\textnormal{II})+(\textnormal{III})+(\textnormal{IV}),
    \end{align*}
    where, with $\overline{\phi}$ and $\underline{\phi}$ as in \eqref{eqn:overline phi and underline phi}, we have
    \begin{align*}
      (\textnormal{I}) & = -\int_{t_1}^{t_2}\int_{\mathbb{R}^d}(A\nabla (\eta \overline{\phi}(f)),\nabla (\eta \overline{\phi})(f))\;dvdt+2\int_{t_1}^{t_2}\int_{\mathbb{R}^d}\overline{\phi}(f)(A\nabla (\eta \overline{\phi}(f)),\nabla \eta)\;dvdt,\\
      (\textnormal{II}) & = -\int_{t_1}^{t_2}\int_{\mathbb{R}^d}\overline{\phi}(f)^2(A\nabla \eta,\nabla \eta)\;dvdt+\int_{t_1}^{t_2}\int_{\mathbb{R}^d}\phi(f)\tr(AD^2(\eta^2))\;dvdt,\\
      (\textnormal{III}) & = \int_{t_1}^{t_2}\int_{\mathbb{R}^d} \eta^2\underline{\phi}(f) h\;dvdt,\\
      (\textnormal{IV}) & = 2\int_{t_1}^{t_2}\int_{\mathbb{R}^d}\phi(f)(\nabla a,\nabla \eta^2)\;dvdt.	
    \end{align*}
  
  \end{prop}

  \begin{proof}
    The proof amounts to keeping track of various elementary pointwise identities and a couple of integration by parts. First, it is clear that
    \begin{align*}
      & \int_{\mathbb{R}^d}(A\nabla f,\nabla (\eta^2 \phi'(f)))\;dv = \int_{\mathbb{R}^d}\eta^2 \phi''(f)(A\nabla f,\nabla f)+\phi'(f)(A\nabla f,\nabla \eta^2)\;dv. 		  
    \end{align*}
    A couple of elementary identities: from the definition of $\overline{\phi}$ it follows that
    \begin{align*}
      & \eta^2 \phi''(f)(A\nabla f,\nabla f) \\ 
      & = (A\nabla (\eta \overline{\phi}(f)),\nabla (\eta \overline{\phi}(f))) -2\eta \overline{\phi}(f)(A\nabla \eta,\nabla \overline{\phi}(f))-\overline{\phi}(f)^2(A\nabla \eta,\nabla \eta),
    \end{align*}	  
    at the same time,
    \begin{align*}
      -2\eta \overline{\phi}(f)(A\nabla \eta,\nabla \overline{\phi}(f)) = -2 \overline{\phi}(f)(A\nabla \eta,\nabla (\eta \overline{\phi}(f)) )+2\overline{\phi}(f)^2(A\nabla \eta,\nabla \eta).
    \end{align*}
    These two formulas yield
    \begin{align*}
      & \int_{\mathbb{R}^d}\eta^2 \phi''(f)(A\nabla f,\nabla f) \;dv\\
      & = \int_{\mathbb{R}^d} (A\nabla (\eta \overline{\phi}(f)),\nabla (\eta \overline{\phi}(f))) -2 \overline{\phi}(f)(A\nabla \eta,\nabla (\eta \overline{\phi}(f)))+\overline{\phi}(f)^2(A\nabla \eta,\nabla \eta) \;dv.		  
    \end{align*}
    On the other hand, the identity $\phi'(f)\nabla f = \nabla \phi(f)$ and integration by parts yield
    \begin{align*}
     \int_{\mathbb{R}^d}\phi'(f)(A\nabla f,\nabla \eta^2)\;dv & = -\int_{\mathbb{R}^d}\phi(f)(\tr(AD^2(\eta^2))+(\nabla a,\nabla (\eta^2)))\;dv.
    \end{align*}
    Integrating these identities in $(t_1,t_2)$ (and recalling definition of $\textnormal{(I)}$ and $\textnormal{(II)}$ leads to	
    \begin{align}\label{eqn:energy identity first half}
      & -\int_{t_1}^{t_2}\int_{\mathbb{R}^d}(A\nabla f,\nabla (\eta^2 \phi'(f)))\;dvdt = \textnormal{(I)}+\textnormal{(II)}+\int_{t_1}^{t_2}\int_{\mathbb{R}^d}\phi(f)(\nabla a,\nabla (\eta^2))\;dvdt.	  
    \end{align}	
    To deal with the other term in the integral, we make use of the identity
    \begin{align*}
      f \nabla (\eta^2 \phi'(f)) & = \eta^2 f\phi''(f)\nabla f+f\phi'(f)\nabla (\eta^2)\\
  	& = \eta^2 \nabla \underline{\phi}(f)+f\phi'(f)\nabla (\eta^2)\\
  	& = \nabla (\eta^2 \underline{\phi}(f) )- \underline{\phi}(f)\nabla (\eta^2)+f\phi'(f)\nabla (\eta^2).	  
    \end{align*}	  
    Integrating, this leads to
    \begin{align*}
      \int_{\mathbb{R}^d}(f\nabla a,\nabla (\eta^2 \phi'(f)) )\;dv = \int_{\mathbb{R}^d} \eta^2\underline{\phi}(f) h\;dvdt+\int_{\mathbb{R}^d}(f\phi'(f)-\underline{\phi}(f))(\nabla a,\nabla (\eta^2))\;dv.		
    \end{align*}
    According to Remark \ref{rem:definition of overline phi and underline phi}, $f\phi'(f)-\underline{\phi}(f) = \phi(f)$, therefore
    \begin{align*}
      \int_{t_1}^{t_2}\int_{\mathbb{R}^d}(f\nabla a,\nabla (\eta^2 \phi'(f)) )\;dvdt = \textnormal{(III)}+\int_{t_1}^{t_2}\int_{\mathbb{R}^d}\phi(f)(\nabla a,\nabla (\eta^2))\;dvdt.		
    \end{align*}
    This identity combined with \eqref{eqn:energy identity first half} proves the Proposition.
  \end{proof}

  The following remark deals with an approximation to the power function $\phi(s) = s^p/p$ via functions $\phi$ which are sublinear as $s\to\infty$.
  \begin{DEF}\label{def:Definitions of approximations to p-th power I}
    Fix $p>1$ and $h>0$. Let $\chi:\mathbb{R}\mapsto\mathbb{R}$ be a $C^\infty$ function such that 
    \begin{align*}
      & 0\leq \chi(s) \leq 1 \textnormal{ and } 0 \leq -\chi'(s)\leq 2 \;\textnormal{ for all } \;s,\\
      &\chi(s) =1 \textnormal{ if } s<0,\;\chi(s) = 0 \textnormal{ if } s>1.
    \end{align*}
    Then, for $u\geq 0$ we define
    \begin{align}\label{eqn:def of approximations to the pth power}
      \chi_h(u) := \int_{0}^u\chi(s-h)\;ds,\;\;\; \phi_{p,h}(u) = \phi(u) :=  \int_0^u\chi_h(s)^{p-1}\;ds.
    \end{align}
  \end{DEF}

  One should think of $\chi_h(u)$ as a smooth approximation to $\min\{u,h+1\}$.  The following Lemma contains useful and straightforward inequalities regarding $\phi_{p,h}$ and the associated functions $\overline{\phi}_{p,h}$ and $\underline{\phi}_{p,h}$. 

  \begin{lem}\label{lem:properties of approximations to p-th power}
    Let $p>1$ and let $h(p)$ be chosen such that 	  
    \begin{align*}
      h\geq h(p) \Rightarrow p h^{-1}(1+h^{-1})^{p-1} \leq 1.
    \end{align*}   
    Then, for $h\geq h(p)$ the function $\phi = \phi_{p,h}$ satisfies the following inequalities for all $u\geq 0$
    \begin{align}
      & \underline{\phi}(u)  \leq \frac{p}{2}\overline{\phi}(u)^2,\;\; \overline{\phi}(u)^2 \leq \frac{8(p-1)}{p} \phi(u), \label{eqn:upper bounds for under and over bar phi}\\
      & |\phi(u)- \frac{p}{4(p-1)}(\overline{\phi}(u))^2| \leq \frac{p+4}{4} h^{p-1}(u-h)_+. \label{eqn:bound difference phi and overline phi}
    \end{align}
	
  \end{lem}

  \begin{proof}
    We fix $p$ and $h$ as in the statement of the Lemma and take $\phi = \phi_{p,h}$ throughout the proof. Note that,
    \begin{align*}
      & \chi_h(u) = u \textnormal{ if } u<h,\\
      & \chi_h(u) \leq h+1 \;\;\forall\;u \geq 0,\\
      &\chi_h(u) \equiv \chi_h(h+1)\;\; \forall\;u\geq h+1.		
    \end{align*}
    Moreover, we have the identities
  \begin{align*}
    \phi'(u) = \chi_h(u)^{p-1},\;\;\phi''(u) = (p-1)\chi_h(u)^{p-2}\chi(u-h).
  \end{align*}
  Now, let us see what the corresponding functions $\overline{\phi}$ and $\underline{\phi}$ from \eqref{eqn:overline phi and underline phi} look like. If $u\leq h$, we have
  \begin{align*}
    \overline{\phi}(u) & = (p-1)^{\frac{1}{2}}\int_{0}^u \chi_h(s)^{\frac{p}{2}-1}\chi(s-h)\;ds  = (p-1)^{\frac{1}{2}}\int_{0}^u s^{\frac{p}{2}-1}\;ds,\\	
    \underline{\phi}(u) & = (p-1)\int_{0}^u s \chi_h(s)^{p-2}\chi(s-h)\;ds = (p-1)\int_{0}^u s^{p-1}\;ds.	
  \end{align*}	
  Then, for $u\leq h$,
  \begin{align*}
    \overline{\phi}(u) = \tfrac{2}{p}(p-1)^{\frac{1}{2}}u^{\frac{p}{2}},\;\; \underline{\phi}(u) = \tfrac{1}{p}(p-1)u^p \Rightarrow \underline{\phi}(u) = \frac{p}{4}\overline{\phi}(u)^2,\;\;\forall\;u\leq h.	
  \end{align*}	
  While, for $u\geq h$,
  \begin{align*}
    \underline{\phi}(u) & \leq \tfrac{p}{4}\overline{\phi}(h)^2+ (p-1)\int_h^{h+1}(h+1)^{p-1}\;ds\\
      & \leq \tfrac{p}{4}\overline{\phi}(h)^2+ \frac{p^2}{4} \overline{\phi}(h)^2h^{-1} (1+h^{-1})^{p-1}\\	
    \overline{\phi}(u) & \leq (p-1)^{\frac{1}{2}}(h+1)^{\frac{p}{2}},
  \end{align*}
  and, using that $\overline{\phi}(u)^2 \geq \overline{\phi}(h)^2 = \tfrac{4}{p^2}(p-1)h^p$ for $u\geq h$ we conclude that
  \begin{align*}
    \underline{\phi}(u) \leq \frac{p}{2}\overline{\phi}(u)^2,\;\; \overline{\phi}(u)^2 \leq \frac{8(p-1)}{p} \phi(u)\;\;\forall\;u\geq 0,
  \end{align*}
  provided that $h$ is larger than $h(p)$, where $h(p)$ is chosen such that by
  \begin{align*}
    h\geq h(p) \Rightarrow p h^{-1}(1+h^{-1})^{p-1} \leq 1. 
  \end{align*} 
  It remains to prove the second inequality. From the previous observations, it follows that
  \begin{align*}
    \phi(u) - \frac{p}{4(p-1)}(\overline{\phi}(u))^2 = 0 \textnormal{ for } u \in [0,h].
  \end{align*}
  For $u\geq h$, we shall differentiate the difference with respect to $u$, so that
  \begin{align*}
    & \phi'(u) - \frac{p}{4(p-1)}\overline{\phi}(u)\overline{\phi}'(u)\\
    & = \chi_h(u)^{p-1}-\frac{p}{4(p-1)}\overline{\phi}(u)(p-1)^{\frac{1}{2}}\chi_{h}(u)^{\frac{p}{2}-1}\left (\chi(u-h) \right )^{\frac{1}{2}}.
  \end{align*}
  Therefore, 
  \begin{align*}
    |\phi'(u) - \frac{p}{4(p-1)}\overline{\phi}(u)\overline{\phi}'(u)| \leq \frac{p+4}{4}(h+1)^{p-1}
  \end{align*}

  \begin{align*}
    & |\phi(u)- \frac{p}{4(p-1)}\overline{\phi}(u)^2|\\
    & = \left |\phi(u)-\frac{p}{4(p-1)}\overline{\phi}(u)^2-\left ( \phi(h)- \frac{p}{4(p-1)}\overline{\phi}(h)^2 \right ) \right |\\
    & \leq \frac{p+4}{4}(h+1)^{p-1}|u-h|.
  \end{align*}
  In conclusion, 
  \begin{align*}
    |\phi(u)- \frac{p}{4(p-1)}(\overline{\phi}(u))^2| \leq \frac{p+4}{4} h^{p-1}(u-h)_+,
  \end{align*}
  and the second inequality is proved.
  \end{proof}
  
  \begin{prop}\label{prop:energy inequality term IV}
    Let $\phi=\phi_{p,h}$ as in Definition \ref{def:Definitions of approximations to p-th power I}. Then, using the notation from the previous proposition, we have
    \begin{align*}
      (\textnormal{IV}) & = (\textnormal{IV})_1+(\textnormal{IV})_2,
    \end{align*}
    where, with $c_p=\frac{p}{4(p-1)}$,
    \begin{align*}
      (\textnormal{IV})_1 & = 2c_p\int_{t_1}^{t_2}\int_{\mathbb{R}^d}\overline{\phi}(f)^2 (4(A\nabla \eta,\nabla \eta)-\tr(AD^2\eta^2))\;dvdt\\
      &\;\;\;\;-8c_p\int_{t_1}^{t_2}\int_{\mathbb{R}^d}\overline{\phi}(f) (A\nabla (\eta \overline{\phi}(f)),\nabla \eta)\;dvdt,\\
      |(\textnormal{IV})_2| & \leq \frac{p+4}{2}\int_{t_1}^{t_2}\int_{\{f >h\} }(h+1)^{p-1}(f-h)|\nabla a| |\nabla \eta^2|\;dvdt.		  
    \end{align*}	  
  \end{prop}

  \begin{proof}
   Firstly, let us write
  \begin{align*}
    (\textnormal{IV}) = 2c_p\int_{\mathbb{R}^d}\overline{\phi}(f)^2(\nabla a,\nabla \eta^2)\;dv+2\int_{\mathbb{R}^d}(\phi(f)-c_p\overline{\phi}(f)^2)(\nabla a,\nabla \eta^2)\;dv,
  \end{align*}
  and set
  \begin{align*}
    (\textnormal{IV})_1 = 2c_p\int_{\mathbb{R}^d}\overline{\phi}(f)^2(\nabla a,\nabla \eta^2)\;dv,\;\;(\textnormal{IV})_2  = 2\int_{\mathbb{R}^d}(\phi(f)-c_p\overline{\phi}(f)^2)(\nabla a,\nabla \eta^2)\;dv.   
  \end{align*}
  Then, for $(\textnormal{IV})_1$, we integrate by parts, obtaining
  \begin{align*}
    (\textnormal{IV})_1 & = 2c_p\int_{\mathbb{R}^d}\dive(A \overline{\phi}(f)^2 (\nabla \eta^2 ) ) - \tr(AD(\overline{\phi}(f)^2\nabla \eta^2)) \;dv	\\
    & = -2c_p\int_{\mathbb{R}^d}\overline{\phi}(f)^2 \tr(AD^2\eta^2)\;dv-2c_p\int_{\mathbb{R}^d}(A\nabla \overline{\phi}(f)^2,\nabla \eta^2)\;dv.
  \end{align*}	
  Furthermore, using the elementary identity,
  \begin{align*}
    (A\nabla \overline{\phi}(f)^2,\nabla \eta^2 ) & = 4\overline{\phi}(f)\eta (A\nabla \overline{\phi}(f),\nabla \eta)\\
    & = 4\overline{\phi}(f) (A\nabla (\eta \overline{\phi}(f)),\nabla \eta)-4\overline{\phi}(f)^2  (A\nabla \eta,\nabla \eta),
  \end{align*}
  it follows that
  \begin{align*}
    (\textnormal{IV})_1 & = -2c_p\int_{\mathbb{R}^d}\overline{\phi}(f)^2 \tr(AD^2\eta^2)\;dv\\
    &\;\;\;\;-8c_p\int_{\mathbb{R}^d}\overline{\phi}(f) (A\nabla (\eta \overline{\phi}(f)),\nabla \eta)\;dv+8c_p\int_{\mathbb{R}^d}\overline{\phi}(f)^2  (A\nabla \eta,\nabla \eta)\;dv.
  \end{align*}  
  As for $(\textnormal{IV})_2$, one has
  \begin{align*}
    |(\textnormal{IV})_2|\leq 2\int_{t_1}^{t_2}\int_{\mathbb{R}^d}|\phi(f)-c_p\overline{\phi}(f)^2||\nabla a||\nabla \eta^2|\;dvdt.
  \end{align*}
  Then, using the pointwise inequality \eqref{eqn:bound difference phi and overline phi} from Lemma \ref{lem:properties of approximations to p-th power}, we obtain
  \begin{align*}
    |(\textnormal{IV})_2|\leq \frac{p+4}{2}\int_{t_1}^{t_2}\int_{\{f>h\}} (h+1)^{p-1}(f-h)|\nabla a||\nabla \eta^2|\;dvdt.
  \end{align*}
  \end{proof}

  In the following Lemma we intend to use Propositions \ref{prop:basic energy inequality} and \ref{prop:energy inequality term IV} with $\phi = \phi_{p,h}$ as above to prove an energy inequality. The $\varepsilon$-Poincar\'e inequality is used as well.
  \begin{lem}\label{lem:energy inequality with bad term} 
    Let $f$ be a weak solution in $(0,T)$ (if $\gamma \in (-2,0]$) or else a classical solution satisfying Assumption \ref{Assumption:Epsilon Poincare} if $\gamma \in [-d,-2]$. Let, $T_1,T_2,T_3$ be times such that $0<T_1<T_2<T_3<T$ and $\phi=\phi_{p,h}$ be as in Definition \ref{def:Definitions of approximations to p-th power I} with $p\geq p_0$ for some $p_0>1$. Then, we have that the quantity
    \begin{align*}
      \sup \limits_{(T_2,T_3)}\int_{\mathbb{R}^d}\eta^2 \phi(f(t))\;dv+\frac{1}{4}\int_{T_2}^{T_3}\int_{\mathbb{R}^d}(A\nabla (\eta \overline{\phi}(f)),\nabla (\eta \overline{\phi}(f) )\;dvdt,
    \end{align*}
    is no larger than
    \begin{align*}
      & \frac{1}{T_2-T_1}\int_{t_1}^{t_2}\int_{\mathbb{R}^d}\eta^2 \phi(f)(t)\;dvdt +C(p_0)\left (\|\nabla \eta\|_\infty^2+\|D^2\eta^2\|_\infty \right )\int_{T_1}^{T_3}\int_{\textnormal{spt}(\eta)}\phi(f)a\;dvdt\\
      & +\tfrac{p}{2}\Lambda_f\left (\tfrac{1}{2p}\right)\int_{T_1}^{T_3}\int_{\mathbb{R}^d}(\eta \overline{\phi}(f))^2\;dvdt+\frac{p+4}{2}\int_{T_1}^{T_3}\int_{\{f>h\}} (h+1)^{p-1}(f-h)|\nabla a||\nabla \eta^2|\;dvdt.	
    \end{align*}
  \end{lem}

  \begin{proof}
    Fix times $t_1,t_2$. From the definition of weak solution, and Proposition \ref{prop:basic energy inequality},
    \begin{align}\label{eqn:energy inequality basic expression}	
      \int_{\mathbb{R}^d}\eta^2 \phi(f(t_2))\;dv-\int_{\mathbb{R}^d}\eta^2 \phi(f(t_1))\;dv & = \textnormal{(I)}+\textnormal{(II)}+\textnormal{(III)}+\textnormal{(IV)}.	
    \end{align}	
    Let us bound each of the terms $\textnormal{(I)}-\textnormal{(IV)}$.  For every $\delta>0$ we have the elementary inequality	
    \begin{align*}
      |\overline{\phi}(f)(A\nabla (\eta \overline{\phi}(f)),\nabla \eta)| \leq \frac{\delta}{2}(A\nabla (\eta \overline{\phi}(f)),\nabla (\eta \overline{\phi}(f)))+ \frac{1}{2\delta}\overline{\phi}(f)^2(A\nabla \eta,\nabla \eta),
    \end{align*}
    which we shall use twice to bound the terms $(\textnormal{I})$ and $(\textnormal{IV})_1$ defined in Proposition \ref{prop:basic energy inequality} and Proposition \ref{prop:energy inequality term IV}. Then, for positive $\delta_1$ and $\delta_2$ (to be determined later), we have
    \begin{align*}
      (\textnormal{I}) & \leq -(1-\delta_1)\int_{t_1}^{t_2}\int_{\mathbb{R}^d}(A\nabla (\eta \overline{\phi}(f)),\nabla (\eta \overline{\phi}(f)))\;dvdt+\frac{1}{\delta_1}\int_{t_1}^{t_2}\int_{\mathbb{R}^d}\overline{\phi}(f)^2(A\nabla \eta,\nabla \eta)\;dvdt\\
      & \leq -(1-\delta_1)\int_{t_1}^{t_2}\int_{\mathbb{R}^d}(A\nabla (\eta \overline{\phi}(f)),\nabla (\eta \overline{\phi}(f)))\;dvdt+\frac{1}{\delta_1}\|\nabla \eta\|_{L^\infty}^2 \int_{t_1}^{t_2}\int_{\textnormal{spt}(\eta)}\overline{\phi}(f)^2a\;dvdt.	
    \end{align*}
    and,
    \begin{align*}  
      (\textnormal{IV})_1 & \leq 4c_p\delta_2 \int_{t_1}^{t_2}\int_{\mathbb{R}^d}(A\nabla (\eta \overline{\phi}(f)),\nabla (\eta \overline{\phi}(f)) )\;dvdt\\
  	&\;\;\;\; +4c_p(2+\delta_2^{-1}) \int_{t_1}^{t_2}\int_{\mathbb{R}^d}\overline{\phi}(f)^2(A\nabla \eta,\nabla \eta)\;dvdt-2c_p\int_{t_1}^{t_2}\int_{\mathbb{R}^d}\overline{\phi}(f)^2\tr(AD^2\eta^2)\;dvdt\\
      & \leq 4c_p \delta_2 \int_{t_1}^{t_2}\int_{\mathbb{R}^d}(A\nabla (\eta \overline{\phi}(f)),\nabla (\eta \overline{\phi}(f)) )\;dvdt\\
      & \;\;\;\;+4c_p\left  ( 2+\delta_2^{-1}\right )\left (\|\nabla \eta\|_{L^\infty}^2+\|D^2\eta^2\|_{L^\infty} \right )\int_{t_1}^{t_2}\int_{\textnormal{spt}(\eta)}\overline{\phi}(f)^2a\;dvdt.
    \end{align*}
    On the other hand, applying inequality \eqref{eqn:upper bounds for under and over bar phi} from Lemma \ref{lem:properties of approximations to p-th power}, it follows that for $h$ large enough
    \begin{align*}
      (\textnormal{II}) \leq 8\left ( \|\nabla \eta\|_{L^\infty}^2 + \|D^2\eta^2\|_{L^\infty} \right )\int_{t_1}^{t_2}\int_{\textnormal{spt}(\eta)}\phi(f)a\;dvdt.	  
    \end{align*}	  
    Then, choosing $\delta_1 = 1/4$ and $\delta_2 = (4c_p)^{-1}/4$, it follows that
    \begin{align*}
      (\textnormal{I})+(\textnormal{II})+(\textnormal{IV}) & \leq -\frac{1}{2}\int_{t_1}^{t_2}\int_{\mathbb{R}^d}(A\nabla (\eta \overline{\phi}(f)),\nabla (\eta \overline{\phi}(f) )\;dvdt\\
      & \;\;\;\;+8c_p(6+16 c_p )\left (\|\nabla \eta\|_{L^\infty}^2+\|D^2\eta^2\|_{L^\infty} \right )\int_{t_1}^{t_2}\int_{\textnormal{spt}(\eta)}\phi(f)a\;dvdt+(\textnormal{IV})_2.			
    \end{align*}
    Again by Lemma \ref{lem:properties of approximations to p-th power}, provided $h$ is sufficiently large, we have
    \begin{align*}
      \int_{t_1}^{t_2}\int_{\mathbb{R}^d}\eta^2 \underline{\phi}(f)h\;dvdt \leq \frac{p}{2}\int_{t_1}^{t_2}\int_{\mathbb{R}^d} (\eta \overline{\phi}(f))^2h\;dvdt.
    \end{align*}
    Next, we apply the $\varepsilon$-Poincar\'e with $\varepsilon = \frac{1}{2p}$,
    \begin{align*}
      (\textnormal{III}) \leq \frac{1}{4}\int_{t_1}^{t_2}\int_{\mathbb{R}^d}(A\nabla (\eta \overline{\phi}(f)),\nabla (\eta \overline{\phi}(f)))\;dvdt+\tfrac{p}{2}\Lambda_f \left ( \tfrac{1}{2p} \right )\int_{t_1}^{t_2}\int_{\mathbb{R}^d}(\eta \overline{\phi}(f))^2\;dvdt.
    \end{align*}
    Going back to \eqref{eqn:energy inequality basic expression}, we conclude that for any pair of times $t_1<t_2$, we have 
    \begin{align*}
      & \int_{\mathbb{R}^d}\eta^2 \phi(f(t_2))\;dv + \frac{1}{4}\int_{t_1}^{t_2}\int_{\mathbb{R}^d}(A\nabla (\eta \overline{\phi}(f)),\nabla (\eta \overline{\phi}(f) )\;dvdt \\
      & \leq \int_{\mathbb{R}^d}\eta^2 \phi(f(t_1))\;dv+\tfrac{p}{2}\Lambda_f (\tfrac{1}{2p})\int_{t_1}^{t_2}\int_{\mathbb{R}^d}(\eta \overline{\phi}(f))^2\;dvdt\\
      & \;\;\;\;+C(p_0)\left (\|\nabla \eta\|_{L^\infty}^2+\|D^2\eta^2\|_{L^\infty} \right )\int_{t_1}^{t_2}\int_{\textnormal{spt}(\eta)}\phi(f)a\;dvdt\\
      &\;\;\;\;+\frac{p+4}{2}\int_{t_1}^{t_2}\int_{\{f>h\}} (h+1)^{p-1}(f-h)|\nabla a||\nabla \eta^2|\;dvdt,
    \end{align*}
    where $C(p_0) := 8c_{p_0}(6+16c_{p_0}) \geq 8c_{p}(6+16c_{p})$ for all $p\geq p_0$. Observe that all terms appearing on both sides of the inequality are non-negative. Fix $t_2 \in (T_2,T_3)$ and take the average of both sides of this inequality with respect to $t_1 \in (T_1,T_2)$, the resulting inequality says then that the quantity
    \begin{align*}
      & \int_{\mathbb{R}^d}\eta^2 \phi(f(t_2))\;dv + \frac{1}{4}\int_{T_2}^{t_2}\int_{\mathbb{R}^d}(A\nabla (\eta \overline{\phi}(f)),\nabla (\eta \overline{\phi}(f) )\;dvdt
    \end{align*}
    is no larger than	
    \begin{align*}		  
      & \leq \frac{1}{T_2-T_1}\int_{T_1}^{T_2}\int_{\mathbb{R}^d}\eta^2 \phi(f)\;dvdt+\tfrac{p}{2}\Lambda_f(\tfrac{1}{2p})\int_{T_1}^{t_2}\int_{\mathbb{R}^d}(\eta \overline{\phi}(f))^2\;dvdt\\
      & \;\;\;\;+C(p_0)\left (\|\nabla \eta\|_{L^\infty}^2+\|D^2\eta^2\|_{L^\infty} \right )\int_{T_1}^{t_2}\int_{\textnormal{spt}(\eta)}\phi(f)a\;dvdt\\
      &\;\;\;\;+\frac{p+4}{2}\int_{T_1}^{t_2}\int_{\{f>h\}} (h+1)^{p-1}(f-h)|\nabla a||\nabla \eta^2|\;dvdt.
    \end{align*}
  
  \end{proof}
  
 
  Let us note that the last term in the previous Lemma ought to go to zero as $h\to \infty$ when $f$ has enough integrability (the integrability requirement being higher the larger $p$ is). Since our intention is to follow Moser's approach to De Giorgi-Nash-Moser estimates, we shall show alternate between obtaining some integrability of $f$, to showing that the expression above goes to zero for some $p$, which in turn yields better integrability, and so on. The next Lemma gives a starting point in the case $\gamma \in (-2,0]$ (for the remaining $\gamma$ we will assume $f$ is bounded and treat the result as an a priori estimate). 
  \begin{lem}\label{lem:Ld/d-1 bound for f unif in time}
  Let $\gamma \in (-2,0]$. Consider times $0<\tau<\tau'$. Then, 
   \begin{align*}
      \sup \limits_{(\tau,\tau')} \int_{\mathbb{R}^d} (f(t))^{\frac{d}{d-1}}\;dv < \infty.
    \end{align*}	
  \end{lem}
  
  \begin{proof}
  Let $\phi = \phi_{p,h}$, be as in Definition \ref{def:Definitions of approximations to p-th power I} with $p = p_\gamma$ as in Proposition \ref{prop: LpLp bound via interpolation1}. For $R>1$ let $\eta_R$ be a smooth function such that $0\leq \eta_R \leq 1$ in $\mathbb{R}^d$ and such that
  \begin{align*}
    \eta_R \equiv 1 \textnormal{ in } B_R, \;\eta_R \equiv 0 \textnormal{ in } \mathbb{R}^d\setminus B_{2R},\;\;|\nabla \eta_R|^2+|D^2\eta_R| \leq CR^{-2}.			  
  \end{align*}	  
  Lemma \ref{lem:energy inequality with bad term} says that 
  \begin{align*}
      & \int_{\mathbb{R}^d}\eta_R^2 \phi(f(t_2))\;dv + \frac{1}{4}\int_{T_2}^{t_2}\int_{\mathbb{R}^d}(A\nabla (\eta_R \overline{\phi}(f)),\nabla (\eta_R \overline{\phi}(f) )\;dvdt \\
      & \leq \frac{2}{T_2-T_1}\int_{T_1}^{T_2}\int_{\mathbb{R}^d}\eta_R^2 \phi(f)\;dvdt+p \Lambda_f  (\tfrac{1}{2p})\int_{T_1}^{T_3}\int_{\mathbb{R}^d}(\eta_R \overline{\phi}(f))^2\;dvdt\\
      & \;\;\;\;+2C(p_\gamma)\left (\|\nabla \eta_R\|_{L^\infty}^2+\|D^2\eta_R^2\|_{L^\infty} \right )\int_{T_1}^{T_3}\int_{\textnormal{spt}(\nabla \eta_R)}\phi(f)a\;dvdt\\
      &\;\;\;\;+(p+4)\int_{T_1}^{T_3}\int_{\{f>h\}} (h+1)^{p-1}(f-h)|\nabla a||\nabla \eta_R^2|\;dvdt. 
  \end{align*}
  Since $\gamma \in [-2,0]$, Proposition \ref{prop:a vs a star gamma larger than -2} yields that $R^{-2}a \leq C(d,\gamma,\fin) \astar$ in $B_{2R}\setminus B_R$, it follows that the term
  \begin{align*}
  \left (\|\nabla \eta_R\|_{L^\infty}^2+\|D^2\eta_R^2\|_{L^\infty} \right )\int_{T_1}^{T_3}\int_{\textnormal{spt}(\nabla \eta_R)}\phi(f)a\;dvdt,
  \end{align*}
  is bounded by
  \begin{align*}
    C(d,\gamma,\fin)\int_{T_1}^{T_3} \int_{\textnormal{spt}(\nabla \eta_R)} \phi(f) \astar \;dvdt.
  \end{align*}
  From the integrability of $f\astar$ over the whole space (which follows  from the upper bound for $\astar$ and the second moment bound for $f)$ it follows that the above integral converges to zero as $R\to \infty$, therefore 
  \begin{align*}
    \lim\limits_{R\to\infty }\left (\|\nabla \eta_R\|_{L^\infty}^2+\|D^2\eta_R^2\|_{L^\infty} \right )\int_{T_1}^{T_3}\int_{\textnormal{spt}(\nabla \eta_R)}\phi(f)a\;dvdt = 0.
  \end{align*}
  Next, we consider the term
  \begin{align*}
    \int_{T_1}^{T_3} \int_{\textnormal{spt}(\nabla \eta_R)} (h+1)^{p-1}(f-h)_+|\nabla a||\nabla \eta_R^2|\;dvdt.
  \end{align*}
  Always keeping in mind that
  \begin{align*}
    |\nabla a[f] | \le \int \frac{f(y)}{|v-y|^{-\gamma-1}}\;dy.
  \end{align*}
  If $\gamma \geq -1$, then arguing as in Proposition \ref{prop:a vs a star gamma larger than -2} (using the second moment bound for $f$) it is not easy to see that $|\nabla a[f]|\leq C(d,\gamma,\fin) \langle v\rangle^{1+\gamma}$ for all $v$. Since $1+\gamma\leq 2$ in any case, using again the second moment bound for $f$  it follows that (with $h$ fixed)
  \begin{align*}
    \lim \limits_{R \to \infty}\int_{T_1}^{T_3} \int_{\textnormal{spt}(\nabla \eta_R)}(h+1)^{p-1}(f-h)_+|\nabla a||\nabla \eta_R^2|\;dvdt = 0.
  \end{align*}
  Next, we prove the same limit as above in the case $-2 <\gamma<-1$. Let us show that
  \begin{align*}
    \int_{T_1}^{T_3}\int_{\mathbb{R}^d}f|\nabla a| \;dvdt<\infty.
  \end{align*}
  From the definition of $\nabla a$, we have $|\nabla a(v)| \leq \tilde a_1+\tilde a_2$, where
  \begin{align*}
    \tilde a_1 & := C(d,\gamma)\int_{B_1(v)}f(w)|v-w|^{1+\gamma}\;dw,\\
    \tilde a_2 & := C(d,\gamma)\int_{\mathbb{R}^d\setminus B_1(v)}f(w)|v-w|^{1+\gamma}\;dw.
  \end{align*}
  Since $1+\gamma\leq 0$ and $|v-w|\geq 1$, we have for all times
  \begin{align*}
    \tilde a_2(v,t) = C(d,\gamma)\int_{\mathbb{R}^d\setminus B_1(v)}f(w,t)|v-w|^{1+\gamma}\;dw \leq C(d,\gamma)\|f\|_{L^1(\mathbb{R}^d)}. 
  \end{align*}	  
  From here, it is immediate that
  \begin{align*}
    \int_{T_1}^{T_3}\int_{\mathbb{R}^d}f \tilde a_2 \;dvdt<\infty.
  \end{align*}
  Let us now deal with the term $\tilde a_1$. Given that $\gamma \in (-2,-1)$, we have
  \begin{align*}
    C(d,\gamma)\chi_{B_1(0)}|v|^{1+\gamma} \in L^{q_\gamma}(\mathbb{R}^d),\;\; q_\gamma := \frac{d}{-\gamma-1}    . 
  \end{align*}
  Therefore, by the preservation of the $L^1(\mathbb{R}^d)$ norm of $f(t)$, we conclude that
  \begin{align*}
    \| \tilde a_1 \|_{L^\infty(T_1,T_3,L^{q_\gamma}(\mathbb{R}^d) )} <\infty. 
  \end{align*}
  With this in mind, we apply Young's inequality, which yields 
  \begin{align*}
    f\tilde a_1 \leq f^{\frac{q_\gamma}{q_\gamma-1}} + \tilde a_1^{q_\gamma}.
  \end{align*}
  Therefore, 
  \begin{align*}
    & \int_{T_1}^{T_3} \int_{\mathbb{R}^d} f \tilde a_1 \;dvdt \leq  \int_{T_1}^{T_3} \int_{\mathbb{R}^d} f^{\frac{q_\gamma}{q_\gamma-1}} \;dvdt + \int_{T_1}^{T_3} \int_{\mathbb{R}^d} \tilde a_1^{q_\gamma }\;dvdt.
  \end{align*}
  Now, recall that by Proposition \ref{prop: LpLp bound via interpolation1} we know that $\int_{T_1}^{T_3}\int_{\mathbb{R}^d} f^{p_\gamma}\;dvdt$ is finite. Moreover, from the definition of $p_\gamma$ and the fact that $q_\gamma>d$ we have  $\frac{q_\gamma}{q_\gamma-1}<p_\gamma$, it follows that
  \begin{align*}
    \int_{T_1}^{T_3} \int_{\mathbb{R}^d} f^{\frac{q_\gamma}{q_\gamma-1}} \;dvdt  <\infty.
  \end{align*}
  At the same time, from the bound for $\tilde a_1$, we have
  \begin{align*}
    \int_{T_1}^{T_3}\int_{\mathbb{R}^d}\tilde a_1^{q_\gamma} \;dvdt  = \int_{T_1}^{T_3} \|\tilde a_1(t)\|_{q_\gamma}^{g_\gamma}\;dt <\infty.
  \end{align*}
  In either case, for $\gamma \in [-2,0]$ and with $h>0$ fixed, we have
  \begin{align*}
    \lim\limits_{R\to\infty} \int_{T_1^{t_2}}\int_{\textnormal{spt}(\nabla \eta_R)}(h+1)^{p-1}(f-h)_+|\nabla a||\nabla \eta_R^2|\;dvdt = 0.
  \end{align*}
  Therefore,
  \begin{align*}
      \int_{\mathbb{R}^d}\phi(f(t_2))\;dv \leq \frac{2}{T_2-T_1}\int_{T_1}^{T_2}\int_{\mathbb{R}^d} \phi(f)\;dvdt+p\Lambda_f(\tfrac{1}{2p})\int_{T_1}^{T_3}\int_{\mathbb{R}^d}\phi(f)\;dvdt.
  \end{align*}  
  Taking the sup with respect to $t_2 \in (T_2,T_3)$ we get 
   \begin{align*}
    \sup_{t\in(T_2,T_3)}  \int_{\mathbb{R}^d}\phi(f)(t)\;dv  \leq \left( \frac{2}{T_2-T_1} \right)\int_{T_1}^{T_3}\int_{\mathbb{R}^d}\phi(f)\;dvdt+p\Lambda_f(\tfrac{1}{2p})\int_{T_1}^{T_3}\int_{\mathbb{R}^d} \phi(f)\;dvdt
    \end{align*}
  It remains to take the limit as $h\to 0^+$, recalling that $\phi = \phi_{p_\gamma,h}$. Using Proposition \ref{prop: LpLp bound via interpolation1} and Lemma \ref{lem:properties of approximations to p-th power}, we conclude that 
  \begin{align*}
    & \limsup \limits_{h\to 0^+} \left \{ \left( \frac{2}{T_2-T_1} \right)\int_{T_1}^{T_3}\int_{\mathbb{R}^d}\phi_{p,\gamma}(f)\;dvd+p\Lambda_f(\tfrac{1}{2p})\int_{T_1}^{T_3}\int_{\mathbb{R}^d} \phi_{p_\gamma,h}(f)\;dvdt  \right \} \\
    & \leq \left ( \frac{2}{T_2-T_1}+p\Lambda_{f}(\tfrac{1}{2p})\right ) \int_{T_1}^{T_3} \int_{\mathbb{R}^d}f^{p_\gamma}\;dvdt \leq C(d,\gamma,\fin)  \left ( \frac{1}{T_2-T_1}+p\Lambda_{f}(\tfrac{1}{2p})\right ) 
  \end{align*}
  At the same time, 
  \begin{align*}
    \sup_{t\in(T_2,T_3)} & \int_{\mathbb{R}^d}\frac{1}{p_\gamma}f^{p_\gamma}(t)\;dv  \leq \limsup\limits_{h\to 0^+} \sup_{t\in(T_2,T_3)}  \int_{\mathbb{R}^d}\phi_{p_\gamma,h}(f)(t)\;dv.
  \end{align*}
  In conclusion, 
  \begin{align*}
    \sup_{t\in(T_2,T_3)} & \int_{\mathbb{R}^d}f^{p_\gamma}(t)\;dv  \leq C(d,\gamma,\fin)  \left ( \frac{1}{T_2-T_1}+p\Lambda_{f}(\tfrac{1}{2p})\right ) < \infty. 
  \end{align*}
  It is not hard to see that $p_\gamma \geq \frac{d}{d-1}$ with $p_\gamma$ as in Proposition \ref{prop: LpLp bound via interpolation1}, in the case $\gamma \in [-2,0]$. From here, the proposition follows by interpolation with the $L^1$ norm. 
  \end{proof}

  \begin{prop}\label{prop:energy inequality negligible term}
    Let $p=1+\tfrac{2}{d}$, $\eta \in C^2_c(\mathbb{R}^d)$ and times $\tau<\tau'$. Suppose that $f$ is a weak solution and either 1) $\gamma \in (-2,0]$ or 2) $\gamma \in [-d,-2]$ and $f$ is bounded.  Then,
    \begin{align*}
      \lim\limits_{h\to\infty}\int_{\tau}^{\tau'}\int_{\{f>h\}}(h+1)^{p-1}(f-h)_+|\nabla a||\nabla \eta^2|\;dvdt = 0.
    \end{align*}

  \end{prop}

  \begin{proof}
    The second case is trivial, since $\{f>h\}$ is empty for large enough $h$, so the integral is exactly zero for all large $h$. Let us prove the proposition in the first case. Taking $h\geq 1$ (without loss of generality), we have
    \begin{align*}
      \int_{\tau}^{\tau'}\int_{\{f>h\}}(h+1)^{p-1}(f-h)_+|\nabla a||\nabla \eta^2|\;dvdt \leq 2^p\|\nabla a\|_{L^\infty}\int_{\tau}^{\tau'}\int_{\{f>h\}}f^p|\nabla \eta^2|\;dvdt
    \end{align*}
    The $L^\infty$ norm being over $(\tau,\tau')\times \textnormal{spt}(
    \nabla \eta^2)$. We observe that the Proposition will be proved once we show that this $L^\infty$ norm of $\nabla a$ is finite. Since we are in the case $\gamma \in (-2,0]$, for any $t\in (\tau,\tau')$ we have for any ball $B_r$,
    \begin{align*}
      \|\nabla a(t)\|_{L^\infty(B_r)} \leq C\| f * |v|^{1+\gamma}\|_{L^\infty(B_r)},
    \end{align*}
    from this inequality, and the $L^{\frac{d}{d-1}}$ bound from Lemma \ref{lem:Ld/d-1 bound for f unif in time}, the proposition follows.
  \end{proof}

  \begin{cor}\label{cor_iterat}
    Let $f$ be a weak solution in $(0,T)$, $p\geq 1+\tfrac{2}{d}$, and $\eta \in C^2_c(\mathbb{R}^d)$. Let us assume that either $\gamma \in (-2,0]$ and
    \begin{align*}
      \sup \limits_{T_1\leq t\leq T_3} \int_{\textnormal{spt}(\eta)} (f(t))^p\;dv < \infty,	  
    \end{align*}	  
    or that $\gamma \in [-d,-2]$ and $f$ is a bounded solution for which Assumption \ref{Assumption:Epsilon Poincare} holds. Then,
    \begin{align*}
      \sup \limits_{(T_2,T_3)}\int_{\mathbb{R}^d}\eta^2 f^p\;dv+\frac{p-1}{p}\int_{T_2}^{T_3}\int_{\mathbb{R}^d}(A\nabla (\eta f^{\frac{p}{2}}),\nabla (\eta f^{\frac{p}{2}})\;dvdt,
    \end{align*}
    is no larger than
    \begin{align*}
      & \left (\frac{1}{T_2-T_1}+\tfrac{p}{2} \Lambda_f (\tfrac{1}{2p}) \right )\int_{T_1}^{T_3}\int_{\mathbb{R}^d}\eta^2 f^p\;dvdt +C(d)\left (\|\nabla \eta\|_\infty^2+\|D^2\eta^2\|_\infty \right )\int_{T_1}^{T_3}\int_{\textnormal{spt}(\eta)}f^pa\;dvdt.	
    \end{align*}
  
  \end{cor}

  \begin{proof}
    Let $\tau \in (T_2,T_3)$, then by Lemma \ref{lem:energy inequality with bad term}
    \begin{align*}
      \int_{\mathbb{R}^d}\eta^2 \phi(f)(\tau)\;dv+\frac{1}{4}\int_{T_2}^{T_3}\int_{\mathbb{R}^d}(A\nabla (\eta \overline{\phi}(f) ) ,\nabla (\eta \overline{\phi}(f) ) )\;dvdt,
    \end{align*}
    is bounded from above by
    \begin{align*}
      & \left ( \frac{1}{T_2-T_1} +\tfrac{p}{2}\Lambda_f(\tfrac{1}{2p})\right )\int_{T_1}^{T_3}\int_{\mathbb{R}^d}\eta^2 \phi(f)(t)\;dvdt +C(d)\left (\|\nabla \eta\|_\infty^2+\|D^2\eta^2\|_\infty \right )\int_{T_1}^{T_3}\int_{\textnormal{spt}(\eta)}\phi(f)a\;dvdt\\
      & +\frac{p+4}{2}\int_{T_1}^{T_3}\int_{\{f>h\}} (h+1)^{p-1}(f-h)|\nabla a||\nabla \eta^2|\;dvdt.	
    \end{align*}
    Applying monotone convergence, the lower semicontinuity of the (weighted) $\dot H^1$ seminorm, and the same argument used in Proposition \ref{prop:energy inequality negligible term}, we obtain in the limit $h\to \infty$ the inequality
    \begin{align*}
      & \frac{1}{p}\int_{\mathbb{R}^d}\eta^2 f^p(\tau)\;dv+\frac{p-1}{4p^2}\int_{T_2}^{T_3}\int_{\mathbb{R}^d}(A\nabla (\eta f^{\frac{p}{2}} ) ,\nabla (\eta f^{\frac{p}{2}} ) )\;dvdt\\
      & \leq \left ( \frac{1}{T_2-T_1} +\tfrac{p}{2}\Lambda_f (\tfrac{1}{2p})\right )\frac{1}{p}\int_{T_1}^{T_3}\int_{\mathbb{R}^d}\eta^2 f^p\;dvdt \\
    & \;\;\;\;+C(d)\left (\|\nabla \eta\|_\infty^2+\|D^2\eta^2\|_\infty \right )\frac{1}{p} \int_{T_1}^{T_3}\int_{\textnormal{spt}(\eta)}f^p a\;dvdt.
    \end{align*}  
    To finish the proof, we multiply both sides of the resulting inequality by $p$ and take the supremum of the left hand side with respect to $\tau \in (T_2,T_3)$.
  \end{proof}
  
 \subsection{De Giorgi-Nash-Moser iteration} Fix $R>1$, we define the sequences
  \begin{align}\label{eqn:iteration times and radii}
    T_n = \frac{1}{4}\left ( 2- \frac{1}{2^n} \right )T,\;\;R_n = \frac{1}{2}\left(1 + \frac{1}{2^n}\right)R.
  \end{align}
  Associated to $R_n$, we will denote by $\eta_n$ a sequence of  functions which are such that
  \begin{align}
    &  \eta_n \in C^2_c(B_{R_n}(0)), \;\;0\leq \eta_n \leq 1 \textnormal{ in } \mathbb{R}^d,\;\; \eta_n \equiv 1 \textnormal{ in } B_{R_{n+1}}(0), \notag\\
    & \|\nabla \eta_n\|_{L^\infty} \leq CR^{-1}2^n,\;\;\|D^2\eta_n\|_{L^\infty} \leq CR^{-2}4^n.	\label{eqn:iteration cut off functions}			  
  \end{align}

\begin{lem}\label{lem:weighted Moser estimate}{(Moser's Iteration)}
 Let $f$ be a weak solution in $(0,T)$ if $\gamma \in (-2,0]$, or let $f$ be a bounded weak solution satisfying Assumptions \ref{Assumption:Epsilon Poincare} and \ref{Assumption:Local Doubling} if $\gamma \in [-d,-2]$.	Then with $p=1+\tfrac{2}{d}$ and $R>1$ there is a constant $C$ with $C=C(d,\gamma,\fin)$ or $C=C(d,\gamma,\fin,C_D,C_P,\kappa_P)$ accordingly, such that
\begin{align*}
   \|f\|_{L^\infty(B_{R/2}(0)\times (T/2,T))} \leq C\left \{ R^{-\gamma}\left (\frac{1}{T}+1 \right ) \right \}^{\frac{1}{p}\frac{q}{q-2}}\left (\int_{T/4}^T \int_{B_R(0)} f^{p}\astar_{f,\gamma} \;dvdt \right )^{\frac{1}{p}}.
\end{align*}
Here, $q = 2+\frac{4}{d}$ if $\gamma>-d$ or $q$ can be any exponent in $(2,2+\frac{4}{d})$ if $\gamma=-d$.
\end{lem}

\begin{proof}
  With $p=1+\tfrac{2}{d}$ fixed, we define the sequence
  \begin{align*}
    p_n = p \left (\frac{q}{2} \right)^n,\;\;\forall\;n\in\mathbb{N},
  \end{align*}	  
  with $q$ as above. Then, for each $n\ge 0$, let $E_n$ denote the quantity,
  \begin{align*}
    E_n := \left ( \int_{T_n}^T\int \eta_n^q f^{p_n}\astar\;dvdt\right )^{\frac{1}{p_n}}.
  \end{align*}
  where the times $T_n$ and the functions $\eta_n$ are as in \eqref{eqn:iteration times and radii} and \eqref{eqn:iteration cut off functions}. 
  
  As it is standard for divergence elliptic equations, we are going to derive a recursive relation for $E_n$. To do this, first note $E_{n+1}$ may be written as,
  \begin{align*}
    E_{n+1}^{p_{n}} = \left ( \int_{T_{n+1}}^T\int_{B_R(0)}  \left ( \eta_{n+1} f^{\frac{p_{n}}{2}}\right )^{q}\astar\;dvdt \right )^{\frac{2}{q}}.
  \end{align*}
 Therefore, Theorem \ref{lem:Inequalities Sobolev weight aII} (if $\gamma\in(-2,0]$) and Theorem \ref{thm:Inequalities Sobolev weight gamma below -2} (if $\gamma\in[-d,-2]$, under Assumption \ref{Assumption:Local Doubling}) lead to
  \begin{align*}
     E_{n+1}^{p_{n}} \leq C \left \{ \sup \limits_{T_{n+1} \leq t\leq T} \left \{ \int \eta_{n+1}^2 f^{p_n}(t)\;dv \right \} + \frac{(p_n-1)}{p_n}\int_{T_{n+1}}^{T}\int a^* |\nabla (\eta_{n+1} f^{\frac{p_n}{2}})|^2\;dvdt \right \},
  \end{align*}
  where (recall $C_D$ denotes the constant in Assumption \ref{Assumption:Local Doubling})
  \begin{align*}
    C & = C(d,\gamma) \textnormal{ if } \gamma \in (-2,0], \; C = C(d,\gamma,C_D) \textnormal{ if } \gamma \in (-d,-2], and\\
    C & = C(d,\gamma,q,C_D) \textnormal{ if } \gamma=-d.
  \end{align*}
  Applying Corollary \ref{cor_iterat} with $T_1 = T_n$, $T_2 = T_{n+1}$ and $T_3 = T$ we have 
  \begin{align*}
    C^{-1}E_{n+1}^{p_{n}} & \leq \left( \frac{2^{n+2}}{T}+ p_n\Lambda_f (\tfrac{1}{2p_n})  \right) \int_{T_n}^{T}\int_{\mathbb{R}^d} \eta_{n+1}^2f^{p_n}\;dvdt\\
    & \;\;\;\;+C(d)\left ( \|\nabla \eta_{n+1}\|_{\infty}^2+\|D^2\eta_{n+1}^2\|_{\infty} \right )\int_{T_n}^{T}\int_{\textnormal{spt}(\eta_{n+1})}  f^{p_n}a \;dvdt.
  \end{align*}	
  Since $\eta_n \equiv 1$ in the support of $\eta_{n+1}$, and $\nabla \eta_n,D^2\eta_n$ are supported in $B_{R_{n+1}}\setminus B_{R_{n+2}}$, where we have $|v|\approx R$ (since $R\geq 1$), we have the pointwise inequalities
  \begin{align*}
    \eta_{n+1}^2 & \leq \eta_{n}^q,\\
    |\nabla \eta_{n+1}| & \leq C2^nR^{-1} \eta_n^q \leq C2^n\langle v\rangle^{-1}\eta_n^q,\\
    \eta_{n+1}|D^2\eta_{n+1}| & \leq C4^nR^{-2}\eta_n^q\leq C 4^n\langle v\rangle^{-2}\eta_n^q.		  
  \end{align*}	  
  Substituting these in the above inequality, we have
  \begin{align*}	
    C^{-1}E_{n+1}^{p_{n}} & \leq \left ( \frac{2^{n+2}}{T}+p_n\Lambda_f(\tfrac{1}{2p_n}) \right )\int_{T_n}^T \int_{\mathbb{R}^d}  \eta_n^q f^{p_n} \;dvdt  +C(d)4^n \int_{T_n}^T \int_{\mathbb{R}^d}  \eta_n^q f^{p_n} \langle v\rangle^{-2}a\;dvdt.
  \end{align*}	
  Now we apply Theorem \ref{thm:sufficient conditions for the Poincare inequality} for $\gamma \in (-2,0]$, and for $\gamma \in [-d,-2]$ we make use of Assumption \ref{Assumption:Epsilon Poincare} (or use Theorem \ref{thm:sufficient conditions for the Poincare inequality} with the respective extra Assumption), and obtain
  \begin{align*}	
    C^{-1}E_{n+1}^{p_{n}} & \leq \left ( \frac{2^{n+2}}{T}+C_P(2p_n)^{1+\kappa_P} \right )\int_{T_n}^T \int_{\mathbb{R}^d}  \eta_n^q f^{p_n} \;dvdt  +C(d)4^n \int_{T_n}^T \int_{\mathbb{R}^d}  \eta_n^q f^{p_n} \langle v\rangle^{-2}a\;dvdt.
  \end{align*}	  
  Let us analyze each integral on the right. First, we have $1\leq CR^{-\gamma}\astar$ in $B_{R}(0)$, so that
  \begin{align*}
    & \left ( \frac{2^{n+2}}{T}+C_P(2p)^{1+\kappa_P}(q/2)^{n(1+\kappa_P)}  \right )\int_{T_n}^T \int_{\mathbb{R}^d}  \eta_n^q f^{p_n} \;dvdt \\
    & \leq C \max\{2,(q/2)^{1+\kappa_P}\}^n\left ( \frac{1}{T}+1 \right )R^{-\gamma} \int_{T_n}^T \int_{\mathbb{R}^d}  \eta_n^q f^{p_n}\astar \;dvdt,	  
  \end{align*}	  
  with $C= C(d,\gamma,\fin,C_P,\kappa_P)$. On the other hand, by Proposition \ref{prop:a vs a star gamma larger than -2} (for $\gamma \in (-2,0]$) and Proposition \ref{prop:a vs a star gamma below -2} (for $\gamma\in [-d,-2]$) we have
  \begin{align*}
    C(d)4^n\int_{T_n}^T\int_{\mathbb{R}^d}\eta_n^q f^{p_n}\langle v\rangle^{-2}a\;dvdt \leq C4^n R^{\underline{m}} \int_{T_n}^T\int_{\mathbb{R}^d}\eta_n^q f^{p_n}\astar \;dvdt,
  \end{align*}
  where, again, we have $C=C(d,\gamma,\fin)$ or $C=C(d,\gamma,\fin,C_D)$ depending on whether $\gamma>-2$ or not, and where 
  \begin{align*}
    \underline{m} = 0 \textnormal{ if } \gamma\in (-2,0], \underline{m} = \max\{-\gamma-4,0\} \textnormal{ if } \gamma\in [-d,-2].	  
  \end{align*}	  
  Let us write
  \begin{align*}
    b := \max\{4,(q/2)^{1+\kappa_P}\}.
  \end{align*}	  
   Since $R\geq 1$, we may add up terms taking the worst factor of $R$ (note that we always have $\underline{m}\leq -\gamma$) therefore we arrive at
  \begin{align*}	
    & E_{n+1}^{p_{n}} \leq Cb^nR^{-\gamma }\left( \frac{1}{T}+1 \right)\int_{T_n}^{T}\int_{\mathbb{R}^d} \eta_n^q f^{p_n} \astar \;\;dvdt
  \end{align*}
  Taking the $1/p_n$ root from both sides of the last inequality, it follows that
  \begin{align*}
    E_{n+1} & \leq C^{\frac{1}{p}\left ( \frac{2}{q}\right)^n} \left ( b^{\frac{1}{p}} \right )^{n\left (\frac{2}{q} \right )^n}\left \{  R^{-\gamma}\left( \frac{1}{T}+1 \right) \right \}^{\frac{1}{p}\left (\frac{2}{q} \right)^n} E_{n}.
  \end{align*}
  Then, for a universal $b_0$, we have 
  \begin{align*}
    E_{n+1} \leq b_0^{n\left (\frac{2}{q}\right )^n}\left \{ R^{-\gamma}\left (\frac{1}{T}+1 \right ) \right \}^{\frac{1}{p}\left ( \frac{2}{q}\right )^n} E_n.
  \end{align*}
  This recursive relation, and a straightforward induction argument, yield that
  \begin{align*}
    E_{n} \leq b_0^{ \sum \limits_{k=0}^{n-1}k \left ( \frac{2}{q}\right )^k}\left \{  R^{-\gamma}\left ( \frac{1}{T}+1\right ) \right \}^{\frac{1}{p} \sum \limits_{k=0}^{n-1} \left (\frac{2}{q} \right )^k } E_0.
  \end{align*}
  Now, since
  \begin{align*}
    \sum \limits_{k=0}^{n-1}\left ( \frac{2}{q}\right )^k \leq \sum \limits_{k=0}^\infty  \left ( \frac{2}{q}\right )^k  \frac{q}{q-2},\;\;\; \sum \limits_{k=0}^{n-1}k \left ( \frac{2}{q}\right )^k \leq \sum \limits_{k=0}^\infty k \left ( \frac{2}{q}\right )^k = \frac{2q}{(q-2)^2},
  \end{align*}
  we conclude that
  \begin{align}\label{eqn:Moser iteration Lpn estimate for all n}
    E_n \leq C\left \{ R^{-\gamma}\left (\frac{1}{T}+1 \right ) \right \}^{\frac{1}{p}\frac{q}{q-2}}E_0,
  \end{align}
  where $C=C(d,\gamma,\fin)$ or $C=(d,\gamma,\fin,C_D,C_P,\kappa_P)$ accordingly, and
  \begin{align*}
    E_0 = \left( \int_{T/4}^T \int_{B_R(0)} f^p \astar  dvdt \right)^{1/p}.
  \end{align*}	
  Now, since $\eta_n\geq 1$ in $B_{R/2}(0)$ and $T_n\leq T/2$ for all $n$, it follows that
  \begin{align*}
    E_n \geq \left ( \int_{T/2}^T\int_{B_{R/2}(0)}   f^{p_n} \astar \;dvdt \right )^{\frac{1}{p_n}}.
  \end{align*}  
  At the same time, for each $n$ we have
  \begin{align*}
    E_n \geq \left ( \inf \limits_{B_R\times (0,T)} \astar \right )^{\frac{1}{p_n}} \left ( \int_{T/2}^T\int_{B_{R/2}(0)}  f^{p_n}\;dvdt \right )^{\frac{1}{p_n}}.
  \end{align*}  
  Since the infimum of $\astar$ over $B_R\times (0,T)$ is strictly positive for each $R$, it follows that
  \begin{align*}
    \limsup \limits_{n \to \infty} E_n \geq \| f\|_{L^\infty(B_{R/2}(0)\times (T/2,T))}.
  \end{align*}
  This, together with \eqref{eqn:Moser iteration Lpn estimate for all n} proves the Lemma.
  \end{proof}
   
\begin{proof}[Proof of Theorems \ref{thm:main_1} and \ref{thm:very soft potentials estimate}.]  
  It is elementary that if $p=1+\frac{2}{d}$ and $q=2(1+\frac{2}{d})$, then
  \begin{align*}
    \frac{q}{p(q-2)} = \frac{d}{2}.
  \end{align*}
  In this case, applying Lemma \ref{lem:weighted Moser estimate} with this selection of $p$ and $q$, and $R>0$, we have 
\begin{align*}
   \|f\|_{L^\infty(B_{R/2}(0)\times (T/2,T))} \leq C R^{-\gamma \tfrac{d}{2}}\left (\frac{1}{T}+1 \right )^{\frac{d}{2}} \left (\int_{T/4}^T \int_{B_R(0)}f^{p}\astar \;dvdt \right )^{\frac{d}{d+2}}.
\end{align*}
  Then, Proposition \ref{prop:LpLp via interpolation with weight} yields 
\begin{align*}
   \|f\|_{L^\infty({T}/{2},T; B(0, R/2))} &   \leq C(\fin,d,\gamma) R^{-\gamma \tfrac{d}{2}} \left (\frac{1}{T}+1 \right )^{\frac{d}{2}}.
\end{align*}

\end{proof}

\section{The Coulomb potential's case}\label{section:Coulomb regularization}

Throughout this section, we focus on the Coulomb case, as such we set $\gamma=-d$ and make no further references to the case $\gamma\neq-d$. At the end we will prove Theorem \ref{thm:Coulomb case good estimate}.  Since $\gamma=-d$, $h_{f,\gamma} = f$, and thus the $\varepsilon$-Poincar\'e inequality when applied to $\phi = f^p$ for some $p$, yields an inequality that may be used in place of the Sobolev embedding in an iteration procedure akin to the more standard one in Lemma \ref{lem:weighted Moser estimate}. 

\begin{rem} Related to the $\varepsilon$-Poincar\'e inequality there is a very interesting \emph{nonlinear} inequality proved by Gressman, Krieger, and Strain in  \cite{GreKriStr2012}: for each $p>0$ and any Lipschitz non-negative $f \in L^1(\mathbb{R}^d)$ we have the inequality 
\begin{align}\label{eqn:GressmannKriegerStrain}
  \int_{\mathbb{R}^d} f^{p+1}\;dv \leq \left (\frac{p+1}{p} \right )^2\int_{\mathbb{R}^d} (A_{f,\gamma}\nabla f^{\frac{p}{2}},\nabla f^{\frac{p}{2}})\;dv. 
\end{align}
As pointed out in \cite{GreKriStr2012}, the explicit constant in the inequality becomes sharp as $p\to 1^+$ (it corresponds to the $H$-theorem). The difference between the $\varepsilon$-Poincar\'e  and \eqref{eqn:GressmannKriegerStrain} is that the former involves an arbitrary Lipschitz, and integrable function $\phi$, whereas the latter one involves the function $\phi = f^{\frac{p}{2}}$. 
\end{rem}
In the next two lemmas we prove higher integrability for $f$ and for $a_{f,\gamma}$, always assuming the $\varepsilon$-Poincar\'e inequality.

\begin{prop}\label{prop:L2L2 bound for Coulomb case}
  Let $f(v,t)$ be a classical solution in $(0,T)$ for which the $\varepsilon$-Poincar\'e holds uniformly in time for some some $\varepsilon<4$ (this is guaranteed by the much stronger Assumption \ref{Assumption:Epsilon Poincare}). Then, the following estimate holds
  \begin{align*}
    \int_{T_1}^{T_2}\int_{\mathbb{R}^d}f^2\;dvdt \leq \frac{4}{4-\varepsilon}(C(\fin)+(T_2-T_1)\Lamp).
  \end{align*}
\end{prop}

\begin{proof}
  Applying the $\varepsilon$-Poincar\'e with $\phi=f^{1/2}$ at any time $t\in (T_1,T_2)$ leads to
  \begin{align*}
    \int_{\mathbb{R}^d} f^2\;dv \leq \varepsilon \int_{\mathbb{R}^d}(A\nabla f^{1/2},\nabla f^{1/2})\;dv+\Lamp\int_{\mathbb{R}^d}f\;dv. 
  \end{align*}
  On the other hand,
  \begin{align*}
    4\int_{T_1}^{T_2}\int_{\mathbb{R}^d}(A\nabla f^{1/2},\nabla f^{1/2})\;dvdt-\int_{T_1}^{T_2}\int_{\mathbb{R}^d}f^2\;dvdt \leq C(\fin). 
  \end{align*}
  Combining these two, always keeping in mind that $\|f(t)\|_{L^1}=1$ for all $t$, it follows that 
  \begin{align*}
    (4-\varepsilon)\int_{T_1}^{T_2}\int_{\mathbb{R}^d}(A\nabla f^{1/2},\nabla f^{1/2})\;dvdt \leq C(\fin)+(T_2-T_1)\Lamp.
  \end{align*}
  Then, we apply \eqref{eqn:GressmannKriegerStrain}, and conclude that 
  \begin{align*}
    \int_{T_1}^{T_2}\int_{\mathbb{R}^d}f^2\;dvdt \leq \frac{4}{4-\varepsilon}\left ( C(\fin)+(T_2-T_1)\Lamp\right ).
  \end{align*}
\end{proof}

Observe that in the last step of the previous proof, we could have just as well used \eqref{eqn:GressmannKriegerStrain}, however, the result is still conditional, as it relies on the $\varepsilon$-Poincar\'e to control the quadratic term in the first place. The next Lemma shows that, with a constant that grows with $p$, one can obtain $L^pL^p$ estimates for $f$ using the energy inequality (which assumes the $\varepsilon$-Poincar\'e inequality) and the inequality \eqref{eqn:GressmannKriegerStrain}.

\begin{lem}\label{lem:LpLp iterative bound for Coulomb}
  Let $f$ be as before, except we now assume the $\varepsilon$-Poincar\'e holds for any sufficiently small $\varepsilon$ (note this is still weaker than Assumption \ref{Assumption:Epsilon Poincare}). Let $n\geq 0$ be any positive integer and $p_0>1$. Then there is a a $C$ determined by $n$, $\Lambda_{f}$ (for the whole time interval), and $p_0$, such that 	
  \begin{align*}
     \sup \limits_{T/4\leq t\leq T}  \|f(t)\|_{L^{p_0+n}(\mathbb{R}^d)} \leq C\left (\frac{1}{T}+1 \right )^{\frac{n+1}{p_0+n}}\left (\int_{0}^T\int_{\mathbb{R}^d} f^{p_0}\;dvdt \right )^{\frac{1}{p_0+n}}.	  
   \end{align*}	  
\end{lem}

\begin{proof}
  Let us write, 
  \begin{align*}	
    p_n  := p_0+n,\;\;T_{n} = \frac{1}{4}\left ( 1- \frac{1}{2^n} \right )T.
  \end{align*}	
  We apply Corollary \ref{cor_iterat} with exponent $p_n$ and  $\eta=1$; or more precisely, we apply it to a sequence of $\eta$'s with compact support which converge locally uniformly to the constant $1$ and pass to the limit. It follows that 
\begin{align*}
    \frac{p_0-1}{2p_0}\int_{T_{n+1}}^T\int_{\mathbb{R}^d}a_*|\nabla f^{\frac{p_n}{2}}|^2\;dvdt \leq \left ( \frac{2^{n+3}}{T}+ p_n\Lambda_f(\tfrac{1}{2p_n}) \right ) \int_{T_n}^T\int_{\mathbb{R}^d}  f^{p_n}\;dvdt.	
  \end{align*}
  Applying \eqref{eqn:GressmannKriegerStrain}, it follows that
  \begin{align*}
    \int_{T_{n+1}}^T \int_{\mathbb{R}^d} f^{p_{n+1}}\;dvdt \leq C(p_0,\Lambda_f(1/2p_n),n) \left ( \frac{1}{T} +1 \right ) \int_{T_n}^T\int_{\mathbb{R}^d} f^{p_n}\;dvdt.
  \end{align*}
  Iterating this estimate, it follows that for some $C=C(p_0,\Lambda_f(\cdot),n)$
    \begin{align*}
    \int_{T_{n+1}}^T \int_{\mathbb{R}^d} f^{p_{n+1}}\;dvdt \leq C \left ( \frac{1}{T} +1 \right )^{n+1} \int_{0}^T\int_{\mathbb{R}^d} f^{p_0}\;dvdt.
  \end{align*}
  Combining this bound with the energy inequality for $f^{p_n+1}$, we have
  \begin{align*}
   \sup_{T_{n+2} \leq t\leq T} \left \{ \int_{\mathbb{R}^d} f^{p_{n+1}}(t)\;dv \right \} &\leq C\left ( \frac{1}{T} +1 \right )^{n+2} \int_{0}^T\int_{\mathbb{R}^d} f^{p_0}\;dvdt,
  \end{align*}
  for some even bigger $C=C(p_0,\Lambda_f(\cdot),n)$, which implies
  \begin{align*}
    \sup \limits_{T_{n+1}\leq t\leq T} \|f(t)\|_{L^{p_n}(\mathbb{R}^d)}   & \leq C\left (\frac{1}{T}+1 \right )^{\frac{n+1}{p_n}}\left (\int_{0}^T\int_{\mathbb{R}^d} f^{p_0}\;dvdt \right )^{\frac{1}{p_n}}.
  \end{align*}	
  Since $T_n\leq T/4$ for each $n\ge 0$ we conclude that 
    \begin{align*}
    \sup \limits_{T/4\leq t\leq T} \|f(t)\|_{L^{p_n}(\mathbb{R}^d)}   & \leq C\left (\frac{1}{T}+1 \right )^{\frac{n+1}{p_n}}\left (\int_{0}^T\int_{\mathbb{R}^d} f^{p_0}\;dvdt \right )^{\frac{1}{p_n}},
  \end{align*}	
  where $C$ is a function of $p_0$, $\Lambda_f(\cdot)$, and $n$.
\end{proof}

\begin{rem}
  The constant obtained in Lemma \ref{lem:LpLp iterative bound for Coulomb} goes to infinity as $n$ increases to $\infty$. This is different from the usual Moser type iteration, the way Lemma \ref{lem:LpLp iterative bound for Coulomb} will be used will not involve passing to the limit, instead we will take a large, but finite, $n$.
\end{rem}

A consequence of the previous lemma is highlighted in the following corollary.
\begin{cor}\label{a_bounded}
  There is a constant $C(\fin,d,C_P,\kappa_P,p)$ such that
  \begin{align*}
    \|a_{f,\gamma}\|_{L^\infty(T/4,T,L^\infty(\mathbb{R}^d))}\leq C\left (\frac{1}{T}+1 \right )^{1-\frac{2}{d}}.
  \end{align*}
  \end{cor}

\begin{proof}
  Since $a = c_d|v|^{2-d} * f$, interpolation yields for $p\in (d/2,\infty)$
  \begin{align}\label{eqn:interpolation for Newtonian potential}
    \|a(t)\|_{L^\infty(\mathbb{R}^d)} \leq C(d,p)\|f(t)\|_{L^1(\mathbb{R}^d)}^{\frac{p}{p-1}\left ( \frac{2}{d}-\frac{1}{p} \right )}\|f(t)\|_{L^p(\mathbb{R}^d)}^{\frac{p}{p-1}\left ( 1- \frac{2}{d} \right )}.	
  \end{align}	
  Indeed, for every $r>0$ and $v\in \mathbb{R}^d$ we have
  \begin{align*}	
    |a(v)|\leq C(d)\int_{B_r(v)}f(w)|v-w|^{2-d}\;dw +C(d)\int_{\mathbb{R}^d\setminus B_r(v)}f(w)|v-w|^{2-d}\;dw	.
  \end{align*}	
  H\"older's inequality yields,	
  \begin{align*}	
    |a(v)|\leq C(d) \|f\|_{L^p(\mathbb{R}^d)}\left (\int_{B_r(v)}|v-w|^{(2-d)p'}\;dw\right )^{\frac{1}{p'}}+C(d) R^{2-d}\|f\|_{L^1(\mathbb{R}^d)}	.
  \end{align*}	
  Note that
  \begin{align*}
    \left (\int_{B_r(v)}|v-w|^{2-d}\;dw\right )^{\frac{1}{p'}} \leq C(d)^{\frac{1}{p'}}r^{\frac{d}{p'}+2-d} = C(d)^{1-\frac{1}{p}}r^{2-\frac{d}{p}}.
  \end{align*}	
  Therefore
  \begin{align*}	
    \|a\|_{L^\infty(\mathbb{R}^d)} \leq C(d) \|f\|_{L^p(\mathbb{R}^d)}r^{2-\frac{d}{p}}+C(d) R^{2-d}\|f\|_{L^1(\mathbb{R}^d)}	.
  \end{align*}
  Optimizing in the parameter $r>0$, we obtain \eqref{eqn:interpolation for Newtonian potential}. Therefore, for $p\in(d/2,\infty)$
  \begin{align}\label{a_opt}
    \|a_{f,\gamma}\|_{L^\infty(T/4,T,L^\infty(\mathbb{R}^d))}\le C(d,p,\fin)\|f\|_{L^\infty(T/4,T,L^p(\mathbb{R}^d))}^{\frac{p}{p-1}\left( 1-\frac{2}{d} \right )}.
  \end{align}
  Next, let us apply Proposition \ref{prop:L2L2 bound for Coulomb case} and Lemma \ref{lem:LpLp iterative bound for Coulomb} with $p_0=2$, yielding for $n\geq 1$ 
 \begin{align}\label{sup_f_p+n}
   \sup \limits_{T/4\leq t\leq T}  \|f(t)\|_{L^{p_0+n}(\mathbb{R}^d)} &\leq C\left (\frac{1}{T}+1 \right )^{\frac{n+1}{n+2}}\left (\int_{0}^T\int f^{2}\;dvdt \right )^{\frac{1}{n+2}} \nonumber \\
  & \leq C\left (\frac{1}{T}+1 \right )^{\frac{p-1}{p}},
  \end{align}	
   where $C = C(\fin,d,\gamma,p,C_P,\kappa_P)$, $p=2+n$, and $n$ is chosen so that $p>d/2$. For such $p$'s, \eqref{a_opt} and \eqref{sup_f_p+n} yield
  \begin{align*}
    \|a_{f,\gamma}\|_{L^\infty(T/4,T,L^\infty(\mathbb{R}^d))}\leq C(d,p,\fin,\varepsilon)\left ( 1 + \frac{1}{T} \right )^{\left( 1-\frac{2}{d} \right )}.
  \end{align*}
 
\end{proof}

\begin{proof}[Proof of Theorem \ref{thm:Coulomb case good estimate}:] Let $R>0$, noting that $\astar\leq a$, and using Corollary \ref{a_bounded}, we have for $p>1$ a constant $C=C(d,\gamma,\fin,p,C_P,\kappa_P)$ such that
\begin{align*}	
  \left ( \int_{T/4}^{T}\int f^{p}\astar\;dvdt\right )^{\frac{1}{p}} \leq C\left (1+\frac{1}{T} \right )^{\frac{1}{p}\left (1-\frac{2}{d}\right)} \left ( \int_{T/4}^{T}\int  f^{p}\;dvdt\right )^{\frac{1}{p}}.	
\end{align*}
This estimate, combined with Lemma \ref{lem:weighted Moser estimate} leads to the following bound for each $R>1$,
\begin{align*}
   \|f\|_{L^\infty({T}/{2},T; B(0,R))} & \leq CR^{-\frac{\gamma}{p}\frac{q}{q-2}}\left (\frac{1}{T}+1 \right )^{\frac{1}{p}\frac{q}{q-2}}\left (1+\frac{1}{T} \right )^{\frac{1}{p}\left ( 1-\frac{2}{d}\right ) }  \|f\|_{L^{p}(T/4,T; L^{p}(\mathbb{R}^d))},  
 \end{align*}
 where this time $C=C(d,\gamma,\fin,p,C_P,\kappa_P,C_D)$ and $q$ is as in Lemma \ref{lem:weighted Moser estimate} (it's exact value will be immaterial for what follows). On the other hand, Lemma \ref{lem:LpLp iterative bound for Coulomb} with $p_0=2$ and $p=2+n$ and Proposition \ref{prop:L2L2 bound for Coulomb case}, yield
  \begin{align*}
    \|f\|_{L^{p}(T/4,T; L^{p}(\mathbb{R}^d))}  \leq C(\fin,d,p,C_P,\kappa_P)\left (1+\frac{1}{T}\right)^{\frac{p-1}{p}}. 
  \end{align*}
  Therefore, for every $p$ of the form $p=2+n$, $n\in\mathbb{N}$, we have the estimate
   \begin{align*}
   \|f\|_{L^\infty({T}/{2},T; B(0, R))} & \leq  CR^{\tilde \delta(p)}\left (1 +\frac{1}{T} \right )^{1+\delta(p)}.
   \end{align*}
  Here, $\delta(p)$ and $\tilde \delta(p)$ denote the expressions
  \begin{align*}
    \delta(p) := \frac{1}{p}\left(\frac{q}{q-2} -\frac{2}{d}\right)  >0,\; \tilde \delta(p) = -\frac{\gamma}{p}\frac{q}{q-2} >0
  \end{align*}
 so evidently $\delta(p),\tilde \delta(p) \to 0$ as $p\to\infty$. Therefore, given any $s>0$ there exists a $p>0$ and $C$ such that $\delta(p),\tilde \delta(p)\leq s$ and thus, for some constant $C$ we have 
\begin{align*}
   \|f\|_{L^\infty({T}/{2},T; B(0, R))} & \leq C R^{s}\left (1 +\frac{1}{T} \right )^{1+s},\;\;\forall\;R>1.
\end{align*}

\end{proof}


\appendix

\section{Auxiliary computations}\label{sec:Auxiliary Computations}
	
\subsection{Some properties of the weight $|w|^{-m}$} 

\begin{lem}\label{lem:averages for powers of v}
 The following inequalities hold for any $m < d $, 
 \begin{align*}
  \fint_{B_r(v_0)} \frac{1}{|v|^{m}}\;dv \approx \left (\max\{ |v_0|,2r\}\right )^{-m}.
  \end{align*}
  The implied constants being determined by $m$ and $d$. The same result holds (with a different implied constant) if one uses cubes instead of balls.
\end{lem}

\begin{rem}\label{rem:cubes versus balls}
  In light of the inclusions
  \begin{align*}
    B_r(v_0) \subset Q_r(v_0) \subset B_{r\sqrt{d}}(v_0),
  \end{align*}	  
  it is clear that the inequality above holds for cubes if and only if it holds balls.
\end{rem}

\begin{proof}[Proof of Lemma \ref{lem:averages for powers of v}]
  In light of Remark \ref{rem:cubes versus balls} it suffices to prove the lemma for balls $B_r(v_0)$. In the integral under consideration, take the change of variables $v=r(w_0+w)$, where $w_0 := v_0/r$, then
  \begin{align*}
    \fint_{B_{r}(v_0)}|v|^{-m}\;dv = r^{-m}\fint_{B_1(0)}|w_0+w|^{-m}\;dw
  \end{align*}
  If $|v_0|\leq 2r$, that is if $|w_0|\leq 2$, then an elementary computation shows that
  \begin{align*}
    \fint_{B_1(0)}|w_0+w|^{-m}\;dw \approx 1.
  \end{align*}
  It follows that if $|v_0|\leq 2r$,
  \begin{align*}
    \fint_{B_{r}(v_0)}|v|^{-m}\;dv \approx r^{-m}.
  \end{align*}  
  If $|v_0|\geq 2r$, that is if $|w_0|\geq 2$, then $|w_0|/2\leq |w_0+w|\leq 2|w_0|$ for all $w\in B_1(0)$, so
  \begin{align*}
    \fint_{B_1(0)}|w_0+w|^{-m}\;dw \approx |w_0|^{-m} = r^{m}|v_0|^{-m}.
  \end{align*}
  It follows that when $|v_0|\geq 2r$,
  \begin{align*}
    \fint_{B_{r}(v_0)}|v|^m\;dv \approx |v_0|^{m}.
  \end{align*}	
\end{proof}

\begin{rem} \label{rem:averages for powers of v}
  If $m\in [0,d)$, then $|v_1|\leq |v_0|+r\lesssim \max\{|v_0|,2r\}$ for all $v_1\in B_r(v_0)$. Then, Lemma \ref{lem:averages for powers of v} implies that
  \begin{align*}
    \fint_{B_r(v_0)}\frac{1}{|v|^m}\;dv \lesssim \frac{1}{|v_1|^m},\;\;\forall\;v_1\in B_r(v_0).
  \end{align*}
  The implied constant depending only on $m$ and $d$. In particular, for every $m\in [0,d)$ the function $\frac{1}{|v|^{m}}$ is a $\mathcal{A}_1$-weight. 
\end{rem}

\begin{lem}\label{lem:averages for powers of bracket v}
  For every $m \in \mathbb{R}$, $v_0 \in \mathbb{R}^d$, and $r\in(0,1)$ we have the estimate
  \begin{align*}
    \fint_{B_r(v_0)}(1+|v|)^{-m}\;dv \approx (1+\max\{|v_0|,2r\})^{-m}.
  \end{align*}
 The implied constants depending only on $d$ and $m$. The same result holds (with a different implied constant) if one uses cubes instead of balls.
\end{lem}

\begin{proof}
  Again by Remark \ref{rem:cubes versus balls}, it is clear it suffices to prove the lemma for balls $B_r(v_0)$. Using the same change of variables as in the previous lemma, it follows that
  \begin{align*}
    \fint_{B_r(v_0)}(1+|v|)^{-m}\;dv = \fint_{B_1(0)}(1+r|w_0+w|)^{-m}\;dw,\;\;w_0 = v_0/r.
  \end{align*}  	
  If $|v_0|\geq 2r$, which is the same as $|w_0|\geq 2$, we have as in the proof of the previous lemma that
  $|w_0|/2\leq |w_0+w|\leq 2|w_0|$, thus
  \begin{align*}
   (1+r|w_0|)/2 \leq 1+r|w_0+w| \leq 2(1+r|w_0|),\;\;\forall\;w\in B_1(0).
  \end{align*}
  Since $r|w_0|=|v_0|$, it follows that for $|v_0|\geq 2r$,
  \begin{align*}
    \fint_{B_1(0)}(1+r|w_0+w|)^{-m}\;dw \approx (1+|v_0|)^{-m}.
  \end{align*}  	
  Next, given that $r\in(0,1)$, we have $1+2r \in (1,3)$. Therefore, for $|v_0|\leq 2r$ we have
  \begin{align*}
    (1+\max\{|v_0|,2r\})^{-m} = (1+2r)^{-m} \approx 1.	  
  \end{align*}	  
  At the same time, when $|v_0|\leq 2r$ we have $0\leq |v|\leq 3r$ for all $v\in B_r(v_0)$. Then, we can conclude that for $|v_0|\leq 2r$ and $r\in(0,1)$ we have
  \begin{align*}
    \fint_{B_r(v_0)}(1+|v|)^{-m}\;dv \approx 1 \approx (1+2r)^{-m} = (1+\max\{|v_0|,2r) )^{-m},	  
  \end{align*}	  
  with all the implied constants being determined by $d$ and $m$, and the Lemma is proved.
\end{proof}

In what follows, given $e\in \mathbb{S}^{d-1}$, we shall write $[e] := \{ r e \mid r\in\mathbb{R} \}$.

\begin{prop} \label{prop:kernel averages}
  Let $\gamma\geq -d$, $m\geq 0$ with $m|2+\gamma|<d$, $v_0 \in \mathbb{R}^d$, $e\in\mathbb{S}^{d}$, and $r>0$, then for any $v_1 \in B_r(v_0)$ we have that
  \begin{align*}
    \fint_{B_r(v_0)}|v|^{(2+\gamma)m}(\Pi(v)e,e)^m\;dv\approx \left \{ \begin{array}{rl} 
      |v_1|^{(2+\gamma)m}(\Pi(v_1)e,e)^m & \textnormal{ if } \textnormal{dist}(v_1,[e]) \geq 2r,\\
      \max\{ |v_1|,2r\}^{\gamma m}r^{2m} & \textnormal{ if } \textnormal{dist}(v_1,[e])\leq 2r.	  	
    \end{array}\right.
  \end{align*}	  
  The implied constants being determined by $d, m$, and $\gamma$. The same result holds (with different implied constants) if one uses cubes instead of balls.
\end{prop}

\begin{proof}
  As before, we shall write the proof for the case of balls, noting the same arguments yield the respective result for cubes. Let us write $p=-(2+\gamma)m$. The change of variables $v=v_0+rw$ yields
  \begin{align*}
    \fint_{B_r(v_0)}|v|^{-p}(\Pi(v)e,e)^m\;dv & = \fint_{B_1(0)}|v_0+rw|^{-p}(\Pi(v_0+rw)e,e)^m\;dw. 
  \end{align*}	  
  Then, writing $v_0 = rw_0$, we have
  \begin{align*}
    |v_0+rw|^{-p}(\Pi(v_0+rw)e,e)^m = r^{-p} |w_0+w|^{-p}(\Pi(w_0+w)e,e)^m.	  
  \end{align*}	  
  We consider two cases, first, if $|w_0|\leq 2$, then
  \begin{align*}
    \fint_{B_1(0)}|w_0+w|^{-p}(\Pi(w_0+w)e,e)^m\;dw \approx \fint_{B_1(0)}(\Pi(w_0+w)e,e)^m\;dw \approx 1.
  \end{align*}	  
  While, if $|w_0|\geq 2$, then
  \begin{align*}
    \frac{1}{2}|w_0| \leq |w_0+w| \leq \frac{3}{2}|w_0|,\;\;\forall\;w\in B_1(0).	  
  \end{align*}	  
  These inequalities together with $(\Pi(w_0+w)e,e)\geq 0$ lead to 
  \begin{align*}
    \fint_{B_1(0)}|w_0+w|^{-p}(\Pi(w_0+w)e,e)^m\;dw \approx \max\{|v_0|,2r\}^{-p}\fint_{B_1(0)}(\Pi(w_0+w)e,e)^m\;dw.
  \end{align*}	  
  Next, we make use of the fact that
  \begin{align*}
    |v|^2(\Pi(v)e,e) = \textnormal{dist}(v, [e] )^2.
  \end{align*}
  By the triangle inequality, if $\textnormal{dist}(w_0,[e]) \geq 2$, we have for every $w\in B_1(0)$,
  \begin{align*}
    \frac{1}{2}\textnormal{dist}(w_0, [e] )\leq \textnormal{dist}(w_0+w, [e] )\leq \frac{3}{2}\textnormal{dist}(w_0, [e] ).
  \end{align*}
  In this case we also have $|w_0|\geq 2$, so
  \begin{align*}
    (\Pi(w_0+w)e,e) \approx (\Pi(w_0)e,e),\;\;\forall\;w\in B_1(0).
  \end{align*}
  Therefore, in the case $\textnormal{dis}(v_0,[e])\geq 2r$ we have
  \begin{align*}
    \fint_{B_r(v_0)}|v|^{-p}(\Pi(v)e,e)^m\;dv & \approx |v_0|^{-p}(\Pi(w_0)e,e)^m.
  \end{align*}	
  It remains to consider the case $|w_0|\geq 2$ and $\textnormal{dist}(w_0,[e])\leq 2$. It is easy to see that
  \begin{align*}
    \fint_{B_1(0)}(\Pi(w_0+w)e,e)^m\;dw \approx \min\{1,|w_0|^{-1}\}^{2m} = |w_0|^{-2m},
  \end{align*}
  from where it follows that,
  \begin{align*}
    \fint_{B_r(v_0)}|v|^{-p}(\Pi(v)e,e)^m\;dv & \approx |v_0|^{-p} (r/|v_0|)^{2m}.
  \end{align*}  
  From the definition of $p$, $ |v_0|^{-p} \min\{1,r/|v_0|\}^{2m} = |v_0|^{(2+\gamma)m} (r^2|v_0|^{-2})^m$ thus
  \begin{align*}
    \fint_{B_r(v_0)}|v|^{-p}(\Pi(v)e,e)^m\;dv & \approx r^{2m}|v_0|^{\gamma m}.
  \end{align*}  
  This proves the proposition for $v_1 = v_0$. For $v_1 \in B_r(v_0)$, we make use of the inclusions $B_r(v_0)\subset B_{2r}(v_1)$ and $B_{r}(v_1) \subset B_{2r}(v_0)$ to obtain the estimates in the general case.
\end{proof}

The estimates on averages from Proposition \ref{prop:kernel averages} yield that the function given by $a_e(v)=(A_{f,\gamma}e,e)$ ($e\in\mathbb{S}^{d-1}$ fixed) is \emph{almost} in the $\mathcal{A}_1$ class.
\begin{prop}\label{prop:weight a_e is almost A1}
  Fix $f\geq 0$, $f\in L^1(\mathbb{R}^d)$, and $v_0 \in \mathbb{R}^d, e\in \mathbb{S}^{d-1}$, and $r>0$. Define
  \begin{align*}
    S(v,r,e) = \{ w \in\mathbb{R}^d \mid\; \textnormal{dist}(w-v,[e]) \leq 2r\},\;\;v \in \mathbb{R}^d.
  \end{align*}
  Then, for any $v_1 \in B_r(v_0)$ we have the inequality	
  \begin{align*}
    \fint_{B_r(v_0)}a_e(v)\;dv  & \leq Ca_e(v_1)+r^2\int_{S(v_1,r,e)}f(w)\max\{2r,|v_1-w|\}^\gamma\;dw,
  \end{align*}
  with $C = C(d,\gamma)$. The same result holds if one uses cubes $Q_r(v_0)$ in place of balls.
\end{prop}

\begin{proof}
  Throughout the proof let us write $a_e(v)=(A_{f,\gamma}e,e)$. Recall that we always have,
  \begin{align*}
    \fint_{B_r(v_0)}a_e(v)\;dv  & = C_{d,\gamma}\int_{\mathbb{R}^d}f(w)\fint_{B_r(v_0)}|v-w|^{2+\gamma}(\Pi(v-w)e,e)\;dvdw\\
	  & = C_{d,\gamma}\int_{\mathbb{R}^d}f(w)\fint_{B_r(v_0)}K_e(v-w)\;dvdw,
  \end{align*}
  where for convenience we are writing
  \begin{align*}
    K_e(v) = |v|^{2+\gamma}(\Pi(v)e,e).	  
  \end{align*}	
  Now, note that
  \begin{align*}
    \fint_{B_r(v_0)}K_e(v-w)\;dv = \fint_{B_r(v_0-w)}K_e(v)\;dv. 
  \end{align*}    
  Note that if $v_1 \in B_r(v_0)$ then $v_1-w\in B_r(v_0-w)$. Therefore, we may apply Proposition \ref{prop:kernel averages} in $B_r(v_0-w)$ with $m=1$ and with respect to the point $v_1-w\in B_r(v_0-w)$, which leads to a bound according to the location of $w$: If $w\not\in S(v_1,r,e)$, then 
  \begin{align*}
    \fint_{B_r(v_0)}K_e(v-w)\;dv \approx K_e(v_1-w),
  \end{align*}
  while for $w \in S(v_1,r,e)$,
  \begin{align*}
    \fint_{B_r(v_0)}K_e(v-w)\;dv \approx r^{2}\max \{2r,|v_1-w|\}^{\gamma}.
  \end{align*}
  Combining these two estimates, it follows that 
  \begin{align*}
    \fint_{B_r(v_0)}a_e(v)\;dv  & \approx \int_{S(v_1,r,e)^c}f(w)K_e(v_1-w)\;dw+r^2\int_{S(v_1,r,e)}f(w)\max \{2r,|v_0-w|\}^{\gamma}\;dw.
  \end{align*}  
  Since the first term is no larger than $Ca_e(v_1)$, for $C=C(d,\gamma)$, the proposition is proved.
\end{proof}

\bibliography{landaurefs}

\def\cprime{$'$}
\begin{thebibliography}{10}

\bibitem{AleLinLia2013}
R.~Alexandre, J.~Liao, and C.~Lin.
\newblock Some a priori estimates for the homogeneous {L}andau equation with
  soft potentials.
\newblock {\em KRM Volume 8, Issue 4, December 2015, 617 - 650}.

\bibitem{AMUXY08}
R.~Alexandre, Y.~Morimoto, S.~Ukai, C.-J. Xu, and T.~Yang.
\newblock Uncertainty principle and kinetic equations.
\newblock {\em Journal of Functional Analysis}, 225(8):2013--2066, 2008.

\bibitem{BarasGoldstein1984}
Pierre Baras and Jerome~A Goldstein.
\newblock The heat equation with a singular potential.
\newblock {\em Transactions of the American Mathematical Society},
  284(1):121--139, 1984.

\bibitem{CamSilSne2017}
S.~Cameron, L.~Silvestre, and S.~Snelson.
\newblock Global a priori estimates for the inhomogeneous {L}andau equation
  with moderately soft potentials.
\newblock {\em Ann. Inst. H. Poincare Anal. Non Lineaire}, (3):625–642, 2018.

\bibitem{CDH15}
K.~Carrapatoso, L.~Desvillettes, and L.~He.
\newblock Estimates for the large time behavior of the {L}andau equation in the
  {C}oulomb case.
\newblock {\em Arch. Ration. Mech. Anal.}, (2):381–420, 2017.

\bibitem{CTW15}
K.~Carrapatoso, I.~Tristani, and K.-C. Wu.
\newblock Cauchy problem and exponential stability for the inhomogeneous
  {L}andau equation.
\newblock {\em Arch. Ration. Mech. Anal. 223 (2017), no. 2, 1035-1037.}

\bibitem{ChanilloWheeden1985}
S.~Chanillo and R.~Wheeden.
\newblock L-p estimates for fractional integrals and {S}obolev inequalities
  with applications to {S}chr{\"o}dinger operators.
\newblock {\em Communications in partial differential equations},
  10(9):1077--1116, 1985.

\bibitem{ChanilloWheeden1985II}
S.~Chanillo and R.~Wheeden.
\newblock Weighted {P}oincare and {S}obolev inequalities and estimates for
  weighted {P}eano maximal functions.
\newblock {\em American Journal of Mathematics}, 107(5):1191--1226, 1985.

\bibitem{CDH09}
Y.~Chen, L.~Desvillettes, and L.~He.
\newblock Smoothing effects for classical solutions of the full {L}andau
  equation.
\newblock {\em Arch. Ration. Mech. Anal.}, 193(1):21--55, 2009.

\bibitem{Desvillettes14}
L.~Desvillettes.
\newblock {E}ntropy dissipation estimates for the {L}andau equation in the
  {C}oulomb case and applications.
\newblock {\em J. Funct. Anal. 269 (2015), no. 5, 1359-1403}.

\bibitem{DesVil2000a}
L.~Desvillettes and C.~Villani.
\newblock On the spatially homogeneous {L}andau equation for hard potentials.
  {I}. {E}xistence, uniqueness and smoothness.
\newblock {\em Comm. Partial Differential Equations}, 25(1-2):179--259, 2000.

\bibitem{DesVil2000b}
L.~Desvillettes and C.~Villani.
\newblock On the spatially homogeneous {L}andau equation for hard potentials.
  {II}. {$H$}-theorem and applications.
\newblock {\em Comm. Partial Differential Equations}, 25(1-2):261--298, 2000.

\bibitem{FKS82}
E.~Fabes, C.~Kenig, and R.~Serapioni.
\newblock The local regularity of solutions of degenerate elliptic equations.
\newblock {\em Comm. in P.D.E.}, (7):77--116, 1982.

\bibitem{Fefferman1983}
C.~Fefferman.
\newblock The uncertainty principle.
\newblock {\em American Mathematical Society}, 9(2):129--206, 1983.

\bibitem{FournierGuerin2009}
N.~Fournier and H.~Gu{\'e}rin.
\newblock Well-posedness of the spatially homogeneous {L}andau equation for
  soft potentials.
\newblock {\em Journal of Functional Analysis}, 256(8):2542--2560, 2009.

\bibitem{Fournier2010}
Nicolas Fournier.
\newblock Uniqueness of bounded solutions for the homogeneous {L}andau equation
  with a {C}oulomb potential.
\newblock {\em Communications in Mathematical Physics}, 299(3):765--782, 2010.

\bibitem{GarciaCuervaRubioDeFrancia1985}
Jos{\'e} Garc{\'\i}a-Cuerva and JL~Rubio De~Francia.
\newblock {\em Weighted norm inequalities and related topics}.
\newblock Elsevier Science Publishers, 1985.

\bibitem{GigaKohn85}
Y.~Giga and R.V. Kohn.
\newblock Asymptotically self-similar blow-up of semilinear heat equations.
\newblock {\em Communications on Pure and Applied Mathematics}, 38(3):297--319,
  1985.

\bibitem{GIMV16}
F.~Golse, C.~Imbert, C.~Mouhot, and A.~Vasseur.
\newblock {H}arnack inequality for kinetic {F}okker-{P}lanck equations with
  rough coefficients and application to the {L}andau equation.
\newblock {\em To appear in Annali della Scuola Normale Superiore di Pisa},
  2017.

\bibitem{GreKriStr2012}
P.~Gressman, J.~Krieger, and R.~Strain.
\newblock A non-local inequality and global existence.
\newblock {\em Advances in Mathematics}, 230(2):642--648, 2012.

\bibitem{GuGu15}
M.~Gualdani and N.~Guillen.
\newblock Estimates for radial solutions of the homogeneous {L}andau equation
  with {C}oulomb potential.
\newblock {\em Analysis and PDE}, 9(8):1772--1809, 2016.

\bibitem{Guo02}
Y.~Guo.
\newblock The {L}andau equation in a periodic box.
\newblock {\em Communications in mathematical physics}, 231(3):391--434, 2002.

\bibitem{GuWhe91}
C.~Gutierrez and R.~Wheeden.
\newblock {H}arnack's inequality for degenerate parabolic equations.
\newblock {\em Communications in PDEs}, 4,5(16):745--770, 1991.

\bibitem{HendersonSnelson2017}
Christopher Henderson and Stanley Snelson.
\newblock C\^{}$\backslash$infty smoothing for weak solutions of the
  inhomogeneous landau equation.
\newblock {\em arXiv preprint arXiv:1707.05710}, 2017.

\bibitem{HendersonSnelsonTarfulea2017}
Christopher Henderson, Stanley Snelson, and Andrei Tarfulea.
\newblock Local existence, lower mass bounds, and smoothing for the landau
  equation.
\newblock {\em arXiv preprint arXiv:1712.07111}, 2017.

\bibitem{Imb_Silv2017}
Cyril Imbert and Luis Silvestre.
\newblock Weak {H}arnack inequality for the {B}oltzmann equation without
  cut-off.
\newblock {\em arXiv preprint arXiv:1608.07571}, 2016.

\bibitem{KimGuoHwang2016}
Jinoh Kim, Yan Guo, and Hyung~Ju Hwang.
\newblock A ${L}^2$ to ${L}^\infty$ approach for the {L}andau equation.
\newblock {\em arXiv preprint arXiv:1610.05346}, 2016.

\bibitem{KriStr2012}
J.~Krieger and R.~Strain.
\newblock Global solutions to a non-local diffusion equation with quadratic
  non-linearity.
\newblock {\em Comm. Partial Differential Equations}, 37(4):647--689, 2012.

\bibitem{PascucciPolidoro2004}
Andrea Pascucci and Sergio Polidoro.
\newblock The moser's iterative method for a class of ultraparabolic equations.
\newblock {\em Communications in Contemporary Mathematics}, 6(03):395--417,
  2004.

\bibitem{SawWhe1992}
E.~Sawyer and R.~Wheeden.
\newblock Weighted inequalities for fractional integrals on {E}uclidean and
  homogeneous spaces.
\newblock {\em American Journal of Mathematics}, 114(4):813 -- 874, 1992.

\bibitem{Silvestre2015}
L.~Silvestre.
\newblock Upper bounds for parabolic equations and the {L}andau equation.
\newblock {\em J. Differential Equations, 262 (2017), no. 3, 3034 - 3055.}

\bibitem{Silvestre2014}
L.~Silvestre.
\newblock A new regularization mechanism for the {B}oltzmann equation without
  cut-off.
\newblock {\em Comm. Math. Phys. 348 (2016), no. 1, 69-100}, 2016.

\bibitem{Turesson}
B.O. Turesson.
\newblock Nonlinear potential theory and weighted {S}obolev spaces.
\newblock {\em Lecture Notes in Mathematics, 1736. Springer-Verlag, Berlin,
  2000}, 2000.

\bibitem{V98}
C.~Villani.
\newblock On a new class of weak solutions to the spatially homogeneous
  {B}oltzmann and {L}andau equations.
\newblock {\em Archive for rational mechanics and analysis}, 143(3):273--307,
  1998.

\bibitem{Wu13}
K-C. Wu.
\newblock Global in time estimates for the spatially homogeneous {L}andau
  equation with soft potentials.
\newblock {\em J. Funct. Anal. 266, no. 5, 3134-3155}, 2014.

\end{thebibliography}
\bibliographystyle{plain}

\end{document}